\documentclass[a4paper,12pt]{amsart}
\usepackage[english]{babel}
\usepackage{amsmath,amsthm,amssymb,amsfonts}
\usepackage{multicol}
\usepackage{todonotes}
 \usepackage{enumitem} 
\usepackage[bookmarks=true,hyperindex,pdftex,colorlinks,citecolor=blue]{hyperref}
\hypersetup{colorlinks=true,  linkcolor=blue, citecolor=red, urlcolor=cyan}

\usepackage{xcolor}
\usepackage[export]{adjustbox}
\usepackage{graphicx}
\usepackage{subcaption}

\usepackage[english]{babel}

\usepackage{graphicx}

\usepackage{tikz}
\usetikzlibrary{mindmap,backgrounds, calc}
\usetikzlibrary{matrix,chains,scopes,positioning,arrows,fit}
\usetikzlibrary{positioning, shapes, patterns}
\usetikzlibrary{automata} 
\usetikzlibrary{shapes.geometric, arrows, arrows.meta}
\usetikzlibrary{positioning,decorations.markings}

\theoremstyle{definition}
\newcounter{maincoro}

\newtheorem{maincor}[maincoro]{Corollary}

\newtheorem{theorem}{Theorem}[section]
\newtheorem{lemma}[theorem]{Lemma}
\newtheorem{fact}[theorem]{Fact}

\newtheorem{proposition}[theorem]{Proposition}
\newtheorem{corollary}[theorem]{Corollary}
\newtheorem{question}[theorem]{Question}
\theoremstyle{definition}
\newcounter{maintheorem}

\newtheorem{mainth}[maintheorem]{Theorem}
\newtheorem{definition}[theorem]{Definition}
\newtheorem{example}[theorem]{Example}

\theoremstyle{remark}

\numberwithin{equation}{section}

\newcommand{\R}{\mathbb{R}}
\newcommand{\C}{\mathbb{C}}
\newcommand{\N}{\mathbb{N}}

 \makeatletter
\renewcommand{\tocsection}[3]{%
	\indentlabel{\@ifnotempty{#2}{\bfseries\ignorespaces#1 #2\quad}}\bfseries#3}
\renewcommand{\tocsubsection}[3]{%
	\indentlabel{\@ifnotempty{#2}{\ignorespaces#1 #2\quad}}#3}

\AtBeginDocument{%
	\expandafter\renewcommand\csname r@tocindent0\endcsname{0pt}
}
\def\l@subsection{\@tocline{2}{0pt}{2.5pc}{5pc}{}}
\makeatother

\newcommand{\nn}[1]{{\left\vert\kern-0.25ex\left\vert\kern-0.25ex\left\vert #1 
		\right\vert\kern-0.25ex\right\vert\kern-0.25ex\right\vert}}

\renewcommand{\geq}{\geqslant}
\renewcommand{\leq}{\leqslant}

\newcommand{\e}{\varepsilon}
\newcommand{\ep}{\varepsilon}

\thanks{}

\subjclass[2020]{}

\date{\today}


\begin{document}

\title[$\ell_1$-like Transportation cost spaces]{Transportation cost spaces and stochastic trees}

\author[R. Medina]{Rubén Medina}
\address[R. Medina]{Universidad de Granada, Facultad de Ciencias. Departamento de Análisis Matemático, 18071-Granada (Spain). \newline
\href{https://orcid.org/0000-0002-4925-0057}{ORCID: \texttt{0000-0002-4925-0057}}}
\email{rubenmedina@ugr.es}

\author[G. Tresch]{Garrett Tresch}
\address[G. Tresch]{Texas A\&M University Department of Mathematics, 155 Ireland St, College Station, TX 77840 \newline
\href{https://orcid.org/0009-0000-0027-6240}{ORCID: \texttt{0009-0000-0027-6240}}}
\email{treschgd@tamu.edu}

\thanks{}

\date{\today}

\begin{abstract}
We study transportation cost spaces over finite metric spaces, also known as Lipschitz free spaces. Our work is motivated by a core problem posed by S. Dilworth, D. Kutzarova and M. Ostrovskii, namely, find a condition on a metric space $M$ equivalent to the Banach-Mazur distance between the transportation cost space over $M$ and $\ell_1^N$ of the corresponding dimension, which we call the $\ell_1^N$-distortion of $M$. In this regard, some examples have been studied like the $N\times N$ grid by Naor and Schechtman (2007) and the Laakso and diamond graphs by Dilworth, Kutzarova and Ostrovskii (2020), later studied by Baudier, 
Gartland and Schlumprecht (2023). We present here three main results. Firstly, we give a partial solution to this problem relating to the tree-like structure of the metric space. For that purpose, we develop a new technique that could potentially lead to a complete solution of the problem and utilize it to find an asymptotically tight upper bound of the $\ell_1^N$-distortion of the Laakso graphs, fully solving an open problem raised by Dilworth, Kutzarova and Ostrovskii. Finally, we apply our technique to prove that finite hyperbolic approximations of doubling metric spaces have uniformly bounded $\ell_1^N$-distortion.
\end{abstract}

\maketitle
\tableofcontents

\section{Introduction}

\subsection{Motivation and background}

Given a finite metric space $M$ we are going to work with the transportation cost space $\mathcal{F}(M)$ which is the normed linear space of \textit{transportation problems} of goods in the space $M$. Roughly speaking, a transportation problem can be understood as a pair of supply and demand of a certain amount of goods in $M$. The norm of a transportation problem is the cost of the optimal way to transport the goods offered to the places where it is demanded.

Transportation cost spaces have proven to be useful in Computer Science \cite{IM04}, Metric Geometry \cite{ANN18,NY17}, Functional Analysis \cite{GK03,Kal12,AGP22} and Optimal Transportation \cite{Vil03,Vil09}. As shown, these spaces have been studied in a wide variety of situations by many different researchers and this is reflected in the amount of names by which they are known (Transportation cost spaces \cite{KN06,OO20, DKO21,Sch23}, Arens-Eells spaces \cite{Kal08,Wea18}, Kantorovich-Rubinstein spaces \cite{Vil03,Vil09}, Wasserstein metrics \cite{BGS23,ANN18,NY17} or Lipschitz free spaces \cite{GK03,DKO20}). The last name (Lipschitz free space) is generally used in Banach space theory where infinite metric spaces are also considered and refers to the Kantorovich duality between $\mathcal{F}(M)$ and the space of Lipschitz functions over $M$ and the fact that it is a free object between the metric category and the Banach category. Researchers working on Computer Science also know this norm as the earth-mover distance \cite{KN06,NS07,ADI09}.

Transportation cost spaces are related to travelling salesman problems and, thus, having a deep understanding of the space is not an easy task. An important line of research in this setting is the relation between Transportation cost spaces and $\ell_1$ (see \cite{Gar24,BGS23,Sch23,FG23,BMSZ22,DKO21,OO20,DKO20, Wea18,Goda10,NS07,KN06,Kal04} for previous related works). Moreover, in \cite{DKO20}, S. Dilworth, D. Kutzarova and M. Ostrovskii explicitly ask the following question which is a setting problem in the area that we address in this paper.

\begin{question}{\cite[Problem 2.6]{DKO20}}\label{MAINQ}
It would be very interesting to find a condition on a finite metric space $M$ which is equivalent to the condition that the space $\mathcal{F}(M)$ is Banach-Mazur close to $\ell_1^N$ of the corresponding dimension. It is not clear whether it is feasible to find
such a condition.
\end{question}

Let us give a brief heuristic description of what it means for $\mathcal{F}(M)$ to be at distance $D\geq1$ to $\ell_1^N$. Given two transportation problems $x,y\in\mathcal{F}(M)$, one may consider the transportation problem $x+y\in\mathcal{F}(M)$ whose supply is the sum of the supplies of $x$ and $y$ and whose demand is also the sum of demands. Heuristically speaking, the equality $\|x+y\|=\|x\|+\|y\|$ means that the optimal transport for $x+y$ is obtained by performing the optimal transports for the problems $x$ and $y$ individually. This is clearly not always the case as shown in Figure \ref{fig:optrans}. Then, $\mathcal{F}(M)$ is at distance $D$ to $\ell_1^N$ if and only if there is a basis $(b_1,\dots,b_N)$ of $\mathcal{F}(M)$ where the $b_n$'s are basic transportation problems in $M$ satisfying that for every $x=\sum_{n}x_nb_n\in\mathcal{F}(M)$,
$$\|x\|\leq\sum_{n}|x_n|\|b_n\|\leq D\|x\|.$$
That is, up to a factor $D\geq1$, every transportation problem $x$ can be decomposed as a sum of basic transportation problems (the $b_n$'s) where the optimal cost for $x$ is attained by performing the optimal transports of these basic transportation problems individually.

The latter constant $D$ is called the $\ell_1^N$-distortion of the basis $(b_n)$. Obtaining the basis $(b_n)$ with the least possible $\ell_1^N$-distortion may be very interesting since it provides a simple description of all possible optimal transports in $M$ with an error that depends on that distortion.  In particular, with our approach we are able to compute an approximation of the optimal cost of any transportation problem with quadratic algorithmic complexity (see Proposition \ref{propbasisvect} and Theorem \ref{complex} where we show that the problem is solved by a triangular system of equations).

\subsection{Layout of results}

In this paper we are going to give a partial solution to Question \ref{MAINQ} (see Theorem \ref{mainthth} in section \ref{sec:prel}). More specifically, in Section \ref{sec:prel} we provide a metric characterization of the minimal $\ell_1^N$-distortion of a very natural family of bases in $\mathcal{F}(M)$, called stochastic bases. Our argument is new and could lead to a complete solution of Question \ref{MAINQ}. We also give precise definitions of the basic concepts of the paper in  Section \ref{sec:prel} whereas Section \ref{sec:mainres} is devoted to the proof of Theorem \ref{mainthth}.

In \cite{DKO21}, Dilworth, Kutzarova and Ostrovskii investigate the $\ell_1^N$-distortion of the $n^{th}$-Laakso graph $\mathcal{L}_n$. They find a lower bound of order $O(n)$ but they claim to `have not succeeded in proving a comparable (e.g. $O(n^a)$) upper bound'. In section \ref{sec:appl}, we give a complete solution to that problem and find an upper bound of order $O(n)$ using the technique established for the proof of Theorem \ref{mainthth}.

Finally, in Section \ref{sec:appl} we also find stochastic bases for finite hyperbolic approximations of metric spaces with uniformly bounded $\ell_1^N$-distortion (uniform in number of points of the approximation, see Theorem \ref{theo:hypdoub}). This distortion does not depend on the number of points of the hyperbolic approximation but is related to the doubling constant of the original metric space. It is worth mentioning that in \cite{Gar24}, a similar result is proven for infinite hyperbolic approximations of metric spaces but the technique utilized is not valid in the finite setting as it relies on the Pe\l cz\'ynski Decomposition Method. Hyperbolic approximations are well known and important objects for researchers in hyperbolic geometry (see, for instance \cite[Chapter 6]{BS07}).

\begin{figure}

\tikzset{every picture/.style={line width=0.75pt}} 

\begin{tikzpicture}[x=0.75pt,y=0.75pt,yscale=-1,xscale=1]

\draw  [dash pattern={on 4.5pt off 4.5pt}]  (340,60) -- (340,210) ;
\draw   (260,105) .. controls (260,102.24) and (262.24,100) .. (265,100) .. controls (267.76,100) and (270,102.24) .. (270,105) .. controls (270,107.76) and (267.76,110) .. (265,110) .. controls (262.24,110) and (260,107.76) .. (260,105) -- cycle ;
\draw  [fill={rgb, 255:red, 0; green, 0; blue, 0 }  ,fill opacity=1 ] (100,105) .. controls (100,102.24) and (102.24,100) .. (105,100) .. controls (107.76,100) and (110,102.24) .. (110,105) .. controls (110,107.76) and (107.76,110) .. (105,110) .. controls (102.24,110) and (100,107.76) .. (100,105) -- cycle ;
\draw  [fill={rgb, 255:red, 0; green, 0; blue, 0 }  ,fill opacity=1 ] (260,165) .. controls (260,162.24) and (262.24,160) .. (265,160) .. controls (267.76,160) and (270,162.24) .. (270,165) .. controls (270,167.76) and (267.76,170) .. (265,170) .. controls (262.24,170) and (260,167.76) .. (260,165) -- cycle ;
\draw   (100,165) .. controls (100,162.24) and (102.24,160) .. (105,160) .. controls (107.76,160) and (110,162.24) .. (110,165) .. controls (110,167.76) and (107.76,170) .. (105,170) .. controls (102.24,170) and (100,167.76) .. (100,165) -- cycle ;
\draw   (570,105) .. controls (570,102.24) and (572.24,100) .. (575,100) .. controls (577.76,100) and (580,102.24) .. (580,105) .. controls (580,107.76) and (577.76,110) .. (575,110) .. controls (572.24,110) and (570,107.76) .. (570,105) -- cycle ;
\draw  [fill={rgb, 255:red, 0; green, 0; blue, 0 }  ,fill opacity=1 ] (410,105) .. controls (410,102.24) and (412.24,100) .. (415,100) .. controls (417.76,100) and (420,102.24) .. (420,105) .. controls (420,107.76) and (417.76,110) .. (415,110) .. controls (412.24,110) and (410,107.76) .. (410,105) -- cycle ;
\draw  [fill={rgb, 255:red, 0; green, 0; blue, 0 }  ,fill opacity=1 ] (570,165) .. controls (570,162.24) and (572.24,160) .. (575,160) .. controls (577.76,160) and (580,162.24) .. (580,165) .. controls (580,167.76) and (577.76,170) .. (575,170) .. controls (572.24,170) and (570,167.76) .. (570,165) -- cycle ;
\draw   (410,165) .. controls (410,162.24) and (412.24,160) .. (415,160) .. controls (417.76,160) and (420,162.24) .. (420,165) .. controls (420,167.76) and (417.76,170) .. (415,170) .. controls (412.24,170) and (410,167.76) .. (410,165) -- cycle ;
\draw    (120.5,105.5) -- (248.5,105.5) ;
\draw [shift={(250.5,105.5)}, rotate = 180] [color={rgb, 255:red, 0; green, 0; blue, 0 }  ][line width=0.75]    (10.93,-3.29) .. controls (6.95,-1.4) and (3.31,-0.3) .. (0,0) .. controls (3.31,0.3) and (6.95,1.4) .. (10.93,3.29)   ;
\draw    (121.5,164.5) -- (249.5,164.5) ;
\draw [shift={(119.5,164.5)}, rotate = 0] [color={rgb, 255:red, 0; green, 0; blue, 0 }  ][line width=0.75]    (10.93,-3.29) .. controls (6.95,-1.4) and (3.31,-0.3) .. (0,0) .. controls (3.31,0.3) and (6.95,1.4) .. (10.93,3.29)   ;
\draw    (415,120.5) -- (415,148.5) ;
\draw [shift={(415,150.5)}, rotate = 270] [color={rgb, 255:red, 0; green, 0; blue, 0 }  ][line width=0.75]    (10.93,-3.29) .. controls (6.95,-1.4) and (3.31,-0.3) .. (0,0) .. controls (3.31,0.3) and (6.95,1.4) .. (10.93,3.29)   ;
\draw    (575,122.5) -- (575,150.5) ;
\draw [shift={(575,120.5)}, rotate = 90] [color={rgb, 255:red, 0; green, 0; blue, 0 }  ][line width=0.75]    (10.93,-3.29) .. controls (6.95,-1.4) and (3.31,-0.3) .. (0,0) .. controls (3.31,0.3) and (6.95,1.4) .. (10.93,3.29)   ;

\draw (178.5,81.9) node [anchor=north west][inner sep=0.75pt]    {$x$};
\draw (179.5,138.9) node [anchor=north west][inner sep=0.75pt]    {$y$};
\draw (478.5,121.4) node [anchor=north west][inner sep=0.75pt]    {$x+y$};

\end{tikzpicture}
    \caption{On the left side, two optimal transports are shown for the transportation problems $x$ and $y$ separately (supply in black dots and demand in white dots). On the right side, the optimal transport of $x+y$ is shown.}
    \label{fig:optrans}
\end{figure}
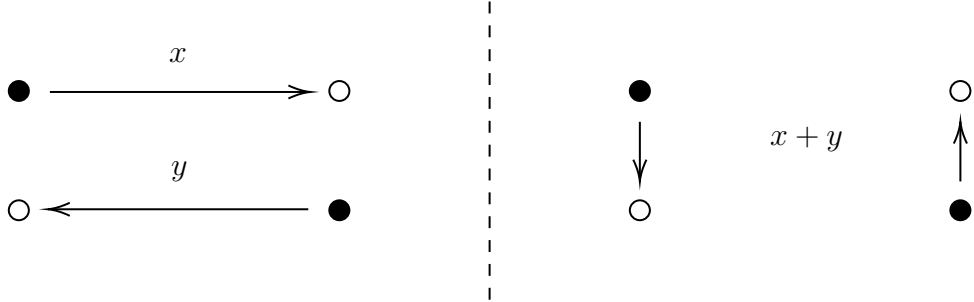

\subsection{Stochastic trees}

In particular, geodesic trees seem to be in the core of the relation between $\mathcal{F}(M)$ and $\ell_1$. In fact, it was shown by Godard in \cite[Theorem 4.2]{Goda10} that $\mathcal{F}(M)$ is isometrically $\ell_1^N$ if and only if $M$ is a geodesic
tree. The more general concept of a stochastic tree has been around in Computer Science since 1996 \cite{Bar96, FRT04,ST24} and it appeared in \cite{DKO21,BMSZ22,MV23,Sch23} in connection to our line of research. We will say that a finite metric space $(M,d)$ is $D$-isomorphic to a stochastic tree if there is a family of geodesic trees $\mathcal{T}$ whose vertices are the points of $M$ satisfying $d_T(x,y)\geq d(x,y)$ for every $x,y\in M$, $(T,d_T)\in \mathcal{T}$, and a probability distribution $p$ over $\mathcal{T}$ satisfing,
$$\mathbb{E}_p\bigg(\frac{d_T(x,y)}{d(x,y)}\bigg):=\sum_{T\in\mathcal{T}}p(T)\frac{d_T(x,y)}{d(x,y)}\leq D.
$$
The best possible $D$ for $M$ will be called the stochastic tree distortion of $M$. In \cite[Corollary 4.10]{Sch23} (see also \cite[Corollary 3.2]{Gar24}) it is proven that, if $M$ is $D$-isomorphic to a stochastic tree, then $\mathcal{F}(M)$ is $D$-isomorphic to a $O(D)$-complemented subspace of $\ell_1$. It is therefore natural to ask the following question.

\begin{question}[F. Baudier and T. Schlumprecht]\label{MAINSTQ}
    Is there a universal function $f:[1,\infty)\to[1,\infty)$ such that if $\mathcal{F}(M)$ is $D$-isomorphic to $\ell_1^N$ for $D\geq1$, then $M$ is $f(D)$-isomorphic to a stochastic tree?
\end{question}

Question \ref{MAINSTQ} may be understood as an analogue in the transportation cost setting of an older problem solved by J. Bourgain in \cite{Bou81}.

It is worth mentioning that in \cite{KL06} the authors study the stochastic tree distortion of subsets of hyperbolic spaces with finite Nagata dimensional boundary. However, this is not enough to compute the $\ell_1^N$-distortion of these spaces.

\section{Preliminaries and a metric characterization of stochastic $\ell_1^N$-distortion}\label{sec:prel}

We agree that every empty sum equals 0 and every empty product is equal to 1 throughout the whole text. We use the standard terminology of Banach space theory \cite{BL2000,Fab1}, graph theory \cite{Die18}, and the theory of metric embeddings \cite{Ost13}.

Now, we will introduce the main object of our discussion as well as some necessary concepts in order to precisely state the main result of the paper (see Theorem \ref{mainthth}).

\subsection{Transportation cost spaces as Lipschitz free objects}

Let $(M,d)$ be a finite metric space with $N+1$ and let $\delta_x:M\to\R$ be the indicator function/measure for every $x\in M$, that is,
$$\delta_x(y)=\begin{cases}
    1\;\;\;&\text{ if }x=y,\\
    0&\text{ otherwise.}
\end{cases}$$
A \textit{transportation problem} in $M$ is a nonzero, finitely supported function/measure $\mu:M\to\R$ such that $\sum_{x\in M}\mu(x)=0$. Therefore, we may decompose $\mu$ into its positive and negative parts
$$\mu=\underbrace{\sum_{x\in P^+}|\mu(x)|\delta_x}_{\mu_+}-\underbrace{\sum_{x\in P^-}|\mu(x)|\delta_x}_{\mu_-},$$
where $P^+=\{x\in M\;:\;\mu(x)>0\}$, $P^-=\{x\in M\;:\;\mu(x)<0\}$ and $\sum_{x\in P^+}|\mu(x)|=\sum_{x\in P^-}|\mu(x)|=C\geq0$. The pair $(\mu^+,\mu^-)$ is the pair of supply and demand respectively and $C$ is the amount of goods offered and demanded. Then, $\mu$ may be understood as a transportation problem where $C$-many goods must be transported from the supply in $P^+$ (distributed according to $\mu_+$) to the demand in $P^-$ (distributed according to $\mu_-$).

A \textit{transportation plan that solves $\mu$} is an element of the set $PL(\mu)$ defined by

$$PL(\mu):=\left\{(a_i,x_i,y_i)_{i=1}^n\subset [0,\infty)\times P^+\times P^-\ :\ n\in \N,\ \mu=\sum_{i=1}^n a_i(\delta_{x_i}-\delta_{y_i})\right\}$$

and can be interpreted as a strategy to transfer all $C$ goods from $P^+$ to $P^-$ by transporting $a_i$ goods from $x_i$ to $y_i$ for each $i$. The \textit{optimal cost} of a transportation problem $\mu$, denoted by $OC(\mu)$ is the minimum cost of a transportation plan that solves the problem $\mu$ or more precisely

$$OC(\mu):=\inf\left\{\sum_{i=1}^n a_i d(x_i,y_i)\ :\  (a_i,x_i,y_i)_{i=1}^n\in PL(\mu)\right\}$$

(see \cite[Section 14.1]{SMO20}). In this regard, Kantorovich-Gavurin duality theorem \cite{KG49} (streamlined in \cite[Theorem 2.7]{Sch23} and \cite[Theorem 14.2]{SMO20}) tells us that this optimal cost can be computed using the space $Lip_0(M)$ of Lipschitz functions from $M$ to $\R$ that vanish at a fixed point denoted by $0\in M$. Indeed, if one denotes $\|f\|_{Lip}$ the least Lipschitz constant of $f\in Lip_0(M)$, then
\begin{equation}\label{eqdefTC}OC(\mu)=\sup_{\substack{f\in Lip_0(M)\\\|f\|_{Lip}\leq1}}\langle\mu,f\rangle=\sup_{\substack{f\in Lip_0(M)\\\|f\|_{Lip}\leq1}}\sum_{x\in M}\mu(x)f(x).\end{equation}
As a consequence of equation \eqref{eqdefTC}, given a finite metric space $M$ with a distance $d$ and a distinguished point $0\in M$, we define the \textit{Transportation cost space} $(\mathcal{F}(M),\|\cdot\|)$ as the dual (as a normed space) of the space $(Lip_0(M),\|\cdot\|_{Lip})$. An element $\mu$ in the space $\mathcal{F}(M)$ will be represented as a function/measure $\mu:M\to\R$ and its norm $\|\mu\|$ will be given by the duality with $(Lip_0(M),\|\cdot\|_{Lip})$, that is,
\begin{equation}\label{eqdefTC2}\|\mu\|=\sup_{\substack{f\in Lip_0(M)\\\|f\|_{Lip}\leq1}}\sum_{x\in M}\mu(x)f(x).\end{equation}
Since $f(0)=0$ for every $f\in Lip_0(M)$, a vector $\mu\in \mathcal{F}(M)$ is independent of the value $\mu(0)$. Therefore, if we consider $\mu(0)\in\R$ such that $\sum_{x\in M}\mu(x)=0$ then the vector $\mu$ may be understood as a transportation problem where, by \eqref{eqdefTC} and \eqref{eqdefTC2}, its norm is the optimal cost of that transportation problem. This way, $\mathcal{F}(M)$ may be seen as the space of all possible transportation problems in $M$. In the case that $\mu=\delta_x-\delta_y$ for $x,y\in M$ distinct, we call $\mu$ a \textit{molecule}.

Let us describe some interesting properties of the Transportation cost space. It is worth mentioning that the definition of $\mathcal{F}(M)$ is independent of the element that is distinguished as $0\in M$ since all these constructions are isometric. Also, the map $\delta:M\to\mathcal{F}(M)$ given by $\delta(x)=\delta_x$ is an isometry and, hence, allows us to see $M$ as a subset of $\mathcal{F}(M)$. Finally, given any Lipschitz map $f$ from $M$ to a Banach space $X$, there is an extension $\widehat{f}:\mathcal{F}(M)\to X$ which is linear and preserves its Lipschitz constant (see \cite[Section 2]{GK03}).

As a consequence, we may find projections in $\mathcal{F}(M)$ as linear extensions of retractions, that is, given a Lipschitz retraction $R:M\to N\subset M$ we may see this retraction as a map $R:M\to\mathcal{F}(N)$ since $\delta:N\to\mathcal{F}(N)$ is an isometric embedding. It is immediate to see that $\widehat{R}$ is a linear projection from $\mathcal{F}(M)$ to $\mathcal{F}(N)$. These are not all possible projections onto $\mathcal{F}(N)$. In fact, the more general \textit{stochastic retractions}, that were implicitly used by J. Lee and A. Naor in \cite{LN05}, and later more explicitly in \cite{AP20,BDS21}, in order to define extension operators, also produce linear projections onto $\mathcal{F}(N)$.

A stochastic retraction onto $N\subset M$ is a map $R$ from $M$ to the space of probabilities over $N$ denoted by $\mathcal{P}(N)$ such that $R(x)=\delta_x$ for every $x\in N$. Since $\mathcal{P}(N)$ is a subset of $\mathcal{F}(N)$ we may see $R$ valued over $\mathcal{F}(N)$ and extend it to a linear projection $\widehat{R}:\mathcal{F}(M)\to\mathcal{F}(N)$. We will refer to both $R$ and $\widehat{R}$ as stochastic retractions. Stochastic retractions have proven to be very useful and natural in Metric Geometry and Functional Analysis (see, for instance  \cite{LN05,AP20,BDS21,Nao21,HQ22,Med23}). In fact, at the moment of writing this work it is unknown whether the best projections from $\mathcal{F}(M)$ onto $\mathcal{F}(N)$ must always be stochastic retractions \cite[Problem 2]{AP20}.

\subsection{Stochastic $\ell_1^N$-distortion}

Let us now discuss the basic notions of the theory of normed spaces utilized throughout the remainder of the text that will allow us to give the definition of stochastic $\ell_1^N$-distortion of a metric space. Given an $N$-dimensional normed space $X$ and an algebraic basis $B:=(b_1,\dots,b_N)$ of $X$, there is a unique basis $(b^*_1,\dots b^*_N)$ of $X^*$ such that, for every $x\in X$, we have $x=\sum_nb^*_n(x)b_n$. We define for $n=1,\dots, N$, the canonical projections of the basis $B$, denoted by $P_n:X\to X$, as
$$P_n(x)=\sum_{i=1}^nb^*_i(x)b_i.$$
We may now introduce the main concept of this work: the stochastic bases of $\mathcal{F}(M)$. A basis $(b_n)_{n=1}^N$ of $\mathcal{F}(M)$ will be called \textit{stochastic} whenever its canonical projections are stochastic retractions (see Proposition \ref{propbasisvect} for a better understanding of its basic elements as transportation problems). Stochastic bases have proven to be a very natural family of bases of $\mathcal{F}(M)$. Indeed, in the infinite dimensional setting, many of the Schauder bases found by now are stochastic \cite{HP14,HN17,Nov20,HM22aa,HM23aa}.

Given $(b_n)_{n=1}^N$, a basis of $X$, we want to determine how far this basis is from being an $\ell_1$-basis. More precisely, we want to compute the $\ell_1^N$\textit{-distortion of the basis} $(b_n)$, denoted from now on as $d_1(b_n)$, which is defined as
$$d_1(b_n)=\max_{x\in S_X}\sum_{n=1}^N|b^*_n(x)|\,\|b_n\|.$$

If $N$ is the dimension of $\mathcal{F}(M)$, we call the Banach-Mazur distance between $\mathcal{F}(M)$ and $\ell_1^N$ the $\ell_1^N$\textit{-distortion of} $M$ and denote it by $d_1(M)$. It is clear that $d_1(M)$ is the minimum of $d_1(b_n)$ over all possible bases $(b_n)$ of $\mathcal{F}(M)$, that is,
$$d_1(M)=\min_{(b_n)\text{ basis of }\mathcal{F}(M)}d_1(b_n).$$

Our aim in this paper is to characterize the $\ell_1^N$-distortion of the stochastic bases in $\mathcal{F}(M)$. That is, we are interested in what we call the \textit{stochastic} $\ell_1^N$\textit{-distortion of }$M$, given by
$$sd_1(M)=\min_{(b_n)\text{ stochastic basis of }\mathcal{F}(M)}d_1(b_n).$$

Since this work is developed under a graph theoretical framework, we need to introduce some basic concepts as well as new definitions for the statement of our main result, Theorem \ref{mainthth}.

\subsection{Trees, graphs and probabilities}

A graph $G$ is defined as a pair $G=(V(G),E(G))$ where $V(G)$ is a finite set known as the set of vertices of $G$ and $E(G)$ is a subset of $V(G)^{[2]}$ known as the set of edges of $G$. We will say that $G$ is complete if $E(G)=V(G)^{[2]}$. A subgraph $S$ of a graph $G$ is going to be a graph satisfying that $V(S)\subset V(G)$ and $E(S)\subset E(G)$. A path $P$ in $G$ is going to be a subgraph of $G$ such that $V(P)=\{x_1,\dots,x_n\}$ and $E(P)=\{\{x_1,x_2\},\dots,\{x_{n-1},x_n\}\}$ (we will also consider singletons as trivial paths). In that case we say that $P$ is a \textit{path between} $x_1$ \textit{and} $x_n$. A graph G together with a function $w:E(G)\to\R^+$ is called a \textit{weighted graph} in which case $w$ will be known as the \textit{weight} of $G$. Given any path $P$ of a weighted graph $G$ we compute the length of $P$ as
$$length(P)=\sum_{e\in E(P)}w(e).$$
Therefore, if $G$ is connected (every pair of points are connected by at least one path) it is possible to define the weighted graph distance (or geodesic distance) in $V(G)$ given for any $x,y\in V(G)$ by the minimal length of a path from $x$ to $y$ and denoted by $d_G(x,y)$. The metric space given by a weighted graph $G$ is denoted $M(G)$. On the other hand, all finite metric spaces are given by the geodesic distance of some weighted graph. We are going to denote by $G_M$ the complete graph with vertices $V(G_M)=M$.

In this setting, a tree is a graph $T$ such that, for every $x,y\in V(T)$, there is a unique path in $T$ between $x$ and $y$ which will be denoted by $[x,y]_T$. We also denote $[x,y)_T=[x,y]_T\setminus\{y\}$, which is empty if $x=y$. If $T$ comes with a weight function we call it a weighted tree and denote its geodesic distance $d_T$ in which case $(M(T),d_T)$ is called a \textit{geodesic tree}.

Given a finite metric space $(M,d)$ with $N+1$ elements and a distinguished point $0\in M$ we denote $\Sigma(M)$ as the family of possible total orders of $M$. More specifically,
$$\Sigma(M)=\{F:\{0\dots,N\}\to M\;:\;F\text{ is bijective and }F(0)=0\}.$$
Now, given $F\in\Sigma(M)$ we say that a weighted tree $T$ \textit{is compatible with $F$} if $V(T)=M$, $w(e)=d(e)$ for every $e\in E(T)$, and whenever $x,y\in M$ satisfy $x\in[0,y)_T$ then $x<y$ in the order $F$ (see Figure \ref{fig:ordtrees}). The family of trees compatible with $F$ will be denoted by $\mathcal{T}(F)$. Finally, if $T\in \mathcal{T}(F)$ and $x,y\in T$, we call the \textit{meeting point} of $x,y$ in $T$ the minimal element of $[x,y]_T$ in the order $F$ and denote it by $m_T(x,y)$ (see Figure \ref{fig:meet}). It is important to notice that $m_T(x,y)$ does not depend on $F$ since the definition coincides for every other order $F'\in\Sigma(M)$ such that $T\in\mathcal{T}(F')$.

\begin{center}

\tikzset{every picture/.style={line width=0.75pt}} 

\begin{figure}      

\begin{tikzpicture}[x=0.75pt,y=0.75pt,yscale=-1,xscale=1]

\draw    (285.97,54.09) -- (330.82,133.64) ;
\draw [shift={(285.97,54.09)}, rotate = 60.59] [color={rgb, 255:red, 0; green, 0; blue, 0 }  ][fill={rgb, 255:red, 0; green, 0; blue, 0 }  ][line width=0.75]      (0, 0) circle [x radius= 2.34, y radius= 2.34]   ;
\draw    (330.82,133.64) -- (375.67,76.82) ;
\draw [shift={(375.67,76.82)}, rotate = 308.29] [color={rgb, 255:red, 0; green, 0; blue, 0 }  ][fill={rgb, 255:red, 0; green, 0; blue, 0 }  ][line width=0.75]      (0, 0) circle [x radius= 2.34, y radius= 2.34]   ;
\draw    (330.82,133.64) -- (386.88,190.45) ;
\draw [shift={(330.82,133.64)}, rotate = 45.38] [color={rgb, 255:red, 0; green, 0; blue, 0 }  ][fill={rgb, 255:red, 0; green, 0; blue, 0 }  ][line width=0.75]      (0, 0) circle [x radius= 2.34, y radius= 2.34]   ;
\draw    (274.76,167.73) -- (330.82,224.55) ;
\draw [shift={(330.82,224.55)}, rotate = 45.38] [color={rgb, 255:red, 0; green, 0; blue, 0 }  ][fill={rgb, 255:red, 0; green, 0; blue, 0 }  ][line width=0.75]      (0, 0) circle [x radius= 2.34, y radius= 2.34]   ;
\draw [shift={(274.76,167.73)}, rotate = 45.38] [color={rgb, 255:red, 0; green, 0; blue, 0 }  ][fill={rgb, 255:red, 0; green, 0; blue, 0 }  ][line width=0.75]      (0, 0) circle [x radius= 2.34, y radius= 2.34]   ;
\draw    (330.82,224.55) -- (386.88,190.45) ;
\draw [shift={(386.88,190.45)}, rotate = 328.7] [color={rgb, 255:red, 0; green, 0; blue, 0 }  ][fill={rgb, 255:red, 0; green, 0; blue, 0 }  ][line width=0.75]      (0, 0) circle [x radius= 2.34, y radius= 2.34]   ;
\draw    (499,54.66) -- (543.85,134.2) ;
\draw [shift={(499,54.66)}, rotate = 60.59] [color={rgb, 255:red, 0; green, 0; blue, 0 }  ][fill={rgb, 255:red, 0; green, 0; blue, 0 }  ][line width=0.75]      (0, 0) circle [x radius= 2.34, y radius= 2.34]   ;
\draw    (499,54.66) -- (588.7,77.39) ;
\draw [shift={(588.7,77.39)}, rotate = 14.22] [color={rgb, 255:red, 0; green, 0; blue, 0 }  ][fill={rgb, 255:red, 0; green, 0; blue, 0 }  ][line width=0.75]      (0, 0) circle [x radius= 2.34, y radius= 2.34]   ;
\draw    (543.85,134.2) -- (599.91,191.02) ;
\draw [shift={(599.91,191.02)}, rotate = 45.38] [color={rgb, 255:red, 0; green, 0; blue, 0 }  ][fill={rgb, 255:red, 0; green, 0; blue, 0 }  ][line width=0.75]      (0, 0) circle [x radius= 2.34, y radius= 2.34]   ;
\draw [shift={(543.85,134.2)}, rotate = 45.38] [color={rgb, 255:red, 0; green, 0; blue, 0 }  ][fill={rgb, 255:red, 0; green, 0; blue, 0 }  ][line width=0.75]      (0, 0) circle [x radius= 2.34, y radius= 2.34]   ;
\draw    (487.79,168.3) -- (543.85,225.11) ;
\draw [shift={(543.85,225.11)}, rotate = 45.38] [color={rgb, 255:red, 0; green, 0; blue, 0 }  ][fill={rgb, 255:red, 0; green, 0; blue, 0 }  ][line width=0.75]      (0, 0) circle [x radius= 2.34, y radius= 2.34]   ;
\draw [shift={(487.79,168.3)}, rotate = 45.38] [color={rgb, 255:red, 0; green, 0; blue, 0 }  ][fill={rgb, 255:red, 0; green, 0; blue, 0 }  ][line width=0.75]      (0, 0) circle [x radius= 2.34, y radius= 2.34]   ;
\draw    (487.79,168.3) -- (543.85,134.2) ;
\draw [shift={(543.85,134.2)}, rotate = 328.7] [color={rgb, 255:red, 0; green, 0; blue, 0 }  ][fill={rgb, 255:red, 0; green, 0; blue, 0 }  ][line width=0.75]      (0, 0) circle [x radius= 2.34, y radius= 2.34]   ;
\draw  [dash pattern={on 0.84pt off 2.51pt}]  (440,20) -- (440,280) ;
\draw [draw opacity=0]   (117.79,225.11) -- (173.85,191.02) ;
\draw [shift={(173.85,191.02)}, rotate = 328.7] [draw opacity=0][line width=0.75]      (0, 0) circle [x radius= 2.34, y radius= 2.34]   ;
\draw  [dash pattern={on 0.84pt off 2.51pt}]  (230,20) -- (230,280) ;
\draw  [fill={rgb, 255:red, 0; green, 0; blue, 0 }  ,fill opacity=1 ] (59.26,168.3) .. controls (59.26,166.93) and (60.36,165.81) .. (61.73,165.81) .. controls (63.09,165.81) and (64.2,166.93) .. (64.2,168.3) .. controls (64.2,169.67) and (63.09,170.78) .. (61.73,170.78) .. controls (60.36,170.78) and (59.26,169.67) .. (59.26,168.3) -- cycle ;
\draw  [fill={rgb, 255:red, 0; green, 0; blue, 0 }  ,fill opacity=1 ] (115.32,225.11) .. controls (115.32,223.74) and (116.42,222.63) .. (117.79,222.63) .. controls (119.15,222.63) and (120.26,223.74) .. (120.26,225.11) .. controls (120.26,226.48) and (119.15,227.59) .. (117.79,227.59) .. controls (116.42,227.59) and (115.32,226.48) .. (115.32,225.11) -- cycle ;
\draw  [fill={rgb, 255:red, 0; green, 0; blue, 0 }  ,fill opacity=1 ] (171.38,191.02) .. controls (171.38,189.65) and (172.48,188.54) .. (173.85,188.54) .. controls (175.21,188.54) and (176.32,189.65) .. (176.32,191.02) .. controls (176.32,192.39) and (175.21,193.5) .. (173.85,193.5) .. controls (172.48,193.5) and (171.38,192.39) .. (171.38,191.02) -- cycle ;
\draw  [fill={rgb, 255:red, 0; green, 0; blue, 0 }  ,fill opacity=1 ] (115.32,134.2) .. controls (115.32,132.83) and (116.42,131.72) .. (117.79,131.72) .. controls (119.15,131.72) and (120.26,132.83) .. (120.26,134.2) .. controls (120.26,135.57) and (119.15,136.69) .. (117.79,136.69) .. controls (116.42,136.69) and (115.32,135.57) .. (115.32,134.2) -- cycle ;
\draw  [fill={rgb, 255:red, 0; green, 0; blue, 0 }  ,fill opacity=1 ] (160.17,77.39) .. controls (160.17,76.02) and (161.27,74.91) .. (162.64,74.91) .. controls (164,74.91) and (165.11,76.02) .. (165.11,77.39) .. controls (165.11,78.76) and (164,79.87) .. (162.64,79.87) .. controls (161.27,79.87) and (160.17,78.76) .. (160.17,77.39) -- cycle ;
\draw  [fill={rgb, 255:red, 0; green, 0; blue, 0 }  ,fill opacity=1 ] (70.47,54.66) .. controls (70.47,53.29) and (71.58,52.18) .. (72.94,52.18) .. controls (74.3,52.18) and (75.41,53.29) .. (75.41,54.66) .. controls (75.41,56.03) and (74.3,57.14) .. (72.94,57.14) .. controls (71.58,57.14) and (70.47,56.03) .. (70.47,54.66) -- cycle ;

\draw (269.88,174.63) node [anchor=north west][inner sep=0.75pt]  [font=\footnotesize]  {$1$};
\draw (325.94,232.01) node [anchor=north west][inner sep=0.75pt]  [font=\footnotesize]  {$0$};
\draw (383.12,198.49) node [anchor=north west][inner sep=0.75pt]  [font=\footnotesize]  {$2$};
\draw (326.5,140.54) node [anchor=north west][inner sep=0.75pt]  [font=\footnotesize]  {$3$};
\draw (371.35,83.15) node [anchor=north west][inner sep=0.75pt]  [font=\footnotesize]  {$4$};
\draw (281.09,62.13) node [anchor=north west][inner sep=0.75pt]  [font=\footnotesize]  {$5$};
\draw (482.91,175.2) node [anchor=north west][inner sep=0.75pt]  [font=\footnotesize]  {$1$};
\draw (538.97,232.58) node [anchor=north west][inner sep=0.75pt]  [font=\footnotesize]  {$0$};
\draw (596.15,199.06) node [anchor=north west][inner sep=0.75pt]  [font=\footnotesize]  {$2$};
\draw (539.53,141.1) node [anchor=north west][inner sep=0.75pt]  [font=\footnotesize]  {$3$};
\draw (584.38,83.72) node [anchor=north west][inner sep=0.75pt]  [font=\footnotesize]  {$4$};
\draw (494.12,62.7) node [anchor=north west][inner sep=0.75pt]  [font=\footnotesize]  {$5$};
\draw (56.85,175.2) node [anchor=north west][inner sep=0.75pt]  [font=\footnotesize]  {$1$};
\draw (112.91,232.58) node [anchor=north west][inner sep=0.75pt]  [font=\footnotesize]  {$0$};
\draw (170.09,199.06) node [anchor=north west][inner sep=0.75pt]  [font=\footnotesize]  {$2$};
\draw (113.47,141.1) node [anchor=north west][inner sep=0.75pt]  [font=\footnotesize]  {$3$};
\draw (158.32,83.72) node [anchor=north west][inner sep=0.75pt]  [font=\footnotesize]  {$4$};
\draw (68.06,62.7) node [anchor=north west][inner sep=0.75pt]  [font=\footnotesize]  {$5$};
\draw (125,269.4) node     {$F\in \Sigma(M)$};
\draw (320,260.4) node [anchor=north west][inner sep=0.75pt]    {$T_{1}$};
\draw (534,260.4) node [anchor=north west][inner sep=0.75pt]    {$T_{2}$};

\end{tikzpicture}

\caption{$F$ is an ordering of a metric space $M$ with 6 elements. Note that the tree $T_1$ is compatible with $F$ but the tree $T_2$ is not since $5\in [0,4)_{T_2}$ and $3\in [0,2)_{T_2}$.}
    \label{fig:ordtrees}
\end{figure}
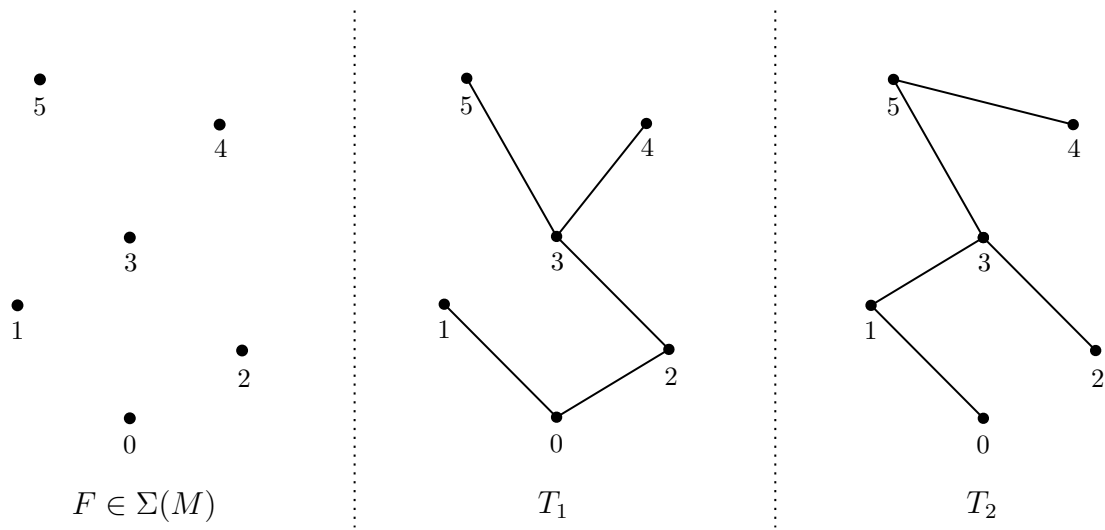
\end{center}

\begin{figure}

\tikzset{every picture/.style={line width=0.75pt}} 

\begin{tikzpicture}[x=0.75pt,y=0.75pt,yscale=-1,xscale=1]

\draw  [fill={rgb, 255:red, 0; green, 0; blue, 0 }  ,fill opacity=1 ] (230,65) .. controls (230,62.24) and (232.24,60) .. (235,60) .. controls (237.76,60) and (240,62.24) .. (240,65) .. controls (240,67.76) and (237.76,70) .. (235,70) .. controls (232.24,70) and (230,67.76) .. (230,65) -- cycle ;
\draw  [fill={rgb, 255:red, 0; green, 0; blue, 0 }  ,fill opacity=1 ] (390,65) .. controls (390,62.24) and (392.24,60) .. (395,60) .. controls (397.76,60) and (400,62.24) .. (400,65) .. controls (400,67.76) and (397.76,70) .. (395,70) .. controls (392.24,70) and (390,67.76) .. (390,65) -- cycle ;
\draw  [fill={rgb, 255:red, 0; green, 0; blue, 0 }  ,fill opacity=1 ] (310,255) .. controls (310,252.24) and (312.24,250) .. (315,250) .. controls (317.76,250) and (320,252.24) .. (320,255) .. controls (320,257.76) and (317.76,260) .. (315,260) .. controls (312.24,260) and (310,257.76) .. (310,255) -- cycle ;
\draw [color={rgb, 255:red, 74; green, 144; blue, 226 }  ,draw opacity=1 ]   (235,65) .. controls (235.19,98.5) and (249.69,154) .. (315,155) ;
\draw [color={rgb, 255:red, 65; green, 117; blue, 5 }  ,draw opacity=1 ]   (395,65) .. controls (394.69,106.5) and (359.19,154.5) .. (315,155) ;
\draw    (315,155) -- (315,255) ;
\draw  [color={rgb, 255:red, 208; green, 2; blue, 27 }  ,draw opacity=1 ][fill={rgb, 255:red, 208; green, 2; blue, 27 }  ,fill opacity=1 ] (310,155) .. controls (310,152.24) and (312.24,150) .. (315,150) .. controls (317.76,150) and (320,152.24) .. (320,155) .. controls (320,157.76) and (317.76,160) .. (315,160) .. controls (312.24,160) and (310,157.76) .. (310,155) -- cycle ;

\draw (307,72.4) node [anchor=north west][inner sep=0.75pt]    {$T$};
\draw (229.5,36.4) node [anchor=north west][inner sep=0.75pt]    {$x$};
\draw (390.5,36.9) node [anchor=north west][inner sep=0.75pt]    {$y$};
\draw (288.5,123.9) node [anchor=north west][inner sep=0.75pt]  [color={rgb, 255:red, 208; green, 2; blue, 27 }  ,opacity=1 ]  {$m_{T}( x,y)$};
\draw (326.5,235.9) node [anchor=north west][inner sep=0.75pt]    {$0$};
\draw (155.5,119.4) node [anchor=north west][inner sep=0.75pt]  [color={rgb, 255:red, 74; green, 144; blue, 226 }  ,opacity=1 ]  {$[ x,m_{T}( x,y)]_{T}$};
\draw (384.5,117.9) node [anchor=north west][inner sep=0.75pt]  [color={rgb, 255:red, 65; green, 117; blue, 5 }  ,opacity=1 ]  {$[ y,m_{T}( x,y)]_{T}$};

\end{tikzpicture}
    \caption{The meeting point of a path $[x,y]_T$ for a tree $T$.}
    \label{fig:meet}
\end{figure}
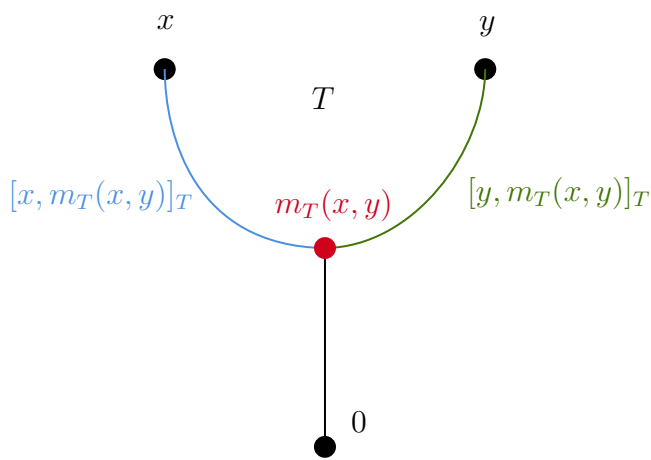

In the following sections we are going to make use of probabilities over a family of trees and, hence, we introduce some notation now. Given $F\in\Sigma(M)$ we denote $\mathcal{P}(\mathcal{T}(F))$ the set of probabilities over $\mathcal{T}(F)$ and
$$\mathcal{P}=\bigcup_{F\in\Sigma(M)}\mathcal{P}(\mathcal{T}(F)).$$
Given an edge $e\in M^{[2]}$, vertices $x,y,z\in M$ and a probability $p\in\mathcal{P}(\mathcal{T}(F))$ we set
$$p(e\in T)=\sum_{\substack{T\in\mathcal{T}(F)\\e\in E(T)}}p(T)\;\;\;\;\;\;\;\;\;\;\text{ and }\;\;\;\;\;\;\;\;\;\;p(x\in [y,z]_T)=\sum_{\substack{T\in\mathcal{T}(F)\\x\in [y,z]_T}}p(T).$$
Also, if $f:\mathcal{T}(F)\to\R$ is given, we denote the \textit{expectation of} $f$ \textit{in} $p$ as
$$\mathbb{E}_p(f)=\sum_{T\in\mathcal{T}(F)}p(T)f(T).$$
In particular, we will focus on the \textit{expected tree distortion of} $p$ \textit{in} $\{x,y\}\in M^{[2]}$ given by
$$\mathbb{E}_p\Big(\frac{d_T(x,y)}{d(x,y)}\Big)=\sum_{T\in\mathcal{T}(F)}p(T)\frac{d_T(x,y)}{d(x,y)}.$$
Our main aim in this work is to connect the expected tree distortion to the $\ell_1^N$-distortion. For this purpose we need to introduce two more concepts.

Given $p\in\mathcal{P}(\mathcal{T}(F))$ and $\{x,y\}\in M^{[2]}$, we say that $p$ is \textit{$(x,y)$-independent} if the probability that an edge $\{s,t\}$ belongs to a random tree $T$ (with $t\in[0,s)_T$) is independent to the event that $s$ belongs to $[x,m_T(x,y))_T$ or $[y,m_T(x,y))_T$. More precisely, for every $z\in\{x,y\}$ and $s,t\in M$ where $s>t$ in the order given by $F$, the following equation holds:
$$p\big(\{s,t\}\in[z,m_T(x,y)]_T\big)=p\big(\{s,t\}\in T\big)\cdot p\big(s\in\big[z,m_T(x,y)\big)_T\big).$$

Secondly, given $p_0,p\in \mathcal{P}$, we say that $p$ and $p_0$ are \textit{compatible} if there exists $F\in \Sigma(M)$ such that $p_0,p\in\mathcal{P}(\mathcal{T}(F))$ and
$$p_0\big(e\in T\big)=p\big(e\in T\big)\;\;\;\text{for every }e\in M^{[2]}.$$
Putting these definitions together, given $p_0\in\mathcal{P}$ and $\{x,y\}\in M^{[2]}$ we define the following subfamily of probabilities:
$$[p_0]_{x,y}=\{p\in\mathcal{P}\;:\;p\text{ is compatible with }p_0\text{ and }(x,y)\text{-independent}\}.$$

\subsection{A metric characterization of stochastic $\ell_1^N$-distortion}

We are finally ready to state our characterization. This characterization connects expected tree distortion to the stochastic $\ell_1^N$-distortion of a space $M$ and serves as a step forward towards answering the central Question \ref{MAINQ}.
\begin{mainth}\label{mainthth}
    Let $(M,d)$ be a finite metric space and $C\geq1$. Then, the following are equivalent:
    \begin{itemize}
        \item The stochastic $\ell_1^N$-distortion of $M$ is smaller than $C$.
        \item There exists $p_0\in\mathcal{P}$ such that for every $\{x,y\}\in M^{[2]}$ there is $p\in[p_0]_{x,y}$ with expected tree distortion in $\{x,y\}$ smaller than $C$.
    \end{itemize}
    Equivalently,
    $$sd_1(M)=\min_{p_0\in\mathcal{P}}\max_{\{x,y\}\in M^{[2]}}\min_{p\in[p_0]_{x,y}}\mathbb{E}_p\Big(\frac{d_T(x,y)}{d(x,y)}\Big).$$
\end{mainth}
Clearly, Theorem \ref{mainthth} gives us a geometric condition on $M$ equivalent to the stochastic $\ell_1^N$-distortion of $M$. That is, it provides a reformulation of the minimal $\ell_1^N$-distortion of a stochastic basis of $\mathcal{F}(M)$ without resorting to the space $\mathcal{F}(M)$ itself. As a straightforward consequence we obtain the following result.

\begin{maincor}\label{MAINCOR}
    Let $(M,d)$ be a finite metric space and $D\geq1$. The following are equivalent:
    \begin{itemize}
        \item $M$ is isomorphic to a geodesic tree with Lipschitz distortion $D$.
        \item There is a basis $(b_n)$ of $\mathcal{F}(M)$ made of molecules with $\ell_1^N$-distortion $D$.
    \end{itemize}
\end{maincor}
Section \ref{sec:mainres} is devoted to the proof of Theorem \ref{mainthth}.

\section{The stochastic $\ell_1^N$-distortion of a metric space}\label{sec:mainres}

This section is devoted to the proof of Theorem \ref{mainthth}. In particular, Subsections \ref{ss21} and \ref{ss22} establish auxiliary results that will be used in the last Subsection \ref{ss23} where the proof of Theorem \ref{mainthth} can be found. Throughout this whole section we fix a finite metric space $(M,d)$ with $N+1$ points for a $N\in\N$ and a distinguished point $0\in M$.

\subsection{Algebraic description of a stochastic basis}\label{ss21}

The unit ball of $\mathcal{F}(M)$ is the convex hull of the elements $\frac{\delta_x-\delta_y}{d(x,y)}$ for $\{x,y\}\in M^{[2]}$ (see \cite[Proposition 3.29]{Wea18}). Then, for every stochastic basis $(b_n)_{n=1}^N$ of $\mathcal{F}(M)$, we have
\begin{equation}\label{stodistbasis}
    d_1(b_n)=\max_{\{x,y\}\in M^{[2]}}\sum_{n=1}^N\Big|b^*_n\Big(\frac{\delta_x-\delta_y}{d(x,y)}\Big)\Big|\,\|b_n\|.
\end{equation}
It is then clear that we want to compute, for each $\{x,y\}\in M^{[2]}$, the value
\begin{equation}\label{eqvalue}
    \sum_{n=1}^N\big|b^*_n\big(\delta_x-\delta_y\big)\big|\,\|b_n\|.
\end{equation}
For this purpose, in this subsection we will prove the following Proposition \ref{propbasisvect} and Theorem \ref{theobasisvect}. Proposition \ref{propbasisvect} gives a  description of the basic elements $b_n$ of a stochastic basis  whereas Theorem \ref{theobasisvect} describes its dual vectors
 $b^*_n$.

\begin{proposition}\label{propbasisvect}
    A basis $(b_n)_{n=1}^N$ of $\mathcal{F}(M)$ is stochastic if and only if there exists $(\rho_n)_{n=1}^N\subset\R^*$, an order $F\in\Sigma(M)$, and a tuple $(\lambda_{n,i})_{\substack{n,i=0,\dots,N\\n>i}}\subset [0,1]$ with $\sum_{i=0}^{n-1}\lambda_{n,i}=1$ for every $1\leq n\leq N$ such that
    $$b_n=\rho_n\Big(\delta_{F(n)}-\sum_{i<n}\lambda_{n,i}\delta_{F(i)}\Big)\;\;\;\;\;\;\text{for }\;\;\; n=1,\dots,N.$$
    Therefore, for every $n=1,\dots,N$ we get
    $$\|b_n\|=|\rho_n|\sum_{i<n}\lambda_{n,i}d(F(n),F(i)).$$
\end{proposition}

Given a stochastic basis $(b_n)_{n=1}^N$ of $\mathcal{F}(M)$, we say that it is \textit{normalised} if $\rho_n=1$ for every $n=1,\dots,N$, where $(\rho_n)_{n=1}^N$ is the sequence given in Proposition \ref{propbasisvect}. We denote $B(M)$ the family of normalised stochastic bases of $\mathcal{F}(M)$.

\begin{theorem}\label{theobasisvect}
    Let $\{x,y\}\in M^{[2]}$ and $(b_n)_{n=1}^N\in B(M)$ be given. If $F\in \Sigma(M)$ and $(\lambda_{n,i})_{n,i}$ are like in Proposition \ref{propbasisvect}, then for $n=1,\dots,N$,
    $$b^*_n(\delta_x-\delta_y)=(\delta_x-\delta_y)(F(n))+\sum_{\substack{m=n+1}}^Nb^*_m(\delta_x-\delta_y)\lambda_{m,n}.$$
\end{theorem}

\begin{proof}[Proof of Proposition \ref{propbasisvect}]
    Let us consider the canonical projections of the basis $(b_n)$ and denote them by $P_n:\mathcal{F}(M)\to\mathcal{F}(M)$. We first assume that the $P_n$'s are stochastic retractions and prove that $(b_n)$ is given as in the statement of Proposition \ref{propbasisvect}. We claim that there is an order $F\in\Sigma(M)$ such that, if we denote $M_n=(F(i))_{i=1}^{n}$ for every $n\in\{1,\dots,N\}$, then $P_n(\mathcal{F}(M))=\mathcal{F}(M_n)$ and $M_N=M$.
    
    Let us prove the claim inductively in $n$. For $n=1$ we consider $F(1)\in M\setminus\{0\}$ such that $P_1(\mathcal{F}(M))=\mathcal{F}(F(1))$. Now, for the inductive step, we assume that we have defined $(F(i))_{i=0}^{n-1}$ with $P_{n-1}(\mathcal{F}(M))=\mathcal{F}(M_{n-1})$. We also know from the definition of stochastic projection that $P_n(\mathcal{F}(M))=\mathcal{F}(S)$ for some subset $S\subset M$. Since $\mathcal{F}(S)=P_n(\mathcal{F}(M))\supset P_{n-1}(\mathcal{F}(M))=\mathcal{F}(M_{n-1})$ we have $M_{n-1}\subset S$ and since $dim(P_n(\mathcal{F}(M)))=dim(\mathcal{F}(M_{n-1}))+1$ we have $|S|=|M_{n-1}|+1$ which implies that there is a unique element $F(n)\in M\setminus M_{n-1}$ such that $S=M_{n-1}\cup\{F(n)\}=M_n$ and we are done with the induction. Clearly, $\mathcal{F}(M_N)=P_N(\mathcal{F}(M))=\mathcal{F}(M)$.

   Let us identify $n=F(n)$ for every $n=0,\dots,N$. We finally prove the conclusion of the proposition and find the tuples $(\lambda_{n,i})$. For each $n=1,\dots,N$ it is clear that $\delta_n\in \mathcal{F}(M_n)\setminus\mathcal{F}(M_{n-1})$ and hence $b_n^*(\delta_n)\neq0$. Also, since $P_n$ is a stochastic retraction there must be a probability $\mu_n\in \mathcal{P}(M_{n-1})$ such that $P_{n-1}(\delta_n)=\mu_n$. We finally take $\rho_n=\frac{1}{b_n^*(\delta_n)}$ and $(\lambda_{n,i})_{i=0}^{n-1}\subset[0,1]$ with $\sum_{i=0}^{n-1}\lambda_{n,i}=1$ such that $\mu_n=\sum_{i=0}^{n-1}\lambda_{n,i}\delta_i$. It is immediate that
    $$b_n=\rho_n\big(P_n-P_{n-1}\big)(\delta_n)=\rho_n\Big(\delta_n-\sum_{i=0}^{n-1}\lambda_{n,i}\delta_i\Big).$$

    Finally, it is clear that in this case
    $$\|b_n\|=\Big\|\rho_n\sum_{i=0}^{n-1}\lambda_{n,i}(\delta_n-\delta_i)\Big\|\leq |\rho_n|\sum_{i=0}^{n-1}\lambda_{n,i}\|\delta_n-\delta_i\|=|\rho_n|\sum_{i=0}^{n-1}\lambda_{n,i}d(n,i).$$
    For the other inequality, by Kantorovich duality we know that $\|b_n\|=\sup_{f\in S_{Lip_0(M)}} b_n(f)$. Hence, if we take $f_n(x)=d(n,x)-d(n,0)$ it is immediate that $f_n\in S_{Lip_0(M)}$ and
    $$\|b_n\|\geq b_n(f_n)=|\rho_n|\sum_{i<n}\lambda_{n,i}d(n,i).$$

    Now, for the converse, If $(b_n)$ is given as in Proposition \ref{propbasisvect} then we put for $n=1,\dots,N$ the subset $M_n=\{F(i)\}_{i=0}^n$ and define the linear projection $p_n:\mathcal{F}(M_{n})\to\mathcal{F}(M_{n-1})$ given by
    $$p_n(\delta_{F(n)})=\sum_{i<n}\lambda_{n,i}\delta_{F(i)}.$$
    Now, it is straightforward to check that the projections given by $P_n=p_{n+1}\circ\cdots\circ p_N:\mathcal{F}(M)\to\mathcal{F}(M_n)$ are the canonical projections of the basis $(b_n)$ and are also stochastic retractions.
\end{proof}

The basic vectors of a stochastic basis may, therefore, be understood as splits of mass from one point to the points below. See Figure \ref{fig:split} for a representation of a vector $b_n$ from a stochastic basis as a transportation problem.

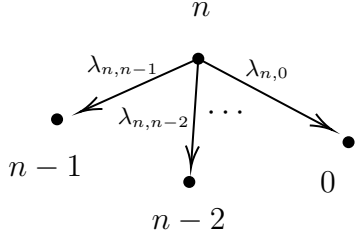
\begin{figure}

\tikzset{every picture/.style={line width=0.75pt}} 

\begin{tikzpicture}[x=0.75pt,y=0.75pt,yscale=-1,xscale=1]

\draw  [fill={rgb, 255:red, 0; green, 0; blue, 0 }  ,fill opacity=1 ] (227.67,91.17) .. controls (227.67,89.79) and (228.79,88.67) .. (230.17,88.67) .. controls (231.55,88.67) and (232.67,89.79) .. (232.67,91.17) .. controls (232.67,92.55) and (231.55,93.67) .. (230.17,93.67) .. controls (228.79,93.67) and (227.67,92.55) .. (227.67,91.17) -- cycle ;
\draw  [fill={rgb, 255:red, 0; green, 0; blue, 0 }  ,fill opacity=1 ] (302.67,133.17) .. controls (302.67,131.79) and (303.79,130.67) .. (305.17,130.67) .. controls (306.55,130.67) and (307.67,131.79) .. (307.67,133.17) .. controls (307.67,134.55) and (306.55,135.67) .. (305.17,135.67) .. controls (303.79,135.67) and (302.67,134.55) .. (302.67,133.17) -- cycle ;
\draw  [fill={rgb, 255:red, 0; green, 0; blue, 0 }  ,fill opacity=1 ] (223.17,153.17) .. controls (223.17,151.79) and (224.29,150.67) .. (225.67,150.67) .. controls (227.05,150.67) and (228.17,151.79) .. (228.17,153.17) .. controls (228.17,154.55) and (227.05,155.67) .. (225.67,155.67) .. controls (224.29,155.67) and (223.17,154.55) .. (223.17,153.17) -- cycle ;
\draw  [fill={rgb, 255:red, 0; green, 0; blue, 0 }  ,fill opacity=1 ] (157.17,121.67) .. controls (157.17,120.29) and (158.29,119.17) .. (159.67,119.17) .. controls (161.05,119.17) and (162.17,120.29) .. (162.17,121.67) .. controls (162.17,123.05) and (161.05,124.17) .. (159.67,124.17) .. controls (158.29,124.17) and (157.17,123.05) .. (157.17,121.67) -- cycle ;
\draw    (230.17,91.17) -- (172.82,117.17) ;
\draw [shift={(171,118)}, rotate = 335.6] [color={rgb, 255:red, 0; green, 0; blue, 0 }  ][line width=0.75]    (10.93,-3.29) .. controls (6.95,-1.4) and (3.31,-0.3) .. (0,0) .. controls (3.31,0.3) and (6.95,1.4) .. (10.93,3.29)   ;
\draw    (230.17,91.17) -- (226.63,144) ;
\draw [shift={(226.5,146)}, rotate = 273.83] [color={rgb, 255:red, 0; green, 0; blue, 0 }  ][line width=0.75]    (10.93,-3.29) .. controls (6.95,-1.4) and (3.31,-0.3) .. (0,0) .. controls (3.31,0.3) and (6.95,1.4) .. (10.93,3.29)   ;
\draw    (230.17,91.17) -- (295.24,126.05) ;
\draw [shift={(297,127)}, rotate = 208.2] [color={rgb, 255:red, 0; green, 0; blue, 0 }  ][line width=0.75]    (10.93,-3.29) .. controls (6.95,-1.4) and (3.31,-0.3) .. (0,0) .. controls (3.31,0.3) and (6.95,1.4) .. (10.93,3.29)   ;

\draw (225.5,61.4) node [anchor=north west][inner sep=0.75pt]    {$n$};
\draw (233,113.4) node [anchor=north west][inner sep=0.75pt]    {$\cdots $};
\draw (134,138.4) node [anchor=north west][inner sep=0.75pt]    {$n-1$};
\draw (205.5,164.9) node [anchor=north west][inner sep=0.75pt]    {$n-2$};
\draw (289.5,146.4) node [anchor=north west][inner sep=0.75pt]    {$0$};
\draw (173,87.9) node [anchor=north west][inner sep=0.75pt]  [font=\scriptsize]  {$\lambda _{n,n-1}$};
\draw (252,89.9) node [anchor=north west][inner sep=0.75pt]  [font=\scriptsize]  {$\lambda _{n,0}$};
\draw (188.5,114.4) node [anchor=north west][inner sep=0.75pt]  [font=\scriptsize]  {$\lambda _{n,n-2}$};

\end{tikzpicture}
    \caption{Representation of a vector $b_n=\delta_n-\sum_{i<n}\lambda_{n,i}\delta_i$ as a transportation problem.}
    \label{fig:split}
\end{figure}

We focus now on the proof of Theorem \ref{theobasisvect} for which we will need some auxiliary results and preparation. We fix during this subsection a normalised stochastic basis $(b_n)\in B(M)$ and a pair $\{x,y\}\in M^{[2]}$. Let us consider $F\in\Sigma(M)$ and $(\lambda_{n,i})_{n,i}$ as in Proposition \ref{propbasisvect}. We may identify $M$ with $\{0,\dots,N\}$ by means of $F$ so that we denote $n=F(n)$ for every $n=0,\dots,N$. We also assume without loss of generality that $x>y$ under this identification.

Given any function $\mu:M\to\mathbb{R}$ we define an associated function $\alpha(\mu):M\to\mathbb{R}$ inductively. We put $\alpha(\mu)(N)=\mu(N)$ and, for every $n<N$,
\begin{equation}\label{defalp}
    \alpha(\mu)(n)=\mu(n)+\sum_{\substack{m\in M\\m>n}}\alpha(\mu)(m)\lambda_{m,n}.
\end{equation}

We need to analyse the behaviour of these associated functions in order to prove Theorem \ref{theobasisvect}.

\begin{lemma}\label{greatalpha}
    If $n\in M$ with $n>y$ then $\alpha(\delta_x-\delta_y)(n)\geq0$. Moreover, $\alpha(\delta_x-\delta_y)(n)=0$ whenever $n>x$ and $\alpha(\delta_x-\delta_y)(x)=1$.
\end{lemma}
\begin{proof}
For simplicity, we put $\alpha(n):=\alpha(\delta_x-\delta_y)(n)$ for every $n\in M$ and $\mu:=\delta_x-\delta_y$.
     Let us first prove that $\alpha(n)=0$ whenever $n>x$. If $x=N$ then the latter is an empty statement. Otherwise, we assume $y<x<N$ and proceed by induction on $n$. If $n=N$ it is immediate since $\alpha(N)=\mu(N)=0$. Now, for $n<N$ we assume that $\alpha(m)=0$ for every $m>n$. Then $\mu(n)=0$ since $y<x<n$ and hence
     $$\alpha(n)=\mu(n)+\sum_{m>n}\lambda_{m,n}\alpha(m)=0.$$
     Therefore, $\alpha(x)=\mu(x)+\sum_{m>x}\lambda_{m,x}\alpha(m)=\mu(x)=1$.
    
    Finally, we show by induction that $\alpha(n)\geq 0$ for every $n>y$. We have just shown that $\alpha(k)\geq0$ for $k\geq x$ so that the first step of the induction is complete. Now, for the inductive step, consider $y<n<x$ and assume that $\alpha(m)\geq0$ for every $m>n$. Taking into account that $\mu(n)\geq0$ it follows that
    $$\alpha(n)=\mu(n)+\sum_{\substack{m\in M\\m>n}}\alpha(m)\lambda_{m,n}\geq0.$$
\end{proof}

\begin{lemma}\label{basicmu}
   Let $\mu:M\to\mathbb{R}$ be an arbitrary function and $n_0\in M$, then,
   $$\alpha(\mu)(0)=\sum_{n\in M}\mu(n)\;\;\;\;\text{ and }\;\;\;\;\alpha(\mu)(n_0)\leq\mu(n_0)+\sum_{\substack{m\in M\\m>n_0}}|\mu(m)|.$$
        
\end{lemma}
\begin{proof}
    Let us prove the first equation inductively on $N=|M|-1$. It is clear when $N=0$ ($|M|=1$). Assume it holds for $M\setminus\{N\}=\{0,\dots,N-1\}$.
    Let $\mu:M\to\mathbb{R}$ be given and consider $\widetilde\mu:M\setminus\{N\}\to\mathbb{R}$ given by
    $$\widetilde\mu(n)=\mu(n)+\lambda_{N,n}\mu(N)\;\;,\;\;\;\;\;n\in M\setminus\{N\}.$$
    Clearly,
    $$\sum_{n\in M}\mu(n)=\sum_{n\in M\setminus\{N\}}\widetilde\mu(n).$$
    Therefore, taking into account our induction hypothesis, it is more than enough to prove that
    \begin{equation}\label{eqclaim}
        \alpha(\mu)(n)=\alpha({\widetilde\mu})(n)\;\;\;\;\forall n\in M\setminus \{N\}=\{0,\dots,N-1\}.
    \end{equation}
    We are going to prove \eqref{eqclaim} again by induction. If $n=N-1$ then $\alpha(\mu)(n)=\mu(n)+\lambda_{N,n}\mu(N)=\alpha({\widetilde\mu})(n)$. Consider now $n<N-1$ and assume that for every $m>n$ we have that $\alpha(\mu)(m)=\alpha({\widetilde\mu})(m)$. Then,
    $$\begin{aligned}\alpha(\mu)(n)=&\mu(n)+\sum_{\substack{m\in M\\m>n}}\lambda_{m,n}\alpha(\mu)(m)=\mu(n)+\lambda_{N,n}\alpha(\mu)(N)+\sum_{\substack{m\in M\setminus\{N\}\\m>n}}\lambda_{m,n}\alpha({\widetilde\mu})(m)\\\stackrel{\alpha(\mu)(N)=\mu(N)}{=}&\widetilde\mu(n)+\sum_{\substack{m\in M\setminus\{N\}\\m>n}}\lambda_{m,n}\alpha({\widetilde\mu})(m)=\alpha({\widetilde\mu})(n).\end{aligned}$$

    Now, to prove the second equation we proceed analogously by induction on $N=|M|-1$. If $N=0$ ($|M|=1$) it is trivially true. Otherwise, consider $N\geq1$ and assume as induction hypothesis (IH) that the statement of this Lemma holds for $M\setminus\{N\}$. The second equation in the statement is immediate for $n_0=N$ so let us assume that $n_0<N$. If we construct $\widetilde\mu:M\setminus\{N\}\to\R$ as above we finally get that
    $$\begin{aligned}\alpha(\mu)(n_0)\stackrel{\eqref{eqclaim}}{=}\alpha({\widetilde\mu})(n_0)\stackrel{(IH)}{\leq}& \widetilde\mu(n_0)+\sum_{\substack{m\in M\setminus\{N\}\\m>n_0}}|\widetilde\mu(m)|\\=&\mu(n_0)+\lambda_{N,n_0}\mu(N)+\sum_{\substack{m\in M\setminus\{N\}\\m>n_0}}|\mu(m)+\lambda_{N,m}\mu(N)|\\\leq& \mu(n_0)+\lambda_{N,n_0}|\mu(N)|+\sum_{\substack{m\in M\setminus\{N\}\\m>n_0}}|\mu(m)|+\lambda_{N,m}|\mu(N)|\\\leq&\mu(n_0)+\sum_{\substack{m\in M\\m>n_0}}|\mu(m)|.\end{aligned}$$
\end{proof}

Let us now introduce some definitions that will come in handy.

A \textit{chain} $C$ for the order $F$ is a path of $G_M$ such that $E(C)=\{\{x_1,x_2\},\dots,\{x_{n-1},x_n\}\}$ satisfying $x_1>x_2>\dots>x_{n-1}>x_n$ in the order given by $F$ (including trivial chains with a unique vertex). Given $s,t\in M$ we consider the family of chains from $s$ to $t$,
$$\overrightarrow{[s,t]}=\{C\text{ chain with }\max C=s\;,\;\min C=t\}.$$
Clearly, if $s<t$ then $\overrightarrow{[s,t]}=\emptyset$. 

We also denote for every $e\in M^{[2]}$ the value $\lambda_e\geq0$ given by $\lambda_{\max(e),\min(e)}$ and define the \textit{product probability given by }$(b_n)$, denoted by $\pi_{(b_n)}$, as
$$\pi_{(b_n)}(T)=\prod_{e\in E(T)}\lambda_{e}\;\;\;\;\;\;\;\;\;\;\;\text{for }\;\;\; T\in\mathcal{T}(F).$$

\begin{lemma}\label{lemprodprob}
    The function $\pi_{(b_n)}$ is a probability distribution over $\mathcal{T}(F)$. Moreover, the following properties are satisfied,
    \begin{itemize}
        \item  For every subgraph $S$ of a tree in $\mathcal{T}(F)$,
        \begin{equation}\label{prodeq2}\sum_{\substack{T\in\mathcal{T}(F)\\S\subset T}}\prod_{\{n,i\}\in E(T)\setminus E(S)}\lambda_{n,i}=1,\end{equation}
        \begin{equation}\label{prodeq3}\pi_{(b_n)}(T\;:\;S\subset T)=\prod_{\{n,i\}\in E(S)}\lambda_{n,i}.\end{equation}
        \item  For every $1\leq z,n\leq N$,
        \begin{equation}\label{prodeq} b^*_n(\delta_z)=\pi_{(b_n)}\big(\overrightarrow{[z,n]}\big):=\sum_{\substack{C\in\overrightarrow{[z,n]}}}\pi_{(b_n)}(T:C\subset T).\end{equation}
    \end{itemize}
\end{lemma}
\begin{proof}

Let us denote for simplicity during this proof $\pi:=\pi_{(b_n)}$. We prove \eqref{prodeq2} by induction in $|E(S)|$. Since $S$ is a subgraph of a tree, we know that $|E(S)|\leq N$ and hence we may start the induction with the case $|E(S)|=N$. In this case, $S$ is a tree of $\mathcal{T}(F)$ and \eqref{prodeq2} is immediate since we agreed that an empty product is equal to 1. Now, for the inductive step, let us assume as induction hypothesis $(IH)$ the equation \eqref{prodeq2} for every other subgraph $\widetilde S$ of a tree in $\mathcal{T}(F)$ with $\big|E\big(\widetilde S\big)\big|=|E(S)|+1$. Assume for now that $0$ is in a connected component of $S$, say $S'$. Hence, we may define $m:=\min\{k\in M\setminus V(S')\}$ and note that $m\neq 0$. In this case we denote for every $k<m$ the graph $S_k$ with $V(S_k)=V(S)\cup\{m,k\}$ and $E(S_k)=E(S)\cup\{\{m,k\}\}$. Since $m\notin S$ we know that $|E(S_k)|=|E(S)|+1$ and $S_k$ is a subgraph of a tree in $\mathcal{T}(F)$. Therefore,
$$\begin{aligned}\sum_{\substack{T\in\mathcal{T}(F)\\S\subset T}}\prod_{\{n,i\}\in E(T)\setminus E(S)}\lambda_{n,i}=&\sum_{k=0}^{m-1}\sum_{\substack{T\in\mathcal{T}(F)\\S_k\subset T}}\prod_{\{n,i\}\in E(T)\setminus E(S)}\lambda_{n,i}\\=&\sum_{k=0}^{m-1}\sum_{\substack{T\in\mathcal{T}(F)\\S_k\subset T}}\lambda_{m,k}\prod_{\{n,i\}\in E(T)\setminus E(S_k)}\lambda_{n,i}\\=&\sum_{k=0}^{m-1}\lambda_{m,k}\sum_{\substack{T\in\mathcal{T}(F)\\S_k\subset T}}\prod_{\{n,i\}\in E(T)\setminus E(S_k)}\lambda_{n,i}\\\stackrel{(IH)}{=}&\sum_{k=0}^{m-1}\lambda_{m,k}=1.\end{aligned}$$

Finally, if $0\notin S$ then  the exact same proof works for $m=\min S\neq0$.

Now, equation \eqref{prodeq3} follows directly from \eqref{prodeq2} since
    $$\begin{aligned}\pi(T\;:\;S\subset T)=&\sum_{\substack{T\in\mathcal{T}(F)\\S\subset T}}\pi(T)=\sum_{\substack{T\in\mathcal{T}(F)\\S\subset T}}\prod_{\{n,i\}\in E(T)}\lambda_{n,i}\\=&\prod_{\{n,i\}\in E(S)}\lambda_{n,i}\bigg(\sum_{\substack{T\in\mathcal{T}(F)\\S\subset T}}\prod_{\{m,j\}\in E(T)\setminus E(S)}\lambda_{m,j}\bigg)\stackrel{\eqref{prodeq2}}{=}\prod_{\{n,i\}\in E(S)}\lambda_{n,i}.\end{aligned}$$
    If we consider a trivial subgraph $S$ with only one vertex and no edges we get from \eqref{prodeq2} that $\pi$ is a probability. We tackle now the proof of \eqref{prodeq} dividing it into 2 cases.

    \textbf{Proof of \eqref{prodeq}, Case 1:} $z<n$. In this case, the right hand side of \eqref{prodeq} is null and hence we have to show that
    \begin{equation}\label{subeq1}
        b_n^*(\delta_z)=0.
    \end{equation}
    We proceed by induction on $z$. If $z=1$ then $\delta_z=b_z$ and \eqref{subeq1} is immediate. Otherwise, let us assume as induction hypothesis $(IH)$ that $b_n^*(\delta_i)=0$ for every $i<z$. Then,
    $$b^*_n(\delta_z)\stackrel{\text{Prop. }\ref{propbasisvect}}{=}b_n^*\bigg(b_z+\sum_{i<z}\lambda_{z,i}\delta_i\bigg)=b_n^*(b_z)+\sum_{i<z}\lambda_{z,i}b_n^*(\delta_i)\stackrel{(IH)}{=}0$$

    \textbf{Proof of \eqref{prodeq}, Case 2:} $n\leq z$. We prove this case by induction on $z-n$. If $z-n=0$ then 
    $$b_n^*(\delta_z)\stackrel{\text{Prop. }\ref{propbasisvect}}{=}b_z^*\bigg(b_z+\sum_{i<z}\lambda_{z,i}\delta_i\bigg)=b_z^*(b_z)+\sum_{i<z}\lambda_{z,i}b_z^*(\delta_i)\stackrel{\text{Case }1}{=}1.$$
    Now, for the inductive step we take $n,z$ with $z-n\geq 1$ and asssume that, for every $i<z$,
    $$b^*_n(\delta_i)=\sum_{\substack{C\in\overrightarrow{[i,n]}}}\pi(T:C\subset T),$$
    then,
    $$\begin{aligned}
        b^*_n(\delta_z)\stackrel{\text{Prop. }\ref{propbasisvect}}{=}&b_n^*\bigg(b_z+\sum_{i<z}\lambda_{z,i}\delta_i\bigg)=\sum_{i<z}\lambda_{z,i}b_n^*(\delta_i)\stackrel{(IH)}{=}\sum_{i<z}\lambda_{z,i}\sum_{\substack{C\in\overrightarrow{[i,n]}}}\pi(T:C\subset T)\\\stackrel{\eqref{prodeq3}}{=}&\sum_{i<z}\sum_{\substack{C\in\overrightarrow{[i,n]}}}\lambda_{z,i}\prod_{e\in C}\lambda_e=\sum_{\substack{C\in\overrightarrow{[z,n]}}}\prod_{e\in C}\lambda_e\stackrel{\eqref{prodeq3}}{=}\sum_{\substack{C\in\overrightarrow{[z,n]}}}\pi(T:C\subset T).
    \end{aligned}$$
\end{proof}

We are finally ready to prove Theorem \ref{theobasisvect}.

\begin{proof}[Proof of Theorem \ref{theobasisvect}]
Let us denote for simplicity $\alpha(n):=\alpha(\delta_x-\delta_y)(n)$ for every $n\in M$ and $\pi:=\pi_{(b_n)}$. It is enough to show that for every $n=1,\dots,N$,
$$\alpha(n)=b^*_n(\delta_x-\delta_y)$$
    By linearity, it is enough to prove it in the case $y=0$. Hence, by equation \eqref{prodeq} of Lemma \ref{lemprodprob} we will be done if we prove that for every $n=1,\dots,N$,
    \begin{equation}\label{alpeq}
        \sum_{C\in\overrightarrow{[x,n]}}\pi(T:C\subset T)=\alpha(n).
    \end{equation}
    We will also divide the proof of \eqref{alpeq} into different cases.

    \textbf{Proof of \eqref{alpeq}, Case 1:} $x< n$. This case is immediate since, by Lemma \ref{greatalpha}, we know that $\alpha(n)=0$ for every $n>x$.
    
    \textbf{Proof of \eqref{alpeq}, Case 2:} $n\leq x$. In this case, it is enough to show that
    \begin{equation}\label{red1}
        \sum_{C\in\overrightarrow{[x,n]}}\pi(T:C\subset T)=\alpha(n).
    \end{equation}
    Let us prove \eqref{red1} by induction. If $n=x$ then by Lemma \ref{greatalpha} we know that $\alpha(x)=\delta_x(x)=1$ so that \eqref{red1} is immediate.
    Now, for the inductive step, let us assume as induction hypothesis that
    \begin{equation}\label{IHeq}
        \sum_{C\in\overrightarrow{[x,m]}}\pi(T:C\subset T)=\alpha(m),\;\;\;\;\text{ for every }\;\;n<m\leq x.
    \end{equation}
    It is clear that \eqref{IHeq} also holds for $m> x$ by Lemma \ref{greatalpha}. Then,
    $$\sum_{C\in\overrightarrow{[x,n]}}\pi(T:C\subset T)\stackrel{\eqref{prodeq3}}{=}\sum_{m>n}\lambda_{m,n}\sum_{C\in\overrightarrow{[x,m]}}\pi(T:C\subset T)\stackrel{\eqref{IHeq}}{=}\sum_{\substack{m>n}}\lambda_{m,n}\alpha(m)=\alpha(n).$$

    
\end{proof}

Finally, thanks to Theorem \ref{theobasisvect} we may describe the dual basis $(b^*_n)$ of a basis $\beta:=(b_n)\in B(M)$ from the perspective of linear algebra. Considering the tuple $(\lambda_{n,i})_{n,i}$ given in Proposition \ref{propbasisvect}, we define a $(N+1)\times(N+1)$-matrix $T_\beta$ given by
$$T_\beta=\begin{pmatrix}
    0 & \lambda_{1,0} & \lambda_{2,0} & \dots & \lambda_{k,0} & \dots & \lambda_{N-1,0} & \lambda_{N,0} \\
    0 & 0 & \lambda_{2,1} & \dots & \lambda_{k,1} & \dots & \lambda_{N-1,1} & \lambda_{N,1} \\
    0 & 0 & 0 & \dots & \lambda_{k,2} & \dots & \lambda_{N-1,2} & \lambda_{N,2} \\
    \vdots & \vdots & \vdots & \ddots & \vdots &  & \vdots & \vdots\\
    0 & 0 & 0 & \dots & 0 & \dots & \lambda_{N-1,k} & \lambda_{N,k}\\
    \vdots & \vdots & \vdots &  & \vdots & \ddots & \vdots & \vdots \\
    0 & 0 & 0 & \dots & 0 & \dots & 0 & \lambda_{N,N-1}\\
    0 & 0 & 0 & \dots & 0 & \dots  & 0 & 0
\end{pmatrix}.$$
Also, for every $\mu\in\mathcal{F}(M)$ we know there is a unique representation $\mu=\sum_{n=0}^N\delta^*_n(\mu)\delta_n$ with $\sum_{n=0}^N\delta^*_n(\mu)=0$. Then, we define the vectors
$$\delta^*(\mu)=\begin{pmatrix}
    \delta^*_0(\mu)\\\delta^*_1(\mu)\\\vdots\\\delta^*_N(\mu)
\end{pmatrix}\;\;\;\;\;\;\;\;\;\text{ and }\;\;\;\;\;\;\;\;\;b^*(\mu)=\begin{pmatrix}
    0\\b^*_1(\mu)\\\vdots\\b^*_N(\mu)
\end{pmatrix}.$$

We then have the following result on the change of basis matrix between $(\delta_n)$ and $(b_n)$.

\begin{theorem}\label{complex}
    For every $\mu\in\mathcal{F}(M)$ and every basis $\beta:=(b_n)_{n=1}^N\in B(M)$ we have 
    $$b^*(\mu)=\delta^*(\mu)+T_\beta \delta^*(\mu)+\dots+T_\beta^N\delta^*(\mu)=\sum_{n=0}^\infty T_\beta^n\delta^*(\mu)=\big(I-T_\beta\big)^{-1}\delta^*(\mu).$$
\end{theorem}
\begin{proof}By Theorem \ref{theobasisvect} we have $(I-T_\beta)b^*(\mu)=\mu$. The rest is folklore from linear algebra and matrix calculus.
\end{proof}

\subsection{The effective charge probability}\label{ss22}

This subsection is devoted to the construction of the probabilities that appear in the statement of Theorem \ref{mainthth}. More precisely, we will prove the following two results making use of the previously shown Proposition \ref{propbasisvect} and Theorem \ref{theobasisvect}.

\begin{theorem}\label{th1}
    Let $\{x,y\}\in M^{[2]}$ and $(b_n)\in B(M)$ be given. If $p\in[\pi_{(b_n)}]_{x,y}$ then
    $$\sum_{n=1}^N|b^*_n(\delta_x-\delta_y)|\|b_n\|\leq\mathbb{E}_p\big(d_T(x,y)\big).$$
\end{theorem}

\begin{theorem}\label{th2}
    For every $\{x,y\}\in M^{[2]}$ and every $(b_n)\in B(M)$ there is $p\in[\pi_{(b_n)}]_{x,y}$ such that
    $$\sum_{n=1}^N|b^*_n(\delta_x-\delta_y)|\|b_n\|=\mathbb{E}_p\big(d_T(x,y)\big).$$
\end{theorem}

\begin{proof}[Proof of Theorem \ref{th1}]
    We first fix $\{x,y\}\in M^{[2]}$ and $(b_n)\in B(M)$. Consider $F\in\Sigma(M)$ and $(\lambda_{n,i})_{n,i}$ as in Proposition \ref{propbasisvect}. We assume without loss of generality that $x>y$ in $F$ and pick $p\in[\pi_{(b_n)}]_{x,y}$ arbitrarily. We denote for simplicity $\pi:=\pi_{(b_n)}$ as well as $m_T:=m_T(x,y)$ for every $T\in\mathcal{T}(F)$ and identify $M$ with $\{0,\dots,N\}$ by F.

    By equation \eqref{prodeq3} of Lemma \ref{lemprodprob} we know that $\pi(\{n,i\}\in T)=\lambda_{n,i}$ for every $n>i$. Now, since $p$ is compatible with $\pi$ we get that $p(\{n,i\}\in T)=\lambda_{n,i}$ for $n>i$ and from the fact that $p$ is $(x,y)$-independent we deduce that
    \begin{equation}\label{thhyp2}
        p(\{n,i\}\in [z,m_T]_T)=p(n\in[z,m_T)_T)\lambda_{n,i}\;\;\;\;\text{ for }n>i\text{ and }z=x,y.
    \end{equation}
    Consider $\alpha_p:M\to\R$ given by
    $$\alpha_p(n)=p(T\;:\;n\in V([x,y]_T)\setminus\{m_T\})\;\;\;\;\;\;\;\;\;\;\;\;\;\;\;(n\in M).$$
    From our hypothesis we get that
    \begin{equation}\label{alphaq2}\begin{aligned}\alpha_p(n)\lambda_{n,i}=&\big(p(n\in [x,m_T)_T)+p(n\in [y,m_T)_T)\big)\lambda_{n,i}\\\stackrel{\eqref{thhyp2}}{=}&p\big(\{n,i\}\in [x,m_T]_T\big)+p\big(\{n,i\}\in [y,m_T]_T\big)\\=&p(\{n,i\}\in [x,y]_T).\end{aligned}\end{equation}
    Therefore, by decomposing the edges of the distinct path from each tree and swapping sums we have
    $$\begin{aligned}\mathbb{E}_p(d_T(x,y))=&\sum_{n=1}^N\sum_{i=0}^{n-1}p(\{n,i\}\in [x,y]_T)d(n,i)\\\stackrel{\eqref{alphaq2}}{=}&\sum_{n=1}^N\sum_{i=0}^{n-1}\alpha_p(n)\lambda_{n,i}d(n,i).\end{aligned}$$
    Now, if we denote $\alpha(n):=b^*_n(\delta_x-\delta_y)$ for every $n\in M$ we have that
    $$\sum_{n=1}^N|b^*_n(\delta_x-\delta_y)|\|b_n\|\stackrel{\text{Prop }\ref{propbasisvect}}{=}\sum_{n=1}^N\sum_{i=0}^{n-1}|\alpha(n)|\lambda_{n,i}d(n,i).$$
    Therefore, it is enough to show that $|\alpha(n)|\leq\alpha_p(n)$ for every $n\in M$. For that, it suffices to prove that
    \begin{equation}\label{eqclaim1232}
        p(n\in[x,m_T)_T)-p(n\in[y,m_T)_T)=\alpha(n)\;\;\;\;\;\;\forall n\in M.
    \end{equation}
    Indeed, if \eqref{eqclaim1232} holds true then
    $$\begin{aligned}|\alpha(n)|\stackrel{\eqref{eqclaim1232}}{=}&|p(n\in[x,m_T)_T)-p(n\in[y,m_T)_T)|\\\leq& p(n\in[x,m_T)_T)+p(n\in[y,m_T)_T)\\=&p(T\;:\;n\in V([x,y]_T)\setminus\{m_T\})\stackrel{\text{def}}{=}\alpha_p(n).\end{aligned}$$
    We prove equation \eqref{eqclaim1232}. If $n> x$ then, since $y<x<n$, both sides of equation \eqref{eqclaim1232} are equal to 0 (see Lemma \ref{greatalpha}). Finally, if $n\leq x$ we prove equation \eqref{eqclaim1232} by induction in $n$. For the first step of the induction ($n=x$) it is clear that both sides of equality \eqref{eqclaim1232} are equal to 1 by Lemma \ref{greatalpha}. Now, for the inductive step we consider $n<x$ and divide the proof of \eqref{eqclaim1232} into 3 different cases assuming as induction hypothesis $(IH)$ that $p(m\in[x,m_T)_T)-p(m\in[y,m_T)_T)=\alpha(m)$ for $m>n$.

    \textbf{Case 1,} $y<n<x$. In this case, since $m_T\leq y$ we get that $n>m_T$ and hence
    $$\begin{aligned}p(n\in[x,m_T)_T)-p(n\in[y,m_T)_T)=&p(n\in[x,m_T)_T)\\=&\sum_{m>n}p(\{m,n\}\in[x,m_T)_T)\\\stackrel{n>m_T}{=}&\sum_{m>n}p(\{m,n\}\in[x,m_T]_T)\\\stackrel{\eqref{thhyp2}}{=}&\sum_{m>n}\lambda_{m,n}p(m\in[x,m_T)_T)\\\stackrel{(*)}{=}&\sum_{m>n}\lambda_{m,n}\big(p(m\in[x,m_T)_T)-p(m\in[y,m_T)_T)\big)\\\stackrel{(IH)}{=}&\sum_{m>n}\lambda_{m,n}\alpha(m)=\alpha(n).\end{aligned}$$
    It is worth mentioning that we are using the fact that $[y,m_T)_T=\emptyset$ when $y=m_T$ and hence $p(m\in[y,m_T)_T)=0$ in $(*)$.

    \textbf{Case 2,} $n=y$. In this case, $p(y\in[x,m_T)_T)=0$ and $p(y\in[y,m_T)_T)=p(y>m_T)$. Therefore,
    $$\begin{aligned}p(y\in[x,m_T)_T)-p(y\in[y,m_T)_T)=&-p(y>m_T)=-(1-p(y=m_T))\\=&(\delta_x-\delta_y)(y)+p(y=m_T)\\\stackrel{\text{Th }\ref{theobasisvect}}{=}&\alpha(n)+p(y=m_T)-\sum_{m>y}\lambda_{m,y}\alpha(m).\end{aligned}$$
    Then, it is enough to show that $\sum_{m>y}\lambda_{m,y}\alpha(m)=p(y=m_T)$ which is true since
    $$\begin{aligned}
        \sum_{m>y}\lambda_{m,y}\alpha(m)\stackrel{(IH)}{=}&\sum_{m>y}\lambda_{m,y}\big(p(m\in[x,m_T)_T)-p(m\in[y,m_T)_T)\big)\\=&\sum_{m>y}\lambda_{m,y}p(m\in[x,m_T)_T)\\\stackrel{\eqref{thhyp2}}{=}&\sum_{m>y}p(\{m,y\}\in[x,m_T]_T)\\=&p(y\in[x,m_T]_T)=p(y=m_T).
    \end{aligned}$$

    \textbf{Case 3,} $n<y$. In this case, it is clear that 
    \begin{equation}\label{equalpro2}
        \sum_{m>n}\sum_{\substack{T\in\mathcal{T}(F)\\\{m,n\}\in[x,m_T]_T\\n=m_T}}p(T)=\sum_{\substack{T\in\mathcal{T}(F)\\n=m_T}}p(T)=\sum_{m>n}\sum_{\substack{T\in\mathcal{T}(F)\\\{m,n\}\in[y,m_T]_T\\n=m_T}}p(T).
    \end{equation}
    Therefore,
    \begin{equation}\label{eqclaim12342}\begin{aligned}
        \sum_{m>n}&p(\{m,n\}\in [x,m_T]_T)-p(\{m,n\}\in [y,m_T]_T)\\=&\sum_{m>n}\sum_{\substack{T\in\mathcal{T}(F)\\\{m,n\}\in[x,m_T]_T}}p(T)-\sum_{m>n}\sum_{\substack{T\in\mathcal{T}(F)\\\{m,n\}\in[y,m_T]_T}}p(T)\\=&\sum_{m>n}\sum_{\substack{T\in\mathcal{T}(F)\\\{m,n\}\in[x,m_T)_T}}p(T)-\sum_{m>n}\sum_{\substack{T\in\mathcal{T}(F)\\\{m,n\}\in[y,m_T)_T}}p(T)\\&+\sum_{m>n}\sum_{\substack{T\in\mathcal{T}(F)\\\{m,n\}\in[x,m_T]_T\\n=m_T}}p(T)-\sum_{m>n}\sum_{\substack{T\in\mathcal{T}(F)\\\{m,n\}\in[y,m_T]_T\\n=m_T}}p(T)\\\stackrel{\eqref{equalpro2}}{=}&\sum_{m>n}\sum_{\substack{T\in\mathcal{T}(F)\\\{m,n\}\in[x,m_T)_T}}p(T)-\sum_{m>n}\sum_{\substack{T\in\mathcal{T}(F)\\\{m,n\}\in[y,m_T)_T}}p(T)\\=&\sum_{m>n}p(\{m,n\}\in[x,m_T)_T)-p(\{m,n\}\in [y,m_T)_T).
    \end{aligned}\end{equation}
    Finally, taking \eqref{eqclaim12342} into account, we are done with this case since
    $$\begin{aligned}\alpha(n)\stackrel{\text{Th }\ref{theobasisvect}}{=}&\sum_{m>n}\alpha(m)\lambda_{m,n}\stackrel{(IH)}{=}\sum_{m>n}\big(p(m\in[x,m_T)_T)-p(m\in[y,m_T)_T)\big)\lambda_{m,n}\\\stackrel{\eqref{thhyp2}}{=}&\sum_{m>n}p(\{m,n\}\in [x,m_T]_T)-p(\{m,n\}\in [y,m_T]_T)\\\stackrel{\eqref{eqclaim12342}}{=}&\sum_{m>n}p(\{m,n\}\in [x,m_T)_T)-p(\{m,n\}\in [y,m_T)_T)\\=&p(n\in[x,m_T)_T)-p(n\in[y,m_T)_T).\end{aligned}$$
\end{proof}

Our goal now is to prove Theorem \ref{th2}. Therefore, we fix a pair $\{x,y\}\in M^{[2]}$ and a basis $(b_n)\in B(M)$ for the rest of this subsection. We aim to define a probability $p\in[\pi_{(b_n)}]_{x,y}$ satisfying
$$\sum_{n=1}^N|b^*_n(\delta_x-\delta_y)|\|b_n\|=\mathbb{E}_p\big(d_T(x,y)\big).$$
This suitable probability will be called \textit{effective charge probability} associated to $(b_n)$ and $\{x,y\}$.
We consider again $F\in \Sigma(M)$ and $(\lambda_{n,i})_{n,i}$ as in Proposition \ref{propbasisvect}, identify $M$ with $\{0,\dots,N\}$ by this mapping $F$ and assume without loss of generality that $x>y$ under the latter identification. Let us now introduce new notation.

We denote $[x,y]=\bigcup_{T\in\mathcal{T}(F)}\{[x,y]_T\}$. Given $G$ a subgraph of $G_M$ we denote the source of $G$ as $s_G=\max V(G)$ and the terminal of $G$ as $t_G=\min V(G)$. If $T\in\mathcal{T}(F)$ and $S$ is any connected subgraph of $T$ containing $x$ and $y$, we denote the meeting point $m_S:=m_T(x,y)$ which only depends on the path $[x,y]_T$ and not on the rest of the tree $T$. Also, we denote the family of chains starting from $z\in M$ by $\overrightarrow{[z]}=\bigcup_{t\leq z}\overrightarrow{[z,t]}$ and the family of chains starting from either $x$ or $y$ by $\mathcal{C}=\overrightarrow{[x]}\cup\overrightarrow{[y]}$. It is important to mention that, even though it is not explicit in the notation, all these notions depend on $F$. Given any chain $C\in \mathcal{C}$, we denote $C_*$ the subchain of $C$ whose vertices are $V(C)\setminus\{t_C\}$, that is, $C_*$ is the chain $C$ without its terminal point. Finally, we will denote for simplicity $\mu=\delta_x-\delta_y$ throughout the rest of this subsection.

We proceed to the definition of the effective charge probability distribution associated to $(b_n)$ and $\{x,y\}$. We consider the \textit{effective charge overflow} function $\alpha:M\to\mathbb{R}$ as $\alpha=\alpha(\mu)$ given in \eqref{defalp}. It is worth mentioning that, by Theorem \ref{theobasisvect}, we know that $\alpha(n)=b^*_n(\mu)$ for every $n\in M$.

We are going to assign to every $n\in M$ a \textit{winning source} $w(n)\in\{x,y\}$ and a \textit{losing source} $l(n)\in\{x,y\}$ as
$$w(n)=\begin{cases}
    x,\;\;&\text{ if }\alpha(n)>0,\\y,&\text{ if }\alpha(n)\leq 0.
\end{cases}\;\;\;\;\;\; l(n)=\begin{cases}
    y,\;\;&\text{ if }\alpha(n)> 0,\\x,&\text{ if }\alpha(n)\leq 0.
\end{cases}$$
Notice that $\{w(n),l(n)\}=\{x,y\}$ for every $n\in M$. Now, we define the \textit{effective cancelled charge} function $\beta:M\to[-1,1]$ given by
$$\beta(n)=\sum_{\substack{m>n\\w(m)\neq w(n)}}\lambda_{m,n}\alpha(j)=\sum_{\substack{m>n\\w(m)= l(n)}}\lambda_{m,n}\alpha(m).$$
It is important to note that, in case it is different from $0$, $\beta(n)$ has the opposite sign to $\alpha(n)$.

We are then ready to extend the function $\alpha$ to the set of chains $\mathcal{C}$ starting from $x$ or $y$. We extend it inductively in $t_C$ for $C\in\mathcal{C}$ with $|V(C)|\geq2$. Let us then assume that $\alpha(C)$ has been defined for every $C\in\mathcal{C}$ with $t_C>n$. Now, if $C\in\mathcal{C}$ with $|V(C)|\geq2$ such that $t_C=n$ then by induction hypothesis we have defined $\alpha(C_*)$ and thus we put
$$\alpha(C)=\begin{cases}
    \alpha(C_*)\lambda_{t_{C_*},t_C}+\gamma(C)\beta(t_C)\;\;,&\text{ if }w(t_C)=w(t_{C_*}),\\
    0&\text{ otherwise,}
\end{cases}$$
where
$$\gamma(C)=\frac{\alpha(C_*)\lambda_{t_{C_*},t_C}}{\sum_{C'\in\overrightarrow{[s_{C},t_{C}]}}\alpha(C'_*)\lambda_{t_{C'_*},t_{C'}}}\;\;\;\;\;\Big(\text{agreeing }\frac{0}{0}=0\Big).$$

We also extend $\beta$ to every $C\in\mathcal{C}$ with $|V(C)|\geq2$ as
$$\beta(C)=\alpha(C_*)\lambda_{t_{C_*},t_C}-\alpha(C).$$

\begin{lemma}\label{proplem}
    The following properties are satisfied:
    \begin{enumerate}
        \item \label{p0} For $C\in\mathcal{C}$, if $\alpha(C)\neq0$ then $s_C=w(n)$ for every $n\in C$.
        \item \label{p1} For $C\in\mathcal{C}$, if $\beta(C)\neq0$ then $s_C=w(n)$ for every $n\in C\setminus\{t_C\}$.
        \item \label{p7} For every $n>y$ we have $\beta(n)=0$.        
        \item \label{p3} For every $n\in M$,
        $$\sum_{\substack{C\in\mathcal{C}\\t_C=n}}\alpha(C)=\sum_{\substack{C\in\overrightarrow{[w(n),n]}}}\alpha(C)=\alpha(n).$$
        \item \label{p35} For $n\in M$, if $n<l(n)$ then
        $$\sum_{C\in\overrightarrow{[l(n),n]}}\beta(C)=\beta(n),$$
        and if $n<w(n)$ then,
        $$\sum_{C\in\overrightarrow{[w(n),n]}}\beta(C)=-\beta(n).$$
        \item \label{p2} For every $C\in\mathcal{C}$,
        $$\gamma(C)\;,\;\alpha(C)\mu(s_C)\geq0.$$
        Moreover, if $|V(C)|\geq 2$ then $$\beta(C)\mu(s_C)=|\alpha(C_*)|\lambda_{t_{C_*},t_C}-|\alpha(C)|\geq0.$$
        \item \label{p4} For every $C\in\mathcal{C}$, $$|\beta(C)|=\gamma(C)|\beta(t_C)|.$$
        \item \label{p5} For every $C\in\mathcal{C}$,
        $$\sum_{\substack{C'\in\mathcal{C}\\C\subsetneq C'}}\beta(C')=\alpha(C).$$
        \item \label{p6} For every $n,i\in M$ with $n>i$,
        $$\sum_{\substack{C\in\mathcal{C}\\n\in C\setminus\{t_C\}}}\beta(C)=\alpha(n)\;\;\;\text{ and }\;\;\;\sum_{\substack{C\in\mathcal{C}\\\{n,i\}\in E(C)}}\beta(C)=\alpha(n)\lambda_{n,i}.$$
        \end{enumerate}
\end{lemma}
\begin{proof} We will prove them all in order:

    \textbf{Proof of \eqref{p0}:} It is clear from the definition of $\alpha(C)$ that $\alpha(C)\neq0$ implies both that $\alpha(C_*)\neq0$ and $w(t_C)=w(t_{C_*})$. Then, a straightforward induction makes the trick.

    \textbf{Proof of \eqref{p1}:} This statement is empty if $|V(C)|=1$ so let us assume that $|V(C)|\geq2$ and hence $\beta(C)=\alpha(C_*)\lambda_{t_{C_*},t_C}-\alpha(C)$. If there is $n\in C_*$ such that $w(n)\neq s_C$ then by \eqref{p0} we know that $\alpha(C)=\alpha(C_*)=0$ and hence by definition $\beta(C)=0$. 

     \textbf{Proof of \eqref{p7}:} 
    We show that $\beta(n)=0$ whenever $y<n$. From Lemma \ref{greatalpha} and the definition of $w(n)$ we get that $w(n)=x$ whenever $\alpha(n)\neq0$. Therefore, for every $m>n$ we get that $w(m)=w(n)=x$ and thus
    $$\beta(n)=\sum_{\substack{m>n\\w(m)\neq w(n)}}\lambda_{m,n}\alpha(m)=0.$$

    \textbf{Proof of \eqref{p3}:} Let us first show that for $n\in M$,
    \begin{equation}\label{alpclaim0}
        w(n)=n\;\;\;\Longleftrightarrow\;\;\;n\in\{x,y\}
    \end{equation}
    and therefore
    \begin{equation}\label{alpclaim}
        \alpha(n)=0\;\;\;\text{ whenever }\;\;\;w(n)<n.
    \end{equation}
    By definition we know that $w(n)\in\{x,y\}$ so we just have to prove that $w(x)=x$ and $w(y)=y$ in order to show equation \eqref{alpclaim0}. Clearly, $w(x)=x$ since $\alpha(x)=1>0$ by Lemma \ref{greatalpha}. Also, by Lemma \ref{basicmu} we have that $\alpha(y)\leq0$ and hence $w(y)=y$. Let us prove equation \eqref{alpclaim}. If $w(n)<n$ then by \eqref{alpclaim0} it follows that $n\notin\{x,y\}$. If $w(n)=x$ then $n>x$ and again by Lemma \ref{greatalpha} we get that $\alpha(n)=0$. Let us then assume that $w(n)=y$ and therefore $n>y$. By definition of $w(n)$ we get that $\alpha(n)\leq0$ and by Lemma \ref{greatalpha}, considering that $n>y$, we get that $\alpha(n)\geq0$. Now that \eqref{alpclaim0} and \eqref{alpclaim} have been proven, we are ready to prove Property \ref{p3}.
    
    It is immediate from \eqref{p0} that for every $n\in M$,
    $$\sum_{\substack{C\in\mathcal{C}\\t_C=n}}\alpha(C)=\sum_{\substack{C\in\overrightarrow{[w(n),n]}}}\alpha(C).$$
    Hence, it is enough to show that for $n\in M$,
    \begin{equation}\label{p3b}
        \sum_{\substack{C\in\overrightarrow{[w(n),n]}}}\alpha(C)=\alpha(n).
    \end{equation}
    If $n>x$ it follows from Lemma \ref{greatalpha} and the fact that $y<x<n$ that both sides of \eqref{p3b} are equal to 0. Also, if $n=x$ then by \eqref{alpclaim0} we have $\overrightarrow{[w(n),n]}=\{x\}$ and hence equation \eqref{p3b} is immediate. Analogously, if $n=y$ then, by \eqref{alpclaim0}, equation \eqref{p3b} is immediate from the fact that $\overrightarrow{[w(n),n]}=\{y\}$. Also, from \eqref{alpclaim} we get that equation \eqref{p3b} holds whenever $w(n)<n$. The only remaining case is when $n<x$ with $w(n)>n$.
    
    Let us show equation \eqref{p3b} in this case by induction in $n$. For the first step of the induction we take $n=N$ in which case \eqref{p3b} is clearly true by the latter considerations. Now, for the inductive step we assume as induction hypothesis that
    \begin{equation}\label{eqalpha}
        \alpha(m)=\sum_{\substack{C\in\overrightarrow{[w(m),m]}}}\alpha(C)\;\;\;\;\;\;\;\forall m>n.
    \end{equation}
    As we mentioned, we stick to the case when $n<x$ and $n<w(n)$. We define for every $m=n+1,\dots,w(n)$ and every chain $C\in\overrightarrow{[w(m),m]}$ the chain $C_n\in\overrightarrow{[w(m),n]}$ with $V(C_n)=V(C)\cup\{n\}$ and $E(C_n)=E(C)\cup\{\{m,n\}\}$. Then, it is clear that $C_n^*=C$ and therefore,
    $$\begin{aligned}
    \sum_{\substack{C\in\overrightarrow{[w(n),n]}}}\alpha(C)\stackrel{\eqref{p0}}{=}&\sum_{\substack{m=n+1\\w(m)=w(n)}}^{w(n)}\sum_{C\in\overrightarrow{[w(m),m]}}\alpha(C_n)\\=&\sum_{\substack{m>n\\w(m)=w(n)}}\sum_{C\in\overrightarrow{[w(m),m]}}\alpha(C)\lambda_{m,n}+\gamma(C_n)\beta(n)\\\stackrel{\eqref{p0}}{=}&\bigg(\sum_{\substack{m>n\\w(m)=w(n)}}\sum_{C\in\overrightarrow{[w(m),m]}}\alpha(C)\lambda_{m,n}\bigg)+\bigg(\sum_{C\in\overrightarrow{[w(n),n]}}\gamma(C)\bigg)\beta(n)\\\stackrel{\eqref{eqalpha}}{=}&\sum_{\substack{m>n\\w(m)=w(n)}}\lambda_{m,n}\alpha(m)+\beta(n)\\=&\sum_{\substack{m>n\\w(m)=w(n)}}\lambda_{m,n}\alpha(m)+\sum_{\substack{m>n\\w(m)\neq w(n)}}\lambda_{m,n}\alpha(m)\\=&\mu(n)+\sum_{m>n}\lambda_{m,n}\alpha(m)=\alpha(n).
    \end{aligned}$$
    
    \textbf{Proof of \eqref{p35}:}
    Let us first prove that $\sum_{C\in\overrightarrow{[l(n),n]}}\beta(C)=\beta(n)$ whenever $n<l(n)$. For every $C\in\overrightarrow{[l(n),n]}$, we know that $V(C)\geq 2$ and also $\alpha(C)=0$ by Property \eqref{p0}. Therefore,
    $$\begin{aligned}\sum_{C\in\overrightarrow{[l(n),n]}}\beta(C)=\sum_{C\in\overrightarrow{[l(n),n]}}\alpha(C_*)\lambda_{t_{C_*},t_C}\stackrel{\eqref{p0}}{=}&\sum_{\substack{m>n\\w(m)=l(n)}}\sum_{C\in\overrightarrow{[w(m),m]}}\lambda_{m,n}\alpha(C)\\\stackrel{\eqref{p3}}{=}&\sum_{\substack{m>n\\w(m)=l(n)}}\lambda_{m,n}\alpha(m)=\beta(n).\end{aligned}$$

    Similarly, if $n<w(n)$ then
    $$\begin{aligned}\sum_{C\in\overrightarrow{[w(n),n]}}\beta(C)=&\sum_{C\in\overrightarrow{[w(n),n]}}\alpha(C_*)\lambda_{t_{C_*},n}-\alpha(C)\\\stackrel{\eqref{p0}}{=}&\bigg(\sum_{\substack{m\geq n\\w(m)=w(n)}}\lambda_{m,n}\sum_{C\in\overrightarrow{[w(m),m]}}\alpha(C)\bigg)-\sum_{C\in\overrightarrow{[w(n),n]}}\alpha(C)\\\stackrel{\eqref{p0},\eqref{p3}}{=}&\sum_{\substack{m\geq n\\w(m)=w(n)}}\lambda_{m,n}\alpha(m)-\alpha(n)=-\sum_{\substack{m\geq n\\w(m)\neq w(n)}}\lambda_{m,n}\alpha(m)=-\beta(n).\end{aligned}$$

    \textbf{Proof of \eqref{p2}:} Let us first focus on the case when $|V(C)|=1$. In this case it is immediate to see that $\gamma(C),\mu(x)\alpha(x)=1>0$. Also, by Lemma \ref{basicmu} we get that $\mu(y),\alpha(y)\leq0$ and hence $\mu(y)\alpha(y)\geq0$. Now, we assume throughout the rest of the proof of \eqref{p2} that $|V(C)|\geq 2$. We start claiming that for every $n\in M\setminus\{y\}$,
    \begin{equation}\label{claimeq2}
        \sum_{\substack{m>n\\w(m)=w(n)}}\lambda_{m,n}|\alpha(m)|\geq\sum_{\substack{m>n\\w(m)\neq w(n)}}\lambda_{m,n}|\alpha(m)|.
    \end{equation}
    Indeed, from the definition of $w(n)$ we get 
    \begin{equation}\label{trivialeq}
        \mu(w(n))\alpha(n)=|\alpha(n)|\;\;\;\text{ and }\;\;\;\mu(l(n))\alpha(n)=-|\alpha(n)|.
    \end{equation}
    Therefore, for any $n\in M\setminus\{x,y\}$
    \begin{equation*}\label{claimeq3}\begin{aligned}0\leq |\alpha(n)|=&\mu(w(n))\alpha(n)=\sum_{\substack{m>n}}\lambda_{m,n}\mu(w(n))\alpha(m)\\=&\sum_{\substack{m>n\\w(m)=w(n)}}\lambda_{m,n}\mu(w(m))\alpha(m)+\sum_{\substack{m>n\\w(m)\neq w(n)}}\lambda_{m,n}\mu(l(m))\alpha(m)\\\stackrel{\eqref{trivialeq}}{=}&\sum_{\substack{m>n\\w(m)=w(n)}}\lambda_{m,n}|\alpha(m)|-\sum_{\substack{m>n\\w(m)\neq w(n)}}\lambda_{m,n}|\alpha(m)|,\end{aligned}\end{equation*}
    
     Since inequality \eqref{claimeq2} is vacuous in the case where $n=x$ we have finished the proof of the claimed inequality \eqref{claimeq2}. Now, Let us prove all the inequalities stated in Property \eqref{p2}. It is enough to show
    \begin{equation}\label{reduct}0\leq\mu(s_C)\alpha(C)\leq|\alpha(C_*)|\lambda_{t_{C_*},t_C}\;\;\;\;\;\;\;\;\forall C\in\mathcal{C}\;\;\;\text{ with }\;\;\;|V(C)|\geq2.\end{equation}

    Indeed, if \eqref{reduct} holds, then $$|\alpha(C)|=\mu(s_C)\alpha(C)\leq|\alpha(C_*)|\lambda_{t_{C_*},t_C},$$
    $$\mu(s_C)\beta(C)=\mu(s_C)\Big(\alpha(C_*)\lambda_{t_{C_*},t_C}-\alpha(C)\Big)=|\alpha(C_*)|\lambda_{t_{C_*},t_C}-|\alpha(C)|\geq 0,$$
    and
    $$\gamma(C)=\frac{\mu(s_C)\alpha(C_*)\lambda_{t_{C_*},t_C}}{\sum_{C'\in\overrightarrow{[s_{C},t_{C}]}}\mu(s_C)\alpha(C'_*)\lambda_{t_{C'_*},t_{C'}}}=\frac{|\alpha(C_*)|\lambda_{t_{C_*},t_C}}{\sum_{C'\in\overrightarrow{[s_{C},t_{C}]}}|\alpha(C'_*)|\lambda_{t_{C'_*},t_{C'}}}\geq0.$$

    Let us then prove \eqref{reduct} by induction on $t_C-s_C$. We may assume that $\alpha(C)\neq0$ and, thus, that $t_C\neq y$ because otherwise \eqref{reduct} is immediate. Thus, $\alpha(C_*)\neq0$ and, by Property \eqref{p0}, we know that $w(n)=s_C$ for every $n\in C$. For the first step of the induction we have $t_C=s_C-1$ and $w(t_C)=w(s_C)=s_C$. Then, it is clear that $V(C)=\{s_C,t_C\}$ and $\gamma(C)=1$. If $m>s_C$ and $w(t_C)=w(m)$ then $w(m)=s_C<m$ and, by \ref{alpclaim}, this implies that $\alpha(m)=0$. Hence,
    \begin{equation}\label{preeqq}
        |\alpha(s_C)|\lambda_{s_C,t_C}=\sum_{\substack{m\geq s_C\\w(m)=w(t_C)}}|\alpha(m)|\lambda_{m,t_C}\stackrel{s_C=t_C+1}{=}\sum_{\substack{m>t_C\\w(m)=w(t_C)}}|\alpha(m)|\lambda_{m,t_C}.
    \end{equation}
    Then,
    \begin{equation}\label{hohoho}
        \begin{aligned}
        \mu(s_C)\alpha(C)&=\mu(s_C)\big(\alpha(s_C)\lambda_{s_C,t_C}+\gamma(C)\beta(t_C)\big)\\&\stackrel{\gamma(C)=1}{=}|\alpha(s_C)|\lambda_{s_C,t_C}+\mu(s_C)\sum_{\substack{m>t_C\\ l(m)=w(t_C)}}\lambda_{m,t_C}\alpha(m)\\&\stackrel{s_C=w(t_C)}{=}|\alpha(s_C)|\lambda_{s_C,t_C}+\sum_{\substack{m>t_C\\ l(m)=w(t_C)}}\lambda_{m,t_C}\mu(l(m))\alpha(m)\\&\stackrel{\eqref{trivialeq}}{=}|\alpha(s_C)|\lambda_{s_C,t_C}-\sum_{\substack{m>t_C\\ l(m)=w(t_C)}}\lambda_{m,t_C}|\alpha(m)|.
    \end{aligned}
    \end{equation}
    This implies that $\mu(s_C)\alpha(C)\leq |\alpha(s_C)|\lambda_{s_C,t_C}$ so that it only remains to show that $\mu(s_C)\alpha(C)\geq0$. This holds since
    $$\begin{aligned}
        \mu(s_C)\alpha(C)&\stackrel{\eqref{hohoho}}{=}|\alpha(s_C)|\lambda_{s_C,t_C}-\sum_{\substack{m>t_C\\ l(m)=w(t_C)}}\lambda_{m,t_C}|\alpha(m)|\\&\stackrel{\eqref{preeqq}}{=}\sum_{\substack{m>t_C\\w(m)=w(t_C)}}|\alpha(m)|\lambda_{m,t_C}-\sum_{\substack{m>t_C\\ l(m)=w(t_C)}}\lambda_{m,t_C}|\alpha(m)|\stackrel{\eqref{claimeq2}}{\geq}0.
    \end{aligned}$$
    Now, for the inductive step we consider $s_C-t_C\geq2$ and assume as induction hypothesis $(IH)$ that \eqref{reduct} holds for every $C'\in\mathcal{C}$ with $s_{C'}-t_{C'}<s_C-t_C$. Then, taking into account that $s_C=w(t_{C_*})=w(t_C)$,
    $$\begin{aligned}
        \mu(s_C)\alpha(C)=&\mu(s_C)\alpha(C_*)\lambda_{t_{C_*},t_C}+\gamma(C)\mu(w(t_{C_*}))\beta(t_C)\\\stackrel{(IH)}{=}&|\alpha(C_*)|\lambda_{t_{C_*},t_C}\bigg(1+\frac{\beta(t_C)}{\sum_{C'\in\overrightarrow{[s_C,t_C]}}\alpha(C'_*)\lambda_{t_{C'_*},t_{C'}}}\bigg)\\\stackrel{s_C=w(t_C)}{=}&|\alpha(C_*)|\lambda_{t_{C_*},t_C}\Bigg(1+\frac{\sum_{\substack{m>t_C\\w(m)\neq w(t_C)}}\mu(l(m))\alpha(m)\lambda_{m,t_C}}{\sum_{C'\in\overrightarrow{[s_C,t_C]}}\mu(s_C)\alpha(C'_*)\lambda_{t_{C'_*},t_{C'}}}\Bigg)\\\stackrel{\eqref{trivialeq},(IH)}{=}&|\alpha(C_*)|\lambda_{t_{C_*},t_C}\Bigg(1-\underbrace{\frac{\sum_{\substack{m>t_C\\w(m)\neq w(t_C)}}|\alpha(m)|\lambda_{m,t_C}}{\sum_{C'\in\overrightarrow{[s_C,t_C]}}|\alpha(C'_*)|\lambda_{t_{C'_*},t_{C'}}}}_{A}\Bigg).
    \end{aligned}$$
    Therefore, \eqref{reduct} will be proven if we show that the quotient $A\in[0,1]$. For that purpose, notice first that
    \begin{equation}\label{claimeq4}\begin{aligned}\sum_{C'\in\overrightarrow{[s_C,t_C]}}|\alpha(C'_*)|\lambda_{t_{C'_*},t_{C'}}&\stackrel{\eqref{p0}}{=}\sum_{\substack{m>t_C\\w(m)=w(t_C)}}\lambda_{m,t_{C}}\sum_{C'\in\overrightarrow{[s_C,m]}}|\alpha(C')|\\&\stackrel{(IH),\eqref{p3}}{=}\sum_{\substack{m>t_C\\w(m)=w(t_C)}}\lambda_{m,t_{C}}\mu(s_C)\alpha(m)\\&\stackrel{w(t_C)=s_C,\eqref{trivialeq}}{=}\sum_{\substack{m>t_C\\w(m)=w(t_C)}}\lambda_{m,t_{C}}|\alpha(m)|.\end{aligned}\end{equation}

    This finishes the proof since we see that $A\in[0,1]$ thanks to \eqref{claimeq2} and \eqref{claimeq4}.
    
    \textbf{Proof of \eqref{p4}:} If $|V(C)|=1$ then it follows from the immediate fact that $\gamma(C)=1$. Let us consider $C\in\mathcal{C}$ with $|V(C)|\geq2$. If $w(t_{C_*})=w(t_C)$ then it is clear from the definition of $\alpha(C)$ that $\beta(C)=\alpha(C_*)\lambda_{t_{C_*},t_C}-\alpha(C)=-\gamma(C)\beta(t_C)$. Otherwise, we may assume that $w(t_{C_*})\neq w(t_C)$ which implies by Property \eqref{p0} that $\alpha(C)=0$ and thus $\beta(C)=\alpha(C_*)\lambda_{t_{C_*},t_C}$. We divide this case into two different subcases.
    
    \textbf{Subcase 1}: There exists $n\in C_*$ such that $w(n)\neq s_C$. In this case by Property \eqref{p1} we know that $\beta(C)=0$ and it is enough to show that $\gamma(C)=0$. This is immediate since  $\alpha(C_*)=0$ by Property \eqref{p0}.

    \textbf{Subcase 2}: $w(n)=s_C$ for every $n\in C_*$ and $l(t_C)= s_C$. In this case,
    $$\begin{aligned}\sum_{C'\in\overrightarrow{[l(t_C),t_C]}}\alpha(C_*')\lambda_{t_{C_*'}.t_C'}\stackrel{\eqref{p0}}{=}&\sum_{\substack{C'\in\overrightarrow{[l(t_C),t_C]}\\w(t_{C_*})=l(t_C)}}\alpha(C_*')\lambda_{t_{C_*'}.t_C'}\stackrel{\eqref{p0}}{=}\sum_{\substack{C'\in\overrightarrow{[l(t_C),t_C]}\\w(t_{C_*})=l(t_C)}}\beta(C')\\\stackrel{\eqref{p1}}{=}&\sum_{C'\in\overrightarrow{[l(t_C),t_C]}}\beta(C')\stackrel{\eqref{p35}}{=}\beta(t_C).\end{aligned}$$

    Then,
    $$\gamma(C)\beta(t_C)=\frac{\alpha(C_*)\lambda_{t_{C_*},t_C}}{\sum_{C'\in\overrightarrow{[l(t_C),t_C]}}\alpha(C_*')\lambda_{t_{C_*'}.t_C'}}\beta(t_C)=\frac{\beta(C)}{\beta(t_C)}\beta(t_C)=\beta(C).$$

    \textbf{Proof of \eqref{p5}:} By induction in $t_C$. If $t_C=0$ then it holds since $\alpha(0)=0$ by Lemma \ref{basicmu}. Now, for the inductive step $t_C\geq1$ and we assume as induction hypothesis $(IH)$ that, for every $C'\in\mathcal{C}$ with $t_{C'}<t_C$, property \eqref{p5} holds. As in the proof of property \eqref{p3}, for every $i<t_C$ a chain $C_i$ whose vertices are $V(C_i)=V(C)\cup\{i\}$ and whose edges are $E(C_i)=E(C)\cup\{\{t_C,i\}\}$. Then, ${C_i}_*=C$ and thus,
    $$\begin{aligned}\sum_{\substack{C'\in\mathcal{C}\\C\subsetneq C'}}\beta(C')=&\sum_{i<t_C}\sum_{\substack{C'\in\mathcal{C}\\C_i\subset C'}}\beta(C')=\sum_{i<t_C}\beta(C_i)+\sum_{\substack{C'\in\mathcal{C}\\C_i\subsetneq C'}}\beta(C')\\\stackrel{(IH)}{=}&\sum_{i<t_C}\beta(C_i)+\alpha(C_i)\stackrel{\text{def}}{=}\sum_{i<t_C}\alpha(C)\lambda_{t_C,i}=\alpha(C).\end{aligned}$$

    \textbf{Proof of \eqref{p6}:}
    Firstly, if $n\in M$,
    $$\sum_{\substack{C\in\mathcal{C}\\n\in C\setminus\{t_C\}}}\beta(C)=\sum_{\substack{C'\in\mathcal{C}\\t_{C'}=n}}\sum_{\substack{C\in\mathcal{C}\\C'\subsetneq C}}\beta(C)\stackrel{\eqref{p5}}{=}\sum_{\substack{C'\in\mathcal{C}\\t_{C'}=n}}\alpha(C')\stackrel{\eqref{p3}}{=}\alpha(n).$$

    Secondly, if $i\in M$ with $n>i$ we define the chain $C_i$ for every $C\in\mathcal{C}$ as in the proof of \eqref{p5}. Then,
    $$\begin{aligned}\sum_{\substack{C\in\mathcal{C}\\\{n,i\}\in E(C)}}\beta(C)=&\sum_{\substack{C\in\mathcal{C}\\t_C=n}}\bigg(\beta(C_i)+\sum_{\substack{C'\in\mathcal{C}\\C_i\subsetneq C'}}\beta(C')\bigg)\\\stackrel{\eqref{p5}}{=}&\sum_{\substack{C\in\mathcal{C}\\t_C=n}}\beta(C_i)+\alpha(C_i)\stackrel{{C_i}_*=C}{=}\sum_{\substack{C\in\mathcal{C}\\t_C=n}}\alpha(C)\lambda_{n,i}\stackrel{\eqref{p3}}{=}\alpha(n)\lambda_{n,i}.\end{aligned}$$
    
\end{proof}

We are finally ready to define the effective charge probability distribution $p$. If $T\in\mathcal{T}(F)$ we define
\begin{equation}\label{defp}
    p(T)=\frac{\big|\beta([x,m_T]_T)\beta([y,m_T]_T)\big|}{|\beta(m_T)|}\cdot\prod_{e\in E(T)\setminus E([x,y]_T)}\lambda_e.
\end{equation}

Before getting into the proof of the next lemma, we need to introduce a new concept that will come in handy: given $C\in\mathcal{C}$, we say that $C$ is an \textit{effective chain} if $s_C=w(n)$ for every $n\in C\setminus\{t_C\}$. We know from Lemma \ref{proplem}, Property \eqref{p1} that if $\beta(C)\neq0$ then $C$ must be an effective chain. Given $s,t\in M$ with $s>t$, we denote the set of effective chains from $s$ to $t$ by $\overrightarrow{[s,t]}_{ef}$. That is,
$$\overrightarrow{[s,t]}_{ef}=\big\{C\in\overrightarrow{[s,t]}\;:\;s=w(n)\;\;\forall n\in C\setminus\{t\}\big\}.$$
Clearly,
\begin{equation}\label{inter}
    \overrightarrow{[w(n),n]}_{ef}\cap \overrightarrow{[l(n),n]}_{ef}=\emptyset\;\;\;\forall n\in M.
\end{equation}

\begin{lemma}\label{chargedef}
    The function $p:\mathcal{T}(F)\to[0,\infty)$ defined as in \eqref{defp} is in $[\pi_{(b_n)}]_{x,y}$ and for every $n\in M$,
    \begin{equation}\label{keyeq}
        p\big(n\in[w(n),m_T)_T\big)=|\alpha(n)|\;\;\;\;\text{ and }\;\;\;\;p\big(n\in[l(n),m_T)_T\big)=0.
    \end{equation}
\end{lemma}

\begin{proof} 
    In order to show that $p\in[\pi_{(b_n)}]_{x,y}$ we have to show that $p$ is a $(x,y)$-independent probability distribution satisfying that
    \begin{equation}\label{quest1}
        p(\{n,i\}\in T)=\pi_{(b_n)}(\{n,i\}\in T)\stackrel{\eqref{prodeq3}}{=\lambda_{n,i}}
    \end{equation}
    for every $n,i\in M$ with $n>i$. Let us first show that $p\in\mathcal{P}(\mathcal{T}(F))$. We just need to show that $p(\mathcal{T}(F))=1$.
    \begin{equation}\label{byprodeq}\begin{aligned}p(\mathcal{T}(F))=&\sum_{T\in\mathcal{T}(F)}p(T)=\sum_{T\in\mathcal{T}(F)}\frac{\big|\beta([x,m_T]_T)\beta([y,m_T]_T)\big|}{|\beta(m_T)|}\cdot\prod_{e\in E(T)\setminus E([x,y]_T)}\lambda_e\\=&\sum_{P\in[x,y]}\sum_{\substack{T\in\mathcal{T}(F)\\ [x,y]_T=P}}\frac{\big|\beta([x,m_T]_T)\beta([y,m_T]_T)\big|}{|\beta(m_T)|}\cdot\prod_{e\in E(T)\setminus E([x,y]_T)}\lambda_e\\=&\sum_{P\in[x,y]}\frac{\big|\beta([x,m_P]_P)\beta([y,m_P]_P)\big|}{|\beta(m_P)|}\sum_{\substack{T\in\mathcal{T}(F)\\ [x,y]_T=P}}\prod_{e\in E(T)\setminus E([x,y]_T)}\lambda_e\\\stackrel{\eqref{prodeq2}}{=}&\sum_{P\in[x,y]}\frac{\big|\beta([x,m_P]_P)\beta([y,m_P]_P)\big|}{|\beta(m_P)|}\stackrel{\eqref{p1},\eqref{inter}}{=}\sum_{\substack{n\in M\\n\leq y}}\sum_{C'\in\overrightarrow{[y,n]}_{ef}}\sum_{\substack{C\in\overrightarrow{[x,n]}_{ef}}}\frac{|\beta(C)\beta(C')|}{|\beta(n)|}\\\stackrel{\eqref{p1}}{=}&\sum_{\substack{n\in M\\n\leq y}}\sum_{C'\in\overrightarrow{[y,n]}}\frac{|\beta(C')|}{|\beta(n)|}\sum_{\substack{C\in\overrightarrow{[x,n]}_{ef}}}|\beta(C)|\stackrel{\eqref{p4}}{=}\sum_{\substack{n\in M\\n\leq y}}\sum_{C'\in\overrightarrow{[y,n]}}\gamma(C')\sum_{C\in\overrightarrow{[x,n]}_{ef}}|\beta(C)|\\=&\sum_{\substack{n\in M\\n\leq y}}\sum_{C\in\overrightarrow{[x,n]}_{ef}}|\beta(C)|\stackrel{\eqref{p2}}{=}\sum_{\substack{n\in M\\n\leq y}}\bigg|\sum_{C\in\overrightarrow{[x,n]}_{ef}}\beta(C)\bigg|\stackrel{\eqref{p7}\eqref{p4}}{=}\sum_{\substack{n\in M\\n\leq x}}\bigg|\sum_{C\in\overrightarrow{[x,n]}_{ef}}\beta(C)\bigg|\\\stackrel{\eqref{p2}}{=}&\bigg|\sum_{\substack{n\in M\\n\leq x}}\sum_{C\in\overrightarrow{[x,n]}_{ef}}\beta(C)\bigg|\stackrel{\eqref{p1}}{=}\bigg|\sum_{C\in\overrightarrow{[x]}}\beta(C)\bigg|\stackrel{\eqref{p5}}{=}|\beta(x)+\alpha(x)|\stackrel{\text{Lemma }\ref{greatalpha},\eqref{p7}}{=}1.\end{aligned}\end{equation}
    We now show that $p(\{n,i\}\in T)=\lambda_{n,i}$ for $n,i\in M$ with $n>i$. We denote for simplicity
    $$p(P)=\frac{\big|\beta([x,m_P]_P)\beta([y,m_P]_P)\big|}{|\beta(m_P)|}\;\;\;\text{ for every }P\in[x,y].$$
    It is worth mentioning that as a byproduct of \eqref{byprodeq} we get that 
    \begin{equation}\label{pathprob}
        \sum_{P\in[x,y]}p(P)=1.
    \end{equation}
    We pick then $n,i\in M$ with $n>i$ and, proceeding in a similar way to the previous computation, we get
    $$\begin{aligned}p(\{n,i\}&\in T)=\sum_{\substack{T\in\mathcal{T}(F)\\\{n,i\}\in E(T)}}p(T)=\sum_{P\in [x,y]}p(P)\sum_{\substack{T\in\mathcal{T}(F)\\ \{n,i\}\in E(T)\\P=[x,y]_T}}\prod_{e\in E(T)\setminus E(P)}\lambda_e\\
    =&\underbrace{\sum_{\substack{P\in [x,y]\\n\in P\setminus \{m_P\}}}p(P)\sum_{\substack{T\in\mathcal{T}(F)\\ \{n,i\}\in E(T)\\P=[x,y]_T}}\prod_{e\in E(T)\setminus E(P)}\lambda_e}_A + \underbrace{\sum_{\substack{P\in [x,y]\\n\notin P\setminus\{m_P\}}}p(P)\sum_{\substack{T\in\mathcal{T}(F)\\ \{n,i\}\in E(T)\\P=[x,y]_T}}\prod_{e\in E(T)\setminus E(P)}\lambda_e}_B.\end{aligned}$$
    Again, following the argument provided in the initial computation of the theorem, we have
    \begin{equation}\label{disteq}\begin{aligned}
        A=&\sum_{\substack{P\in [x,y]\\\{n,i\}\in E(P)}}p(P)\sum_{\substack{T\in\mathcal{T}(F)\\P=[x,y]_T}}\prod_{e\in E(T)\setminus E(P)}\lambda_e\stackrel{\eqref{prodeq2}}{=}\sum_{\substack{P\in [x,y]\\\{n,i\}\in E(P)}}p(P)\\\stackrel{\eqref{p1},(*)}{=}&\sum_{\substack{m\in M\\m\leq i}}\sum_{\substack{C'\in\overrightarrow{[l(n),m]}}}\sum_{\substack{C\in\overrightarrow{[w(n),m]}_{ef}\\\{n,i\}\in E(C)}}\frac{|\beta(C)\beta(C')|}{|\beta(m)|}\\\stackrel{\eqref{p4}}{=}&\sum_{\substack{m\in M\\m\leq i}}\sum_{\substack{C'\in\overrightarrow{[l(n),m]}}}\gamma(C')\sum_{\substack{C\in\overrightarrow{[w(n),m]}_{ef}\\\{n,i\}\in E(C)}}|\beta(C)|\\\stackrel{\eqref{p1},\eqref{p2}}{=}&\bigg|\sum_{\substack{C\in\mathcal{C}\\\{n,i\}\in E(C)}}\beta(C)\bigg|\stackrel{\eqref{p6}}{=}\lambda_{n,i}|\alpha(n)|.
    \end{aligned}\end{equation}
    Equality $(*)$ holds because if $C'\in\overrightarrow{[l(n),m]}$ with $\{n,i\}\in E(C')$ then $n\in C'\setminus\{t_{C'}\}$ and hence $\beta(C')=0$ by Property \eqref{p1} of Lemma \ref{proplem}.
    
    Also, if $\{n,i\}\in E(T)$ but $n\notin V([x,y]_T)\setminus \{m_T\}$ for some $T\in\mathcal{T}(F)$ then $\{n,i\}\in E(T)\setminus E([x,y]_T)$. Moreover, in that case the graph $S_T$ given by $V(S_T)=V([x,y]_T)\cup\{n,i\}$ and $E(S_T)=E([x,y]_T)\cup\{\{n,i\}\}$ is a subgraph of $T$. Therefore,
    $$\begin{aligned}
        B=&\lambda_{n,i}\sum_{\substack{P\in [x,y]\\n\notin P\setminus\{m_P\}}}p(P)\sum_{\substack{T\in\mathcal{T}(F)\\ \{n,i\}\in E(T)\\P=[x,y]_T}}\prod_{\substack{e\in E(T)\setminus E(P)\\e\neq\{n,i\}}}\lambda_e\\=&\lambda_{n,i}\sum_{\substack{P\in [x,y]\\n\notin P\setminus\{m_P\}}}p(P)\sum_{\substack{T\in\mathcal{T}(F)\\ S_T\subset T}}\prod_{\substack{e\in E(T)\setminus E(S_T)}}\lambda_e\\\stackrel{\eqref{prodeq2}}{=}&\lambda_{n,i}\sum_{\substack{P\in [x,y]\\n\notin P\setminus\{m_P\}}}p(P)\\\stackrel{\eqref{pathprob}}{=}&\lambda_{n,i}\bigg( 1- \sum_{\substack{P\in [x,y]\\n\in P\setminus\{m_P\}}}p(P)\bigg).
    \end{aligned}$$
    It is then enough to show that $\sum_{\substack{P\in [x,y]\\n\in P\setminus\{m_P\}}}p(P)=|\alpha(n)|$. This follows from Property \eqref{p6} of Lemma \ref{proplem} since
    \begin{equation}\label{disteq2}
        \begin{aligned}\sum_{\substack{P\in [x,y]\\n\in P\setminus\{m_P\}}}p(P)\stackrel{m_P<n}{=}&\sum_{\substack{m\in M\\m<n}}\sum_{\substack{C'\in\overrightarrow{[l(n),m]}}}\sum_{\substack{C\in\overrightarrow{[w(n),m]}_{ef}\\n\in C}}\frac{|\beta(C)\beta(C')|}{|\beta(m)|}\\\stackrel{\eqref{p4}}{=}&\sum_{\substack{m\in M\\m<n}}\sum_{\substack{C'\in\overrightarrow{[l(n),m]}}}\gamma(C')\sum_{\substack{C\in\overrightarrow{[w(n),m]}_{ef}\\n\in C}}|\beta(C)|\\\stackrel{\eqref{p1},\eqref{p2}}{=}&\bigg|\sum_{\substack{C\in\mathcal{C}\\n\in C\setminus\{t_C\}}}\beta(C)\bigg|\stackrel{\eqref{p6}}{=}|\alpha(n)|.\end{aligned}
    \end{equation}
    This proves that $p$ is a probability distribution compatible with $\pi_{(b_n)}$. Also, by Property \eqref{p1} of Lemma \ref{proplem} we get that $p(n\in[l(n),m_T)_T)=0$ and therefore \eqref{disteq2} proves equation \eqref{keyeq}. Let us finally show that $p$ is $(x,y)$-independent. This is immediate from the latter computations. Indeed, for $n,i\in M$ with $n>i$,
    $$\begin{aligned}p(\{n,i\}\in [w(n),m_T]_T)\stackrel{\eqref{keyeq}}{=}&p(\{n,i\}\in [x,y]_T)\stackrel{\eqref{disteq}}{=}\lambda_{n,i}|\alpha(n)|=p(\{n,i\}\in T)|\alpha(n)|\\\stackrel{\eqref{keyeq}}{=}&p(\{n,i\}\in T)p(n\in[w(n),m_T)_T).\end{aligned}$$
    On the other hand, by equation \eqref{keyeq} we get that both $p(n\in[l(n),m_T)_T)$ and $p(\{n,i\}\in [l(n),m_T]_T)$ are null. Therefore, $p$ is $(x,y)$-independent and we are done.
\end{proof}

We are finally ready to prove Theorem \ref{th2}.

\begin{proof}[Proof of Theorem \ref{th2}]
    By Proposition \ref{propbasisvect} and Theorem \ref{theobasisvect}, we know it is enough to prove that
    \begin{equation}\label{step1}
        \mathbb{E}_{p}\big(d_T(x,y)\big)=\sum_{n=1}^{N}|\alpha(n)|\sum_{i=0}^{n-1}\lambda_{n,i}d(n,i).
    \end{equation}
    Since $p$ is $(x,y)$-independent, by Lemma \ref{chargedef}, for every $n,i\in M$ with $n>i$, 
    $$\begin{aligned}p(\{n,i\}\in[x,y]_T)=&p(\{n,i\}\in T)\big(p(n\in[w(n),m_T)_T)+p(n\in[l(n),m_T)_T)\big)\\\stackrel{\eqref{keyeq}}{=}&p(\{n,i\}\in T)|\alpha(n)|\stackrel{\eqref{prodeq3}}{=}\lambda_{n,i}|\alpha(n)|.\end{aligned}$$
    The last equality holds because $p$ is compatible with $\pi_{(b_n)}$.Finally, equation \eqref{step1} is now straightforward since 
    $$\begin{aligned}\mathbb{E}_p(d_T(x,y))=&\sum_{n=1}^{N}\sum_{i=0}^{n-1}p(\{n,i\}\in [x,y]_T)d(n,i)\\=&\sum_{n=1}^{N}\sum_{i=0}^{n-1}\lambda_{n,i}|\alpha(n)|d(n,i)=\sum_{n=1}^{N}|\alpha(n)|\sum_{i=0}^{n-1}\lambda_{n,i}d(n,i).\end{aligned}$$
\end{proof}

\subsection{Final results}\label{ss23}

As a consequence of Theorems \ref{th1} and \ref{th2}, we obtain the following corollary that precedes our main Theorem \ref{mainthth}.

\begin{corollary}\label{basiscor}
    For every stochastic basis $(b_n)\in B(M)$ we have
    $$d_1(b_n)=\max_{\{x,y\}\in M^{[2]}}\min_{p\in[\pi_{(b_n)}]_{x,y}}\mathbb{E}_p\bigg(\frac{d_T(x,y)}{d(x,y)}\bigg).$$
\end{corollary}
\begin{proof}
Let us fix a basis $(b_n)\in B(M)$. By Theorem \eqref{th1} we know that for every $\{x,y\}\in M^{[2]}$,
$$\sum_{n=1}^N\Big|b^*_n\Big(\frac{\delta_x-\delta_y}{d(x,y)}\Big)\Big|\|b_n\|\leq\min_{p\in[\pi_{(b_n)}]_{x,y}}\mathbb{E}_p\Big(\frac{d_T(x,y)}{d(x,y)}\Big).$$
Also, from Theorem \eqref{th2} it is clear that
$$\sum_{n=1}^N\Big|b^*_n\Big(\frac{\delta_x-\delta_y}{d(x,y)}\Big)\Big|\|b_n\|\geq\min_{p\in[\pi_{(b_n)}]_{x,y}}\mathbb{E}_p\Big(\frac{d_T(x,y)}{d(x,y)}\Big).$$
Now, taking equation \eqref{stodistbasis} into account we get that
$$d_1(b_n)=\max_{\{x,y\}\in M^{[2]}}\sum_{n=1}^N\Big|b^*_n\Big(\frac{\delta_x-\delta_y}{d(x,y)}\Big)\Big|\|b_n\|=\max_{\{x,y\}\in M^{[2]}}\min_{p\in[\pi_{(b_n)}]_{x,y}}\mathbb{E}_p\Big(\frac{d_T(x,y)}{d(x,y)}\Big).$$
\end{proof}

\begin{proof}[Proof of Theorem \ref{mainthth}]
If $(b_n)_{n=1}^N$ is a stochastic basis of $\mathcal{F(M)}$ and $(\rho_n)_{n=1}^N\subset\mathbb{R}^*$ it is immediate that $d_1(b_n)=d_1(\rho_nb_n)$. Therefore, we may assume without loss of generality that $(b_n)$ is normalised, that is, $(b_n)\in B(M)$. Equivalently,
\begin{equation}\label{eqeq1}
    sd_1(M)=\min_{(b_n)\in B(M)}d_1(b_n)\stackrel{\text{Cor }\ref{basiscor}}{=}\min_{(b_n)\in B(M)}\max_{\{x,y\}\in M^{[2]}}\min_{p\in[\pi_{(b_n)}]_{x,y}}\mathbb{E}_p\bigg(\frac{d_T(x,y)}{d(x,y)}\bigg).
\end{equation}

We now claim that for every $p_0\in\mathcal{P}=\bigcup_{F\in \Sigma(M)}\mathcal{P}(\mathcal{T}(F))$ there is $(b_n)\in B(M)$ such that $[p_0]_{x,y}=[\pi_{(b_n)}]_{x,y}$. Indeed, consider $F\in\Sigma(M)$ such that $p_0\in\mathcal{P}(\mathcal{T}(F))$ and set $\lambda_{n,i}=p_0(\{F(n),F(i)\}\in T)$ for every $n,i\in\{0,\dots,N\}$ with $n>i$. Now, the basis $(b_n)_{n=1}^N$ given by
$$b_n=\delta_{F(n)}-\sum_{i<n}\lambda_{n,i}\delta_{F(i)}\;\;\;\;\text{ for }n=1,\dots,N$$
makes the trick since $\pi_{(b_n)}(e\in T)=\lambda_{e}=p_0(e\in T)$ for every $e\in M^{[2]}$. Therefore,
$$\begin{aligned}
    sd_1(M)=&\min_{(b_n)\in B(M)}\max_{\{x,y\}\in M^{[2]}}\min_{p\in[\pi_{(b_n)}]_{x,y}}\mathbb{E}_p\bigg(\frac{d_T(x,y)}{d(x,y)}\bigg)\\=&\min_{p_0\in \mathcal{P}}\max_{\{x,y\}\in M^{[2]}}\min_{p\in[p_0]_{x,y}}\mathbb{E}_p\bigg(\frac{d_T(x,y)}{d(x,y)}\bigg).
\end{aligned}$$
\end{proof}

It is tempting to think that the family $[\pi_{(b_n)}]_{x,y}$ has $\pi_{(b_n)}$ as its unique element. We present here a very simple situation were this does not happen. Moreover, the effective charge probability and the product probability are different.

\begin{example}\label{examp1}
Let $N=4$ and $M=\{0,\dots,4\}$ with the geodesic graph distance given by the graph $G$ shown in the left side of Figure \ref{5point} with an arbitrary weight function $d:E(G)\to\mathbb{R}^+$. We consider $F=Id_{M}$ and the basis $(b_n)_{n=1}^4$ given by
$$b_1=\delta_1,\;\;\;\;\;\;b_2=\delta_2,\;\;\;\;\;\;b_3=\delta_3-\frac{1}{2}(\delta_1+\delta_2),\;\;\;\;\;\;b_4=\delta_4-\frac{1}{2}(\delta_1+\delta_2).$$
See the right side of Figure \ref{5point} for a geometric representation of the basis $(b_n)$. Now, if we consider $x=4$ and $y=3$, it is straighforward to see that $\delta_x-\delta_y=b_4-b_3$ and hence
$$\begin{aligned}\sum_{n=1}^4|b_n^*(\delta_x-\delta_y)|\|b_n\|=&\|b_4\|+\|b_3\|=\frac{1}{2}(d(4,1)+d(4,2))+\frac{1}{2}(d(3,1)+d(3,2))\\=&\frac{1}{2}(d(4,2)+d(2,3))+\frac{1}{2}(d(4,1)+d(1,3))\\=&p(T_1)d_{T_1}(x,y)+p(T_2)d_{T_2}(x,y).\end{aligned}$$
Where $T_1$, $T_2$ and $p$ are like in the right side of Figure \ref{5tree}. Clearly, $p$ is the effective charge probability associated to $(b_n)_{n=0}^4$ and $\{x,y\}$. On the other hand, the product probability $\pi$ associated to $(b_n)_{n=0}^4$ is shown in the left side of Figure \ref{5tree} and its expected distance $\mathbb{E}_{\pi}(d_T(x,y))$ is larger than $\mathbb{E}_{p}(d_T(x,y))$ as it is predicted in Theorem \ref{th1}.
\end{example}

\begin{figure}

\tikzset{every picture/.style={line width=0.75pt}} 

\begin{tikzpicture}[x=0.75pt,y=0.75pt,yscale=-1,xscale=1]

\draw  [fill={rgb, 255:red, 0; green, 0; blue, 0 }  ,fill opacity=1 ] (90,85) .. controls (90,82.24) and (92.24,80) .. (95,80) .. controls (97.76,80) and (100,82.24) .. (100,85) .. controls (100,87.76) and (97.76,90) .. (95,90) .. controls (92.24,90) and (90,87.76) .. (90,85) -- cycle ;
\draw  [fill={rgb, 255:red, 0; green, 0; blue, 0 }  ,fill opacity=1 ] (210,85) .. controls (210,82.24) and (212.24,80) .. (215,80) .. controls (217.76,80) and (220,82.24) .. (220,85) .. controls (220,87.76) and (217.76,90) .. (215,90) .. controls (212.24,90) and (210,87.76) .. (210,85) -- cycle ;
\draw  [fill={rgb, 255:red, 0; green, 0; blue, 0 }  ,fill opacity=1 ] (150,255) .. controls (150,252.24) and (152.24,250) .. (155,250) .. controls (157.76,250) and (160,252.24) .. (160,255) .. controls (160,257.76) and (157.76,260) .. (155,260) .. controls (152.24,260) and (150,257.76) .. (150,255) -- cycle ;
\draw  [fill={rgb, 255:red, 0; green, 0; blue, 0 }  ,fill opacity=1 ] (90,185) .. controls (90,182.24) and (92.24,180) .. (95,180) .. controls (97.76,180) and (100,182.24) .. (100,185) .. controls (100,187.76) and (97.76,190) .. (95,190) .. controls (92.24,190) and (90,187.76) .. (90,185) -- cycle ;
\draw  [fill={rgb, 255:red, 0; green, 0; blue, 0 }  ,fill opacity=1 ] (210,185) .. controls (210,182.24) and (212.24,180) .. (215,180) .. controls (217.76,180) and (220,182.24) .. (220,185) .. controls (220,187.76) and (217.76,190) .. (215,190) .. controls (212.24,190) and (210,187.76) .. (210,185) -- cycle ;
\draw  [fill={rgb, 255:red, 0; green, 0; blue, 0 }  ,fill opacity=1 ] (430,85) .. controls (430,82.24) and (432.24,80) .. (435,80) .. controls (437.76,80) and (440,82.24) .. (440,85) .. controls (440,87.76) and (437.76,90) .. (435,90) .. controls (432.24,90) and (430,87.76) .. (430,85) -- cycle ;
\draw  [fill={rgb, 255:red, 0; green, 0; blue, 0 }  ,fill opacity=1 ] (550,85) .. controls (550,82.24) and (552.24,80) .. (555,80) .. controls (557.76,80) and (560,82.24) .. (560,85) .. controls (560,87.76) and (557.76,90) .. (555,90) .. controls (552.24,90) and (550,87.76) .. (550,85) -- cycle ;
\draw  [fill={rgb, 255:red, 0; green, 0; blue, 0 }  ,fill opacity=1 ] (490,255) .. controls (490,252.24) and (492.24,250) .. (495,250) .. controls (497.76,250) and (500,252.24) .. (500,255) .. controls (500,257.76) and (497.76,260) .. (495,260) .. controls (492.24,260) and (490,257.76) .. (490,255) -- cycle ;
\draw  [fill={rgb, 255:red, 0; green, 0; blue, 0 }  ,fill opacity=1 ] (430,185) .. controls (430,182.24) and (432.24,180) .. (435,180) .. controls (437.76,180) and (440,182.24) .. (440,185) .. controls (440,187.76) and (437.76,190) .. (435,190) .. controls (432.24,190) and (430,187.76) .. (430,185) -- cycle ;
\draw  [fill={rgb, 255:red, 0; green, 0; blue, 0 }  ,fill opacity=1 ] (550,185) .. controls (550,182.24) and (552.24,180) .. (555,180) .. controls (557.76,180) and (560,182.24) .. (560,185) .. controls (560,187.76) and (557.76,190) .. (555,190) .. controls (552.24,190) and (550,187.76) .. (550,185) -- cycle ;
\draw  [dash pattern={on 4.5pt off 4.5pt}]  (330,40) -- (330,310) ;
\draw    (95,85) -- (95,185) ;
\draw    (215,85) -- (215,185) ;
\draw    (95,85) -- (215,85) ;
\draw    (95,185) -- (215,185) ;
\draw    (95,85) -- (215,185) ;
\draw    (95,185) -- (215,85) ;
\draw    (95,185) -- (155,255) ;
\draw    (215,185) -- (155,255) ;
\draw    (555,185) -- (506.81,243.42) ;
\draw [shift={(505.54,244.96)}, rotate = 309.52] [color={rgb, 255:red, 0; green, 0; blue, 0 }  ][line width=0.75]    (10.93,-3.29) .. controls (6.95,-1.4) and (3.31,-0.3) .. (0,0) .. controls (3.31,0.3) and (6.95,1.4) .. (10.93,3.29)   ;
\draw    (435,185) -- (484.23,241.45) ;
\draw [shift={(485.54,242.96)}, rotate = 228.91] [color={rgb, 255:red, 0; green, 0; blue, 0 }  ][line width=0.75]    (10.93,-3.29) .. controls (6.95,-1.4) and (3.31,-0.3) .. (0,0) .. controls (3.31,0.3) and (6.95,1.4) .. (10.93,3.29)   ;
\draw    (435,85) -- (435.04,169.46) ;
\draw [shift={(435.04,171.46)}, rotate = 269.97] [color={rgb, 255:red, 0; green, 0; blue, 0 }  ][line width=0.75]    (10.93,-3.29) .. controls (6.95,-1.4) and (3.31,-0.3) .. (0,0) .. controls (3.31,0.3) and (6.95,1.4) .. (10.93,3.29)   ;
\draw    (555,85) -- (555.04,169.46) ;
\draw [shift={(555.04,171.46)}, rotate = 269.97] [color={rgb, 255:red, 0; green, 0; blue, 0 }  ][line width=0.75]    (10.93,-3.29) .. controls (6.95,-1.4) and (3.31,-0.3) .. (0,0) .. controls (3.31,0.3) and (6.95,1.4) .. (10.93,3.29)   ;
\draw    (435,85) -- (541.5,173.68) ;
\draw [shift={(543.04,174.96)}, rotate = 219.78] [color={rgb, 255:red, 0; green, 0; blue, 0 }  ][line width=0.75]    (10.93,-3.29) .. controls (6.95,-1.4) and (3.31,-0.3) .. (0,0) .. controls (3.31,0.3) and (6.95,1.4) .. (10.93,3.29)   ;
\draw    (555,85) -- (448.07,174.68) ;
\draw [shift={(446.54,175.96)}, rotate = 320.01] [color={rgb, 255:red, 0; green, 0; blue, 0 }  ][line width=0.75]    (10.93,-3.29) .. controls (6.95,-1.4) and (3.31,-0.3) .. (0,0) .. controls (3.31,0.3) and (6.95,1.4) .. (10.93,3.29)   ;

\draw (167,252.4) node [anchor=north west][inner sep=0.75pt]    {$0$};
\draw (221,192.4) node [anchor=north west][inner sep=0.75pt]    {$1$};
\draw (81,192.4) node [anchor=north west][inner sep=0.75pt]    {$2$};
\draw (226.5,62.9) node [anchor=north west][inner sep=0.75pt]    {$3$};
\draw (77.5,60.9) node [anchor=north west][inner sep=0.75pt]    {$4$};
\draw (546.5,199.4) node [anchor=north west][inner sep=0.75pt]    {$b_{1}$};
\draw (432.5,199.4) node [anchor=north west][inner sep=0.75pt]    {$b_{2}$};
\draw (457.5,81.4) node [anchor=north west][inner sep=0.75pt]    {$b_{4}$};
\draw (519,80.4) node [anchor=north west][inner sep=0.75pt]    {$b_{3}$};
\draw (430,57.4) node [anchor=north west][inner sep=0.75pt]    {$x$};
\draw (550,56.4) node [anchor=north west][inner sep=0.75pt]    {$y$};
\draw (416,105.9) node [anchor=north west][inner sep=0.75pt]  [font=\scriptsize]  {$\frac{1}{2}$};
\draw (452,110.4) node [anchor=north west][inner sep=0.75pt]  [font=\scriptsize]  {$\frac{1}{2}$};
\draw (527.5,108.9) node [anchor=north west][inner sep=0.75pt]  [font=\scriptsize]  {$\frac{1}{2}$};
\draw (562.5,105.4) node [anchor=north west][inner sep=0.75pt]  [font=\scriptsize]  {$\frac{1}{2}$};

\end{tikzpicture}
    \caption{The graph $G$ and the basis $(b_n)$ of Example \ref{examp1}}.
    \label{5point}
\end{figure}
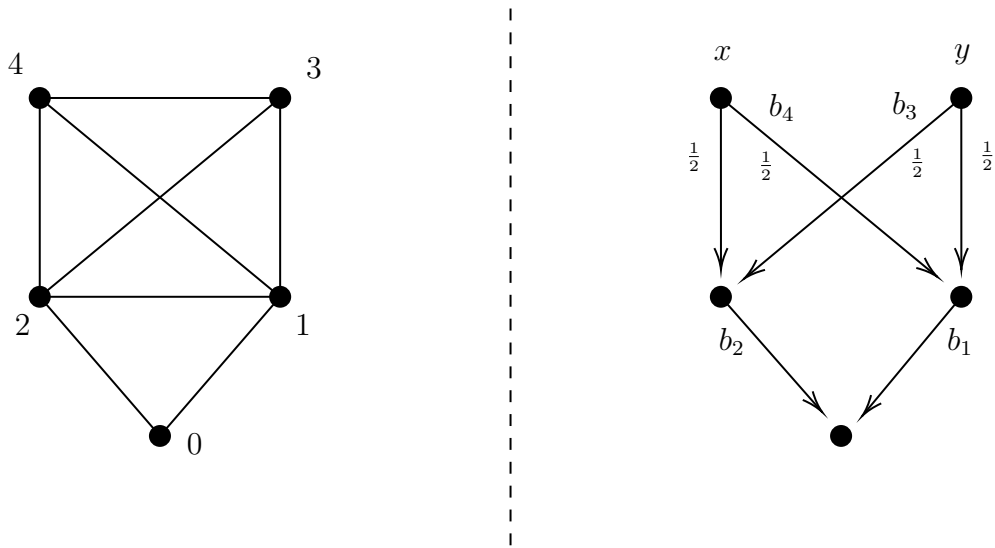

\begin{figure}

\tikzset{every picture/.style={line width=0.75pt}} 

\begin{tikzpicture}[x=0.75pt,y=0.75pt,yscale=-1,xscale=1]

\draw  [fill={rgb, 255:red, 0; green, 0; blue, 0 }  ,fill opacity=1 ] (34.81,57.63) .. controls (34.81,55.84) and (36.27,54.39) .. (38.07,54.39) .. controls (39.87,54.39) and (41.33,55.84) .. (41.33,57.63) .. controls (41.33,59.42) and (39.87,60.87) .. (38.07,60.87) .. controls (36.27,60.87) and (34.81,59.42) .. (34.81,57.63) -- cycle ;
\draw  [fill={rgb, 255:red, 0; green, 0; blue, 0 }  ,fill opacity=1 ] (113.1,57.63) .. controls (113.1,55.84) and (114.56,54.39) .. (116.36,54.39) .. controls (118.17,54.39) and (119.63,55.84) .. (119.63,57.63) .. controls (119.63,59.42) and (118.17,60.87) .. (116.36,60.87) .. controls (114.56,60.87) and (113.1,59.42) .. (113.1,57.63) -- cycle ;
\draw  [fill={rgb, 255:red, 0; green, 0; blue, 0 }  ,fill opacity=1 ] (73.95,167.78) .. controls (73.95,165.99) and (75.42,164.54) .. (77.22,164.54) .. controls (79.02,164.54) and (80.48,165.99) .. (80.48,167.78) .. controls (80.48,169.57) and (79.02,171.02) .. (77.22,171.02) .. controls (75.42,171.02) and (73.95,169.57) .. (73.95,167.78) -- cycle ;
\draw  [fill={rgb, 255:red, 0; green, 0; blue, 0 }  ,fill opacity=1 ] (34.81,122.43) .. controls (34.81,120.64) and (36.27,119.19) .. (38.07,119.19) .. controls (39.87,119.19) and (41.33,120.64) .. (41.33,122.43) .. controls (41.33,124.21) and (39.87,125.67) .. (38.07,125.67) .. controls (36.27,125.67) and (34.81,124.21) .. (34.81,122.43) -- cycle ;
\draw  [fill={rgb, 255:red, 0; green, 0; blue, 0 }  ,fill opacity=1 ] (113.1,122.43) .. controls (113.1,120.64) and (114.56,119.19) .. (116.36,119.19) .. controls (118.17,119.19) and (119.63,120.64) .. (119.63,122.43) .. controls (119.63,124.21) and (118.17,125.67) .. (116.36,125.67) .. controls (114.56,125.67) and (113.1,124.21) .. (113.1,122.43) -- cycle ;
\draw  [dash pattern={on 4.5pt off 4.5pt}]  (330.5,45.5) -- (330.5,315.5) ;
\draw    (38.07,57.63) -- (38.07,122.43) ;
\draw    (38.07,122.43) -- (116.36,57.63) ;
\draw    (38.07,122.43) -- (77.22,167.78) ;
\draw    (116.36,122.43) -- (77.22,167.78) ;
\draw  [fill={rgb, 255:red, 0; green, 0; blue, 0 }  ,fill opacity=1 ] (189.31,59.13) .. controls (189.31,57.34) and (190.77,55.89) .. (192.57,55.89) .. controls (194.37,55.89) and (195.83,57.34) .. (195.83,59.13) .. controls (195.83,60.92) and (194.37,62.37) .. (192.57,62.37) .. controls (190.77,62.37) and (189.31,60.92) .. (189.31,59.13) -- cycle ;
\draw  [fill={rgb, 255:red, 0; green, 0; blue, 0 }  ,fill opacity=1 ] (267.6,59.13) .. controls (267.6,57.34) and (269.06,55.89) .. (270.86,55.89) .. controls (272.67,55.89) and (274.13,57.34) .. (274.13,59.13) .. controls (274.13,60.92) and (272.67,62.37) .. (270.86,62.37) .. controls (269.06,62.37) and (267.6,60.92) .. (267.6,59.13) -- cycle ;
\draw  [fill={rgb, 255:red, 0; green, 0; blue, 0 }  ,fill opacity=1 ] (228.45,169.28) .. controls (228.45,167.49) and (229.92,166.04) .. (231.72,166.04) .. controls (233.52,166.04) and (234.98,167.49) .. (234.98,169.28) .. controls (234.98,171.07) and (233.52,172.52) .. (231.72,172.52) .. controls (229.92,172.52) and (228.45,171.07) .. (228.45,169.28) -- cycle ;
\draw  [fill={rgb, 255:red, 0; green, 0; blue, 0 }  ,fill opacity=1 ] (189.31,123.93) .. controls (189.31,122.14) and (190.77,120.69) .. (192.57,120.69) .. controls (194.37,120.69) and (195.83,122.14) .. (195.83,123.93) .. controls (195.83,125.71) and (194.37,127.17) .. (192.57,127.17) .. controls (190.77,127.17) and (189.31,125.71) .. (189.31,123.93) -- cycle ;
\draw  [fill={rgb, 255:red, 0; green, 0; blue, 0 }  ,fill opacity=1 ] (267.6,123.93) .. controls (267.6,122.14) and (269.06,120.69) .. (270.86,120.69) .. controls (272.67,120.69) and (274.13,122.14) .. (274.13,123.93) .. controls (274.13,125.71) and (272.67,127.17) .. (270.86,127.17) .. controls (269.06,127.17) and (267.6,125.71) .. (267.6,123.93) -- cycle ;
\draw    (270.86,59.13) -- (270.86,123.93) ;
\draw    (192.57,59.13) -- (270.86,123.93) ;
\draw    (192.57,123.93) -- (231.72,169.28) ;
\draw    (270.86,123.93) -- (231.72,169.28) ;
\draw  [fill={rgb, 255:red, 0; green, 0; blue, 0 }  ,fill opacity=1 ] (34.31,222.63) .. controls (34.31,220.84) and (35.77,219.39) .. (37.57,219.39) .. controls (39.37,219.39) and (40.83,220.84) .. (40.83,222.63) .. controls (40.83,224.42) and (39.37,225.87) .. (37.57,225.87) .. controls (35.77,225.87) and (34.31,224.42) .. (34.31,222.63) -- cycle ;
\draw  [fill={rgb, 255:red, 0; green, 0; blue, 0 }  ,fill opacity=1 ] (112.6,222.63) .. controls (112.6,220.84) and (114.06,219.39) .. (115.86,219.39) .. controls (117.67,219.39) and (119.13,220.84) .. (119.13,222.63) .. controls (119.13,224.42) and (117.67,225.87) .. (115.86,225.87) .. controls (114.06,225.87) and (112.6,224.42) .. (112.6,222.63) -- cycle ;
\draw  [fill={rgb, 255:red, 0; green, 0; blue, 0 }  ,fill opacity=1 ] (73.45,332.78) .. controls (73.45,330.99) and (74.92,329.54) .. (76.72,329.54) .. controls (78.52,329.54) and (79.98,330.99) .. (79.98,332.78) .. controls (79.98,334.57) and (78.52,336.02) .. (76.72,336.02) .. controls (74.92,336.02) and (73.45,334.57) .. (73.45,332.78) -- cycle ;
\draw  [fill={rgb, 255:red, 0; green, 0; blue, 0 }  ,fill opacity=1 ] (34.31,287.43) .. controls (34.31,285.64) and (35.77,284.19) .. (37.57,284.19) .. controls (39.37,284.19) and (40.83,285.64) .. (40.83,287.43) .. controls (40.83,289.21) and (39.37,290.67) .. (37.57,290.67) .. controls (35.77,290.67) and (34.31,289.21) .. (34.31,287.43) -- cycle ;
\draw  [fill={rgb, 255:red, 0; green, 0; blue, 0 }  ,fill opacity=1 ] (112.6,287.43) .. controls (112.6,285.64) and (114.06,284.19) .. (115.86,284.19) .. controls (117.67,284.19) and (119.13,285.64) .. (119.13,287.43) .. controls (119.13,289.21) and (117.67,290.67) .. (115.86,290.67) .. controls (114.06,290.67) and (112.6,289.21) .. (112.6,287.43) -- cycle ;
\draw    (37.57,222.63) -- (37.57,287.43) ;
\draw    (115.86,222.63) -- (115.86,287.43) ;
\draw    (37.57,287.43) -- (76.72,332.78) ;
\draw    (115.86,287.43) -- (76.72,332.78) ;
\draw  [fill={rgb, 255:red, 0; green, 0; blue, 0 }  ,fill opacity=1 ] (188.81,224.13) .. controls (188.81,222.34) and (190.27,220.89) .. (192.07,220.89) .. controls (193.87,220.89) and (195.33,222.34) .. (195.33,224.13) .. controls (195.33,225.92) and (193.87,227.37) .. (192.07,227.37) .. controls (190.27,227.37) and (188.81,225.92) .. (188.81,224.13) -- cycle ;
\draw  [fill={rgb, 255:red, 0; green, 0; blue, 0 }  ,fill opacity=1 ] (267.1,224.13) .. controls (267.1,222.34) and (268.56,220.89) .. (270.36,220.89) .. controls (272.17,220.89) and (273.63,222.34) .. (273.63,224.13) .. controls (273.63,225.92) and (272.17,227.37) .. (270.36,227.37) .. controls (268.56,227.37) and (267.1,225.92) .. (267.1,224.13) -- cycle ;
\draw  [fill={rgb, 255:red, 0; green, 0; blue, 0 }  ,fill opacity=1 ] (227.95,334.28) .. controls (227.95,332.49) and (229.42,331.04) .. (231.22,331.04) .. controls (233.02,331.04) and (234.48,332.49) .. (234.48,334.28) .. controls (234.48,336.07) and (233.02,337.52) .. (231.22,337.52) .. controls (229.42,337.52) and (227.95,336.07) .. (227.95,334.28) -- cycle ;
\draw  [fill={rgb, 255:red, 0; green, 0; blue, 0 }  ,fill opacity=1 ] (188.81,288.93) .. controls (188.81,287.14) and (190.27,285.69) .. (192.07,285.69) .. controls (193.87,285.69) and (195.33,287.14) .. (195.33,288.93) .. controls (195.33,290.71) and (193.87,292.17) .. (192.07,292.17) .. controls (190.27,292.17) and (188.81,290.71) .. (188.81,288.93) -- cycle ;
\draw  [fill={rgb, 255:red, 0; green, 0; blue, 0 }  ,fill opacity=1 ] (267.1,288.93) .. controls (267.1,287.14) and (268.56,285.69) .. (270.36,285.69) .. controls (272.17,285.69) and (273.63,287.14) .. (273.63,288.93) .. controls (273.63,290.71) and (272.17,292.17) .. (270.36,292.17) .. controls (268.56,292.17) and (267.1,290.71) .. (267.1,288.93) -- cycle ;
\draw    (192.07,224.13) -- (270.36,288.93) ;
\draw    (192.07,288.93) -- (270.36,224.13) ;
\draw    (192.07,288.93) -- (231.22,334.28) ;
\draw    (270.36,288.93) -- (231.22,334.28) ;
\draw  [fill={rgb, 255:red, 0; green, 0; blue, 0 }  ,fill opacity=1 ] (380.81,125.63) .. controls (380.81,123.84) and (382.27,122.39) .. (384.07,122.39) .. controls (385.87,122.39) and (387.33,123.84) .. (387.33,125.63) .. controls (387.33,127.42) and (385.87,128.87) .. (384.07,128.87) .. controls (382.27,128.87) and (380.81,127.42) .. (380.81,125.63) -- cycle ;
\draw  [fill={rgb, 255:red, 0; green, 0; blue, 0 }  ,fill opacity=1 ] (459.1,125.63) .. controls (459.1,123.84) and (460.56,122.39) .. (462.36,122.39) .. controls (464.17,122.39) and (465.63,123.84) .. (465.63,125.63) .. controls (465.63,127.42) and (464.17,128.87) .. (462.36,128.87) .. controls (460.56,128.87) and (459.1,127.42) .. (459.1,125.63) -- cycle ;
\draw  [fill={rgb, 255:red, 0; green, 0; blue, 0 }  ,fill opacity=1 ] (419.95,235.78) .. controls (419.95,233.99) and (421.42,232.54) .. (423.22,232.54) .. controls (425.02,232.54) and (426.48,233.99) .. (426.48,235.78) .. controls (426.48,237.57) and (425.02,239.02) .. (423.22,239.02) .. controls (421.42,239.02) and (419.95,237.57) .. (419.95,235.78) -- cycle ;
\draw  [fill={rgb, 255:red, 0; green, 0; blue, 0 }  ,fill opacity=1 ] (380.81,190.43) .. controls (380.81,188.64) and (382.27,187.19) .. (384.07,187.19) .. controls (385.87,187.19) and (387.33,188.64) .. (387.33,190.43) .. controls (387.33,192.21) and (385.87,193.67) .. (384.07,193.67) .. controls (382.27,193.67) and (380.81,192.21) .. (380.81,190.43) -- cycle ;
\draw  [fill={rgb, 255:red, 0; green, 0; blue, 0 }  ,fill opacity=1 ] (459.1,190.43) .. controls (459.1,188.64) and (460.56,187.19) .. (462.36,187.19) .. controls (464.17,187.19) and (465.63,188.64) .. (465.63,190.43) .. controls (465.63,192.21) and (464.17,193.67) .. (462.36,193.67) .. controls (460.56,193.67) and (459.1,192.21) .. (459.1,190.43) -- cycle ;
\draw    (384.07,125.63) -- (384.07,190.43) ;
\draw    (384.07,190.43) -- (462.36,125.63) ;
\draw    (384.07,190.43) -- (423.22,235.78) ;
\draw    (462.36,190.43) -- (423.22,235.78) ;
\draw  [fill={rgb, 255:red, 0; green, 0; blue, 0 }  ,fill opacity=1 ] (535.31,127.13) .. controls (535.31,125.34) and (536.77,123.89) .. (538.57,123.89) .. controls (540.37,123.89) and (541.83,125.34) .. (541.83,127.13) .. controls (541.83,128.92) and (540.37,130.37) .. (538.57,130.37) .. controls (536.77,130.37) and (535.31,128.92) .. (535.31,127.13) -- cycle ;
\draw  [fill={rgb, 255:red, 0; green, 0; blue, 0 }  ,fill opacity=1 ] (613.6,127.13) .. controls (613.6,125.34) and (615.06,123.89) .. (616.86,123.89) .. controls (618.67,123.89) and (620.13,125.34) .. (620.13,127.13) .. controls (620.13,128.92) and (618.67,130.37) .. (616.86,130.37) .. controls (615.06,130.37) and (613.6,128.92) .. (613.6,127.13) -- cycle ;
\draw  [fill={rgb, 255:red, 0; green, 0; blue, 0 }  ,fill opacity=1 ] (574.45,237.28) .. controls (574.45,235.49) and (575.92,234.04) .. (577.72,234.04) .. controls (579.52,234.04) and (580.98,235.49) .. (580.98,237.28) .. controls (580.98,239.07) and (579.52,240.52) .. (577.72,240.52) .. controls (575.92,240.52) and (574.45,239.07) .. (574.45,237.28) -- cycle ;
\draw  [fill={rgb, 255:red, 0; green, 0; blue, 0 }  ,fill opacity=1 ] (535.31,191.93) .. controls (535.31,190.14) and (536.77,188.69) .. (538.57,188.69) .. controls (540.37,188.69) and (541.83,190.14) .. (541.83,191.93) .. controls (541.83,193.71) and (540.37,195.17) .. (538.57,195.17) .. controls (536.77,195.17) and (535.31,193.71) .. (535.31,191.93) -- cycle ;
\draw  [fill={rgb, 255:red, 0; green, 0; blue, 0 }  ,fill opacity=1 ] (613.6,191.93) .. controls (613.6,190.14) and (615.06,188.69) .. (616.86,188.69) .. controls (618.67,188.69) and (620.13,190.14) .. (620.13,191.93) .. controls (620.13,193.71) and (618.67,195.17) .. (616.86,195.17) .. controls (615.06,195.17) and (613.6,193.71) .. (613.6,191.93) -- cycle ;
\draw    (616.86,127.13) -- (616.86,191.93) ;
\draw    (538.57,127.13) -- (616.86,191.93) ;
\draw    (538.57,191.93) -- (577.72,237.28) ;
\draw    (616.86,191.93) -- (577.72,237.28) ;

\draw (95,18.5) node [anchor=north west][inner sep=0.75pt]   [align=left] {Product probability};
\draw (408.5,20.5) node [anchor=north west][inner sep=0.75pt]   [align=left] {Effective charge probability};
\draw (35.5,173.9) node [anchor=north west][inner sep=0.75pt]    {$\pi ( T) =1/4$};
\draw (188.5,174.9) node [anchor=north west][inner sep=0.75pt]    {$\pi ( T) =1/4$};
\draw (34,338.4) node [anchor=north west][inner sep=0.75pt]    {$\pi ( T) =1/4$};
\draw (188,338.4) node [anchor=north west][inner sep=0.75pt]    {$\pi ( T) =1/4$};
\draw (375,248.9) node [anchor=north west][inner sep=0.75pt]    {$p( T_{1}) =1/2$};
\draw (533,247.9) node [anchor=north west][inner sep=0.75pt]    {$p( T_{2}) =1/2$};
\draw (414,105.4) node [anchor=north west][inner sep=0.75pt]    {$T_{1}$};
\draw (568,105.4) node [anchor=north west][inner sep=0.75pt]    {$T_{2}$};

\end{tikzpicture}
    \caption{Differences between the product probability and the effective charge probability.}
    \label{5tree}
\end{figure}
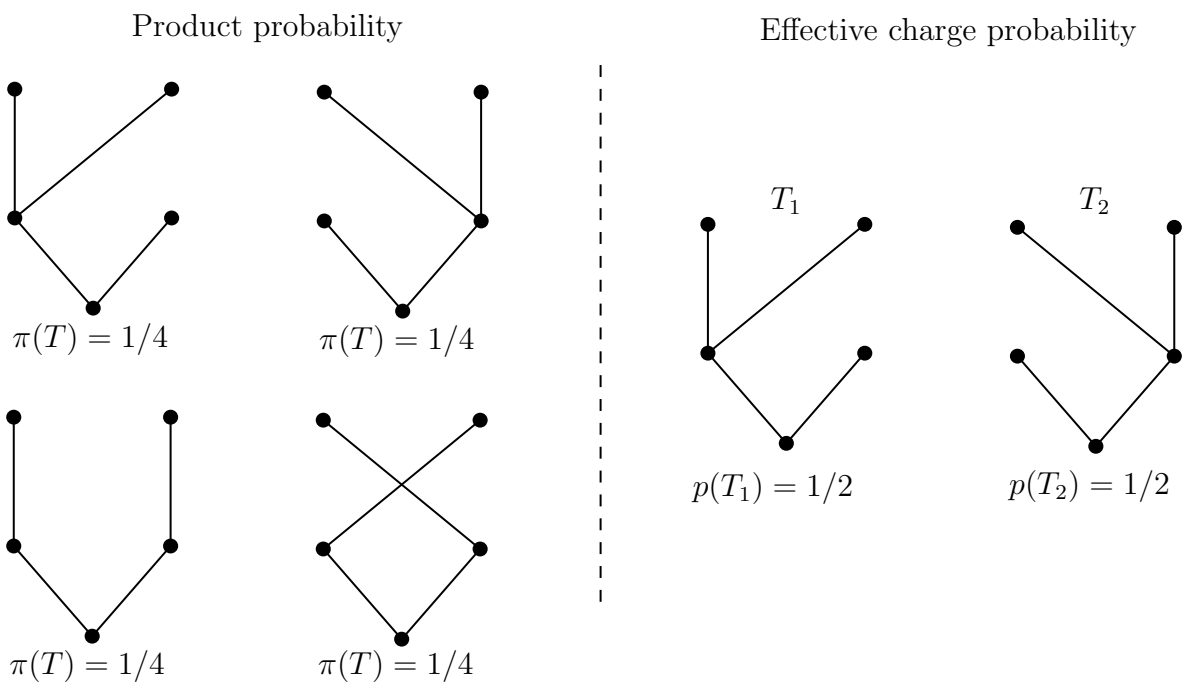

To finish this section we prove Corollary \ref{MAINCOR} connecting $\ell_1$-like transportation cost spaces to trees. A basis $(b_n)$ of $\mathcal{F}(M)$ will be called molecular if its basic elements are molecules. We state here Corollary \ref{MAINCOR} in a more precise way.

\begin{corollary}
    Let $M$ be a finite metric space and $(b_n)$ a molecular basis of $\mathcal{F}(M)$. Then, there is a tree $T\in\bigcup_{F\in\Sigma(M)}\mathcal{T}(F)$ such that 
    $$\max_{x,y\in M}\frac{d_T(x,y)}{d(x,y)}=d_1(b_n).$$
\end{corollary}
\begin{proof}
    We may assume without loss of generality that for every $n$ there are $x_n,y_n\in M$ such that $b_n=\delta_{x_n}-\delta_{y_n}$. Now, since $(b_n)$ is linearly independent, we know there must be an order $F\in\Sigma(M)$ and a reordering of the basis such that $b_n=\delta_{F(n)}-\delta_{F(i(n))}$ for some function $i:\{1,\dots,N\}\to \{0,\dots,N-1\}$ satisfying $i(n)<n$ for $n=1,\dots,N$. Therefore, $(b_n)\in B(M)$ and moreover 
    $$[\pi_{(b_n)}]_{x,y}\subset\big\{p\in\mathcal{P}(\mathcal{T}(F))\;:\;p(\{n,i\}\in T)=\delta_{\{n,i(n)\}}\;\forall n,i\in M,\;n>i\big\}.$$
    Clearly, $[\pi_{(b_n)}]_{x,y}=\{\pi_{(b_n)}\}$ where the support of $\pi_{(b_n)}$ consists of only one tree $T$. Therefore the conclusion of this corollary follows from Theorem \ref{mainthth}.
\end{proof}

\section{Applications: upper bounds of $d_1(M)$} \label{sec:appl}

In this section we are going to use the results of the previous section in order to obtain upper bounds for the stochastic $\ell_1^N$-distortion of certain families of finite metric spaces.

\subsection{Laakso graphs}

In this section we are going to give a positive solution to a problem raised by Dilworth, Kutzarova and Ostrovskii in \cite[Section 1.3, page 7]{DKO21}. In fact, their main goal in \cite{DKO21} is to investigate the Banach-Mazur distance from the Transportation cost space over the $k^{th}$-Laakso graph $\mathcal{L}_k$ to $\ell_1^N$ of the corresponding dimension. In fact, they obtained the following result.

\begin{theorem}[Corollary 17, \cite{DKO21}]
    For every $k\in\N$,
    $$d_1(\mathcal{L}_k)\geq \frac{3k-5}{8}.$$
\end{theorem}
It is worth mentioning that Laakso graphs play important roles in Metric Geometry as counterexamples to many natural questions. Here, we show that their lower bound is tight up to a universal constant factor. Indeed, using the machinery developed in Section \ref{sec:mainres} we prove the following result.

\begin{theorem}\label{Laath}
    For every $k\in\N$, 
    $$d_1(\mathcal{L}_k)\leq sd_1(\mathcal{L}_k)\leq 8k.$$
\end{theorem}

Laakso graphs where introduce in \cite{UC01} and their name comes from the paper \cite{Laa00} by T. J. Laakso.

\textbf{Description of the graphs.} Let us first introduce the Laakso graphs and the notation that we will use. The Laakso graphs $\{\mathcal{L}_k\}_{k=0}^\infty$ are graphs that are defined recursively. Firstly, $\mathcal{L}_1$ is the graph shown in Figure \ref{Laakso1}. Then, the Laakso graph $\mathcal{L}_k$ is obtained from $\mathcal{L}_{k-1}$ replacing each eadge $\{u,v\}\in E(\mathcal{L}_{k-1})$ by the graph $\mathcal{L}_1$ where the vertices $u$ and $v$ are identified with the vertices of degree 1 of $\mathcal{L}_1$.

With this identification in mind, we have that $V(\mathcal{L}_{k-1})\subset V(\mathcal{L}_{k})$ for every $k\in\N$. A vertex $u\in V(\mathcal{L}_k)$ will be then assigned a generation $g(u)\in\{1,\dots,k\}$ which is the moment when it first appeared in the latter construction. That is, $g(x)=1$ for $x\in V(\mathcal{L}_1)$ and for every $x\in \bigcup_{k\geq2}V(\mathcal{L}_k)\setminus V(\mathcal{L}_1)$ there must exist a unique $g(x)\geq2$ such that $x\in V(\mathcal{L}_{g(x)})\setminus V(\mathcal{L}_{g(x)-1})$.

Also, if $e_{k-1}\in E(\mathcal{L}_{k-1})$ is replaced by a copy of $\mathcal{L}_1$ in the construction of $\mathcal{L}_k$ we then denote $e_k\triangleleft e_{k-1}$ for every edge $e_k$ in that copy (in particular, $e_k\in E(\mathcal{L}_k)$).

Now, given $x\in V(\mathcal{L}_{k})$ with $g(x)\geq2$, there exists a unique $e(x)\in E(\mathcal{L}_{g(x)-1})$ such that every edge $e\in E(\mathcal{L}_{g(x)})$ incident to $x$, i.e. $x\in e$, satisfies $e\triangleleft e(x)$. In other words, $e(x)$ is the unique edge from $\mathcal{L}_{g(x)-1}$ that is replaced by a copy of $\mathcal{L}_1$ in the construction of $\mathcal{L}_{g(x)}$ such that $x$ belongs to that copy (see Figure \ref{Laakso1}).

\begin{center}
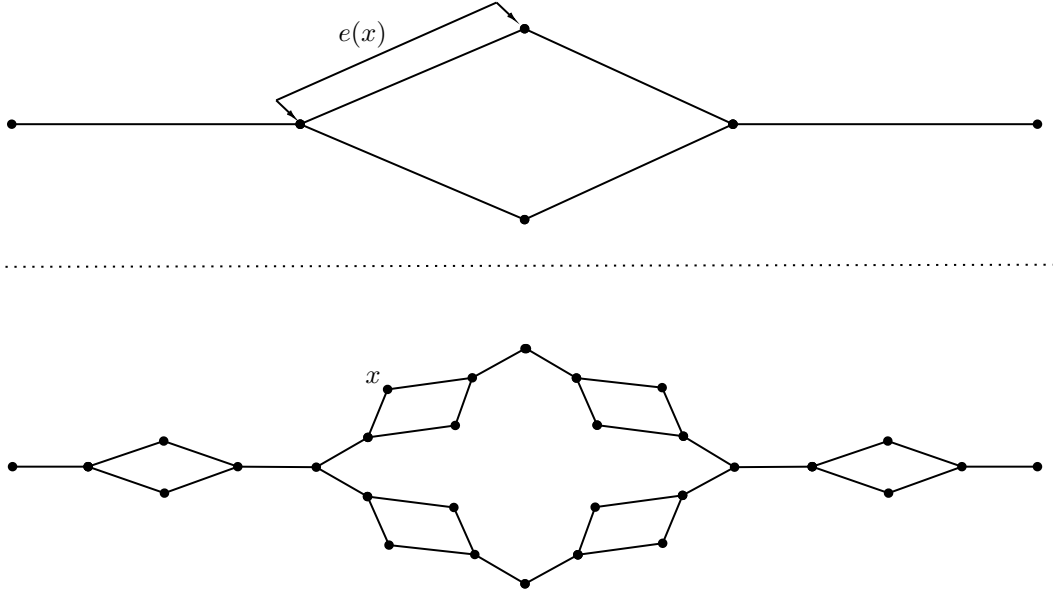
\begin{figure}
    \tikzset{every picture/.style={line width=0.75pt}} 

\begin{tikzpicture}[x=0.75pt,y=0.75pt,yscale=-.8,xscale=.8]
 \path (0,400); 

\draw    (10,80) -- (190,80) ;
\draw [shift={(190,80)}, rotate = 0] [color={rgb, 255:red, 0; green, 0; blue, 0 }  ][fill={rgb, 255:red, 0; green, 0; blue, 0 }  ][line width=0.75]      (0, 0) circle [x radius= 2.34, y radius= 2.34]   ;
\draw [shift={(10,80)}, rotate = 0] [color={rgb, 255:red, 0; green, 0; blue, 0 }  ][fill={rgb, 255:red, 0; green, 0; blue, 0 }  ][line width=0.75]      (0, 0) circle [x radius= 2.34, y radius= 2.34]   ;
\draw    (190,80) -- (330,20) ;
\draw [shift={(330,20)}, rotate = 336.8] [color={rgb, 255:red, 0; green, 0; blue, 0 }  ][fill={rgb, 255:red, 0; green, 0; blue, 0 }  ][line width=0.75]      (0, 0) circle [x radius= 2.34, y radius= 2.34]   ;
\draw [shift={(190,80)}, rotate = 336.8] [color={rgb, 255:red, 0; green, 0; blue, 0 }  ][fill={rgb, 255:red, 0; green, 0; blue, 0 }  ][line width=0.75]      (0, 0) circle [x radius= 2.34, y radius= 2.34]   ;
\draw    (330,140) -- (190,80) ;
\draw [shift={(190,80)}, rotate = 203.2] [color={rgb, 255:red, 0; green, 0; blue, 0 }  ][fill={rgb, 255:red, 0; green, 0; blue, 0 }  ][line width=0.75]      (0, 0) circle [x radius= 2.34, y radius= 2.34]   ;
\draw [shift={(330,140)}, rotate = 203.2] [color={rgb, 255:red, 0; green, 0; blue, 0 }  ][fill={rgb, 255:red, 0; green, 0; blue, 0 }  ][line width=0.75]      (0, 0) circle [x radius= 2.34, y radius= 2.34]   ;
\draw    (330,140) -- (460,80) ;
\draw [shift={(460,80)}, rotate = 335.22] [color={rgb, 255:red, 0; green, 0; blue, 0 }  ][fill={rgb, 255:red, 0; green, 0; blue, 0 }  ][line width=0.75]      (0, 0) circle [x radius= 2.34, y radius= 2.34]   ;
\draw [shift={(330,140)}, rotate = 335.22] [color={rgb, 255:red, 0; green, 0; blue, 0 }  ][fill={rgb, 255:red, 0; green, 0; blue, 0 }  ][line width=0.75]      (0, 0) circle [x radius= 2.34, y radius= 2.34]   ;
\draw    (330,20) -- (460,80) ;
\draw [shift={(460,80)}, rotate = 24.78] [color={rgb, 255:red, 0; green, 0; blue, 0 }  ][fill={rgb, 255:red, 0; green, 0; blue, 0 }  ][line width=0.75]      (0, 0) circle [x radius= 2.34, y radius= 2.34]   ;
\draw [shift={(330,20)}, rotate = 24.78] [color={rgb, 255:red, 0; green, 0; blue, 0 }  ][fill={rgb, 255:red, 0; green, 0; blue, 0 }  ][line width=0.75]      (0, 0) circle [x radius= 2.34, y radius= 2.34]   ;
\draw    (460,80) -- (650,80) ;
\draw [shift={(650,80)}, rotate = 0] [color={rgb, 255:red, 0; green, 0; blue, 0 }  ][fill={rgb, 255:red, 0; green, 0; blue, 0 }  ][line width=0.75]      (0, 0) circle [x radius= 2.34, y radius= 2.34]   ;
\draw [shift={(460,80)}, rotate = 0] [color={rgb, 255:red, 0; green, 0; blue, 0 }  ][fill={rgb, 255:red, 0; green, 0; blue, 0 }  ][line width=0.75]      (0, 0) circle [x radius= 2.34, y radius= 2.34]   ;
\draw    (10.5,295.33) -- (57.66,295.33) ;
\draw [shift={(57.66,295.33)}, rotate = 0] [color={rgb, 255:red, 0; green, 0; blue, 0 }  ][fill={rgb, 255:red, 0; green, 0; blue, 0 }  ][line width=0.75]      (0, 0) circle [x radius= 2.34, y radius= 2.34]   ;
\draw [shift={(10.5,295.33)}, rotate = 0] [color={rgb, 255:red, 0; green, 0; blue, 0 }  ][fill={rgb, 255:red, 0; green, 0; blue, 0 }  ][line width=0.75]      (0, 0) circle [x radius= 2.34, y radius= 2.34]   ;
\draw    (57.66,295.33) -- (104.82,279.22) ;
\draw [shift={(104.82,279.22)}, rotate = 341.14] [color={rgb, 255:red, 0; green, 0; blue, 0 }  ][fill={rgb, 255:red, 0; green, 0; blue, 0 }  ][line width=0.75]      (0, 0) circle [x radius= 2.34, y radius= 2.34]   ;
\draw [shift={(57.66,295.33)}, rotate = 341.14] [color={rgb, 255:red, 0; green, 0; blue, 0 }  ][fill={rgb, 255:red, 0; green, 0; blue, 0 }  ][line width=0.75]      (0, 0) circle [x radius= 2.34, y radius= 2.34]   ;
\draw    (105.29,311.96) -- (57.66,295.33) ;
\draw [shift={(57.66,295.33)}, rotate = 199.25] [color={rgb, 255:red, 0; green, 0; blue, 0 }  ][fill={rgb, 255:red, 0; green, 0; blue, 0 }  ][line width=0.75]      (0, 0) circle [x radius= 2.34, y radius= 2.34]   ;
\draw [shift={(105.29,311.96)}, rotate = 199.25] [color={rgb, 255:red, 0; green, 0; blue, 0 }  ][fill={rgb, 255:red, 0; green, 0; blue, 0 }  ][line width=0.75]      (0, 0) circle [x radius= 2.34, y radius= 2.34]   ;
\draw    (105.29,311.96) -- (151.04,295.33) ;
\draw [shift={(151.04,295.33)}, rotate = 340.02] [color={rgb, 255:red, 0; green, 0; blue, 0 }  ][fill={rgb, 255:red, 0; green, 0; blue, 0 }  ][line width=0.75]      (0, 0) circle [x radius= 2.34, y radius= 2.34]   ;
\draw    (104.82,279.22) -- (151.04,295.33) ;
\draw    (151.04,295.33) -- (200,295.65) ;
\draw [shift={(200,295.65)}, rotate = 0.38] [color={rgb, 255:red, 0; green, 0; blue, 0 }  ][fill={rgb, 255:red, 0; green, 0; blue, 0 }  ][line width=0.75]      (0, 0) circle [x radius= 2.34, y radius= 2.34]   ;
\draw    (200,295.65) -- (232.24,276.93) ;
\draw [shift={(232.24,276.93)}, rotate = 329.86] [color={rgb, 255:red, 0; green, 0; blue, 0 }  ][fill={rgb, 255:red, 0; green, 0; blue, 0 }  ][line width=0.75]      (0, 0) circle [x radius= 2.34, y radius= 2.34]   ;
\draw    (232.24,276.93) -- (244.51,246.77) ;
\draw [shift={(244.51,246.77)}, rotate = 292.14] [color={rgb, 255:red, 0; green, 0; blue, 0 }  ][fill={rgb, 255:red, 0; green, 0; blue, 0 }  ][line width=0.75]      (0, 0) circle [x radius= 2.34, y radius= 2.34]   ;
\draw    (286.75,269.42) -- (232.24,276.93) ;
\draw [shift={(232.24,276.93)}, rotate = 172.15] [color={rgb, 255:red, 0; green, 0; blue, 0 }  ][fill={rgb, 255:red, 0; green, 0; blue, 0 }  ][line width=0.75]      (0, 0) circle [x radius= 2.34, y radius= 2.34]   ;
\draw [shift={(286.75,269.42)}, rotate = 172.15] [color={rgb, 255:red, 0; green, 0; blue, 0 }  ][fill={rgb, 255:red, 0; green, 0; blue, 0 }  ][line width=0.75]      (0, 0) circle [x radius= 2.34, y radius= 2.34]   ;
\draw    (286.75,269.42) -- (297.36,239.46) ;
\draw    (244.51,246.77) -- (297.36,239.46) ;
\draw [shift={(297.36,239.46)}, rotate = 352.12] [color={rgb, 255:red, 0; green, 0; blue, 0 }  ][fill={rgb, 255:red, 0; green, 0; blue, 0 }  ][line width=0.75]      (0, 0) circle [x radius= 2.34, y radius= 2.34]   ;
\draw    (297.36,239.46) -- (330.25,221.13) ;
\draw [shift={(330.25,221.13)}, rotate = 330.86] [color={rgb, 255:red, 0; green, 0; blue, 0 }  ][fill={rgb, 255:red, 0; green, 0; blue, 0 }  ][line width=0.75]      (0, 0) circle [x radius= 2.34, y radius= 2.34]   ;
\draw    (200,295.65) -- (231.98,314.12) ;
\draw [shift={(231.98,314.12)}, rotate = 30.01] [color={rgb, 255:red, 0; green, 0; blue, 0 }  ][fill={rgb, 255:red, 0; green, 0; blue, 0 }  ][line width=0.75]      (0, 0) circle [x radius= 2.34, y radius= 2.34]   ;
\draw    (231.98,314.12) -- (285.85,320.98) ;
\draw [shift={(285.85,320.98)}, rotate = 7.27] [color={rgb, 255:red, 0; green, 0; blue, 0 }  ][fill={rgb, 255:red, 0; green, 0; blue, 0 }  ][line width=0.75]      (0, 0) circle [x radius= 2.34, y radius= 2.34]   ;
\draw    (245.41,344.64) -- (231.98,314.12) ;
\draw [shift={(231.98,314.12)}, rotate = 246.26] [color={rgb, 255:red, 0; green, 0; blue, 0 }  ][fill={rgb, 255:red, 0; green, 0; blue, 0 }  ][line width=0.75]      (0, 0) circle [x radius= 2.34, y radius= 2.34]   ;
\draw [shift={(245.41,344.64)}, rotate = 246.26] [color={rgb, 255:red, 0; green, 0; blue, 0 }  ][fill={rgb, 255:red, 0; green, 0; blue, 0 }  ][line width=0.75]      (0, 0) circle [x radius= 2.34, y radius= 2.34]   ;
\draw    (245.41,344.64) -- (298.91,350.57) ;
\draw [shift={(298.91,350.57)}, rotate = 6.33] [color={rgb, 255:red, 0; green, 0; blue, 0 }  ][fill={rgb, 255:red, 0; green, 0; blue, 0 }  ][line width=0.75]      (0, 0) circle [x radius= 2.34, y radius= 2.34]   ;
\draw    (285.85,320.98) -- (298.91,350.57) ;
\draw    (298.91,350.57) -- (330.25,369) ;
\draw [shift={(330.25,369)}, rotate = 30.46] [color={rgb, 255:red, 0; green, 0; blue, 0 }  ][fill={rgb, 255:red, 0; green, 0; blue, 0 }  ][line width=0.75]      (0, 0) circle [x radius= 2.34, y radius= 2.34]   ;
\draw    (461,295.65) -- (428.59,313.38) ;
\draw [shift={(428.59,313.38)}, rotate = 151.32] [color={rgb, 255:red, 0; green, 0; blue, 0 }  ][fill={rgb, 255:red, 0; green, 0; blue, 0 }  ][line width=0.75]      (0, 0) circle [x radius= 2.34, y radius= 2.34]   ;
\draw    (428.59,313.38) -- (416.14,343.51) ;
\draw    (374.04,320.79) -- (428.59,313.38) ;
\draw [shift={(428.59,313.38)}, rotate = 352.27] [color={rgb, 255:red, 0; green, 0; blue, 0 }  ][fill={rgb, 255:red, 0; green, 0; blue, 0 }  ][line width=0.75]      (0, 0) circle [x radius= 2.34, y radius= 2.34]   ;
\draw    (374.04,320.79) -- (363.25,350.73) ;
\draw [shift={(363.25,350.73)}, rotate = 109.82] [color={rgb, 255:red, 0; green, 0; blue, 0 }  ][fill={rgb, 255:red, 0; green, 0; blue, 0 }  ][line width=0.75]      (0, 0) circle [x radius= 2.34, y radius= 2.34]   ;
\draw [shift={(374.04,320.79)}, rotate = 109.82] [color={rgb, 255:red, 0; green, 0; blue, 0 }  ][fill={rgb, 255:red, 0; green, 0; blue, 0 }  ][line width=0.75]      (0, 0) circle [x radius= 2.34, y radius= 2.34]   ;
\draw    (416.14,343.51) -- (363.25,350.73) ;
\draw [shift={(363.25,350.73)}, rotate = 172.23] [color={rgb, 255:red, 0; green, 0; blue, 0 }  ][fill={rgb, 255:red, 0; green, 0; blue, 0 }  ][line width=0.75]      (0, 0) circle [x radius= 2.34, y radius= 2.34]   ;
\draw [shift={(416.14,343.51)}, rotate = 172.23] [color={rgb, 255:red, 0; green, 0; blue, 0 }  ][fill={rgb, 255:red, 0; green, 0; blue, 0 }  ][line width=0.75]      (0, 0) circle [x radius= 2.34, y radius= 2.34]   ;
\draw    (363.25,350.73) -- (330.25,369) ;
\draw [shift={(330.25,369)}, rotate = 151.02] [color={rgb, 255:red, 0; green, 0; blue, 0 }  ][fill={rgb, 255:red, 0; green, 0; blue, 0 }  ][line width=0.75]      (0, 0) circle [x radius= 2.34, y radius= 2.34]   ;
\draw [shift={(363.25,350.73)}, rotate = 151.02] [color={rgb, 255:red, 0; green, 0; blue, 0 }  ][fill={rgb, 255:red, 0; green, 0; blue, 0 }  ][line width=0.75]      (0, 0) circle [x radius= 2.34, y radius= 2.34]   ;
\draw    (461,295.65) -- (429.07,276.19) ;
\draw [shift={(429.07,276.19)}, rotate = 211.36] [color={rgb, 255:red, 0; green, 0; blue, 0 }  ][fill={rgb, 255:red, 0; green, 0; blue, 0 }  ][line width=0.75]      (0, 0) circle [x radius= 2.34, y radius= 2.34]   ;
\draw [shift={(461,295.65)}, rotate = 211.36] [color={rgb, 255:red, 0; green, 0; blue, 0 }  ][fill={rgb, 255:red, 0; green, 0; blue, 0 }  ][line width=0.75]      (0, 0) circle [x radius= 2.34, y radius= 2.34]   ;
\draw    (429.07,276.19) -- (375.25,269.22) ;
\draw [shift={(375.25,269.22)}, rotate = 187.38] [color={rgb, 255:red, 0; green, 0; blue, 0 }  ][fill={rgb, 255:red, 0; green, 0; blue, 0 }  ][line width=0.75]      (0, 0) circle [x radius= 2.34, y radius= 2.34]   ;
\draw [shift={(429.07,276.19)}, rotate = 187.38] [color={rgb, 255:red, 0; green, 0; blue, 0 }  ][fill={rgb, 255:red, 0; green, 0; blue, 0 }  ][line width=0.75]      (0, 0) circle [x radius= 2.34, y radius= 2.34]   ;
\draw    (415.82,245.64) -- (429.07,276.19) ;
\draw    (415.82,245.64) -- (400.78,243.95) -- (362.35,239.61) ;
\draw [shift={(415.82,245.64)}, rotate = 186.44] [color={rgb, 255:red, 0; green, 0; blue, 0 }  ][fill={rgb, 255:red, 0; green, 0; blue, 0 }  ][line width=0.75]      (0, 0) circle [x radius= 2.34, y radius= 2.34]   ;
\draw    (375.25,269.22) -- (362.35,239.61) ;
\draw [shift={(362.35,239.61)}, rotate = 246.47] [color={rgb, 255:red, 0; green, 0; blue, 0 }  ][fill={rgb, 255:red, 0; green, 0; blue, 0 }  ][line width=0.75]      (0, 0) circle [x radius= 2.34, y radius= 2.34]   ;
\draw [shift={(375.25,269.22)}, rotate = 246.47] [color={rgb, 255:red, 0; green, 0; blue, 0 }  ][fill={rgb, 255:red, 0; green, 0; blue, 0 }  ][line width=0.75]      (0, 0) circle [x radius= 2.34, y radius= 2.34]   ;
\draw    (362.35,239.61) -- (331.13,221.13) ;
\draw [shift={(331.13,221.13)}, rotate = 210.62] [color={rgb, 255:red, 0; green, 0; blue, 0 }  ][fill={rgb, 255:red, 0; green, 0; blue, 0 }  ][line width=0.75]      (0, 0) circle [x radius= 2.34, y radius= 2.34]   ;
\draw [shift={(362.35,239.61)}, rotate = 210.62] [color={rgb, 255:red, 0; green, 0; blue, 0 }  ][fill={rgb, 255:red, 0; green, 0; blue, 0 }  ][line width=0.75]      (0, 0) circle [x radius= 2.34, y radius= 2.34]   ;
\draw    (461,295.65) -- (509.46,295.33) ;
\draw [shift={(509.46,295.33)}, rotate = 359.62] [color={rgb, 255:red, 0; green, 0; blue, 0 }  ][fill={rgb, 255:red, 0; green, 0; blue, 0 }  ][line width=0.75]      (0, 0) circle [x radius= 2.34, y radius= 2.34]   ;
\draw    (509.46,295.33) -- (556.62,279.22) ;
\draw [shift={(556.62,279.22)}, rotate = 341.14] [color={rgb, 255:red, 0; green, 0; blue, 0 }  ][fill={rgb, 255:red, 0; green, 0; blue, 0 }  ][line width=0.75]      (0, 0) circle [x radius= 2.34, y radius= 2.34]   ;
\draw [shift={(509.46,295.33)}, rotate = 341.14] [color={rgb, 255:red, 0; green, 0; blue, 0 }  ][fill={rgb, 255:red, 0; green, 0; blue, 0 }  ][line width=0.75]      (0, 0) circle [x radius= 2.34, y radius= 2.34]   ;
\draw    (557.09,311.96) -- (509.46,295.33) ;
\draw [shift={(509.46,295.33)}, rotate = 199.25] [color={rgb, 255:red, 0; green, 0; blue, 0 }  ][fill={rgb, 255:red, 0; green, 0; blue, 0 }  ][line width=0.75]      (0, 0) circle [x radius= 2.34, y radius= 2.34]   ;
\draw [shift={(557.09,311.96)}, rotate = 199.25] [color={rgb, 255:red, 0; green, 0; blue, 0 }  ][fill={rgb, 255:red, 0; green, 0; blue, 0 }  ][line width=0.75]      (0, 0) circle [x radius= 2.34, y radius= 2.34]   ;
\draw    (557.09,311.96) -- (602.84,295.33) ;
\draw [shift={(602.84,295.33)}, rotate = 340.02] [color={rgb, 255:red, 0; green, 0; blue, 0 }  ][fill={rgb, 255:red, 0; green, 0; blue, 0 }  ][line width=0.75]      (0, 0) circle [x radius= 2.34, y radius= 2.34]   ;
\draw    (556.62,279.22) -- (602.84,295.33) ;
\draw    (602.84,295.33) -- (650,295.33) ;
\draw [shift={(650,295.33)}, rotate = 0] [color={rgb, 255:red, 0; green, 0; blue, 0 }  ][fill={rgb, 255:red, 0; green, 0; blue, 0 }  ][line width=0.75]      (0, 0) circle [x radius= 2.34, y radius= 2.34]   ;
\draw  [dash pattern={on 0.84pt off 2.51pt}]  (0.5,169.69) -- (659.5,168.36) ;
\draw    (175,65) -- (184.55,74.12) ;
\draw [shift={(186,75.5)}, rotate = 223.67] [color={rgb, 255:red, 0; green, 0; blue, 0 }  ][line width=0.75]    (4.37,-1.32) .. controls (2.78,-0.56) and (1.32,-0.12) .. (0,0) .. controls (1.32,0.12) and (2.78,0.56) .. (4.37,1.32)   ;
\draw    (312.5,3.5) -- (322.05,12.62) ;
\draw [shift={(323.5,14)}, rotate = 223.67] [color={rgb, 255:red, 0; green, 0; blue, 0 }  ][line width=0.75]    (4.37,-1.32) .. controls (2.78,-0.56) and (1.32,-0.12) .. (0,0) .. controls (1.32,0.12) and (2.78,0.56) .. (4.37,1.32)   ;
\draw    (175,65) -- (312.5,3.5) ;

\draw (235.5,239.18) node [font=\footnotesize]  {$x$};
\draw (229,24) node [font=\footnotesize]  {$ \begin{array}{l}
e(x)\\
\end{array}$};

\end{tikzpicture}

\caption{Above we have $\mathcal{L}_1$ and below we have $\mathcal{L}_2$. Note that for the vertex $x\in V(\mathcal{L}_2)$, $g(x)=2$ and $e(x)$ is the edge in $E(\mathcal{L}_1)$ that when replaced by a copy of $\mathcal{L}_1$ in the construction of $\mathcal{L}_{2}$, $x$ belongs to that copy.}
    \label{Laakso1}
\end{figure}
\end{center}

Finally, $d_k$ denotes the metric in $V(\mathcal{L}_k)$ given by the geodesic distance of $\mathcal{L}_k$ for the constant 1 weight function. We also denote $\mathcal{L}_k$ the metric space $(V(\mathcal{L}_k),d_k)$.

From now on, we will focus on the proof of Theorem \ref{Laath}. Therefore, we fix $k\in\N$ for the rest of this section and put $N=|V(\mathcal{L}_k)|-1$. Before we define a stochastic basis of $\mathcal{L}_k$, we need a simple lemma.

\begin{lemma}\label{Claim2}
    If $\{n,m\}\in E(\mathcal{L}_j)$ for $j>g(n)$ then $n\in e(m)$ and $g(m)=j$.
\end{lemma}
\begin{proof}
     It is clear that for every edge $e\in\mathcal{L}_j$ there is at least one point $z\in e$ with $g(z)=j$. Since $g(n)\neq j$ then $g(m)=j$. Now, consider $C_1$ the copy of $\mathcal{L}_1$ in $\mathcal{L}_j$ containing the edge $\{n,m\}$. By definition, we know that the vertices of degree 1 in $C_1$ are the elements of $e(m)$. If $n\notin e(m)$ then $n\in V(C_1)$ with degree $>1$ so that we would have $g(n)=j$ which is false. Therefore, $n\in e(m)$.
\end{proof}

\textbf{The stochastic basis.} We first need to define an order $F_k:\{0,\dots,N\}\to V(\mathcal{L}_k)$. We find $F_k$ inductively. We first define $F_1:\{0,\dots,5\}\to V(\mathcal{L}_1)$ arbitrarily. Then, if $F_{k-1}$ is given, we consider $F_k$ any bijective extension of $F_{k-1}$. For aesthetic reasons, we will identify from now on $V(\mathcal{L}_k)$ with $\{0,\dots,N\}$ so that $F_k$ is the identity mapping under this identification. The order of $V(\mathcal{L}_k)$ is set to satisfy the following property:
\begin{equation}\label{Leq1}
    0\leq n<m\leq N\;\;\;\;\;\;\Longrightarrow\;\;\;\;\;\;g(n)\leq g(m).
\end{equation}
We are then ready to define the stochastic basis $(b_n)_{n=1}^N$. If $n\in\{1,\dots,5\}$ we put $b_n=\delta_n$. Otherwise, if $6\leq n\leq N$ then

\begin{figure}

\tikzset{every picture/.style={line width=0.75pt}} 

\begin{tikzpicture}[x=0.75pt,y=0.75pt,yscale=-1,xscale=1]

\draw  [fill={rgb, 255:red, 0; green, 0; blue, 0 }  ,fill opacity=1 ] (320,132.5) .. controls (320,129.74) and (322.24,127.5) .. (325,127.5) .. controls (327.76,127.5) and (330,129.74) .. (330,132.5) .. controls (330,135.26) and (327.76,137.5) .. (325,137.5) .. controls (322.24,137.5) and (320,135.26) .. (320,132.5) -- cycle ;
\draw  [fill={rgb, 255:red, 0; green, 0; blue, 0 }  ,fill opacity=1 ] (550,82.5) .. controls (550,79.74) and (552.24,77.5) .. (555,77.5) .. controls (557.76,77.5) and (560,79.74) .. (560,82.5) .. controls (560,85.26) and (557.76,87.5) .. (555,87.5) .. controls (552.24,87.5) and (550,85.26) .. (550,82.5) -- cycle ;
\draw [color={rgb, 255:red, 0; green, 0; blue, 0 }  ,draw opacity=1 ]   (95,82.5) -- (205,82.5) ;
\draw    (325,32.5) -- (205,82.5) ;
\draw    (445,82.5) -- (325,132.5) ;
\draw [color={rgb, 255:red, 0; green, 0; blue, 0 }  ,draw opacity=1 ]   (445,82.5) -- (325,32.5) ;
\draw    (325,132.5) -- (205,82.5) ;
\draw    (445,82.5) -- (555,82.5) ;
\draw [color={rgb, 255:red, 0; green, 0; blue, 0 }  ,draw opacity=1 ]   (94.89,162.96) -- (554.89,162.96) ;
\draw [color={rgb, 255:red, 0; green, 0; blue, 0 }  ,draw opacity=1 ]   (94.89,162.96) -- (94.89,144.96) ;
\draw [shift={(94.89,142.96)}, rotate = 90] [color={rgb, 255:red, 0; green, 0; blue, 0 }  ,draw opacity=1 ][line width=0.75]    (10.93,-3.29) .. controls (6.95,-1.4) and (3.31,-0.3) .. (0,0) .. controls (3.31,0.3) and (6.95,1.4) .. (10.93,3.29)   ;
\draw [color={rgb, 255:red, 0; green, 0; blue, 0 }  ,draw opacity=1 ]   (554.89,162.96) -- (554.89,144.96) ;
\draw [shift={(554.89,142.96)}, rotate = 90] [color={rgb, 255:red, 0; green, 0; blue, 0 }  ,draw opacity=1 ][line width=0.75]    (10.93,-3.29) .. controls (6.95,-1.4) and (3.31,-0.3) .. (0,0) .. controls (3.31,0.3) and (6.95,1.4) .. (10.93,3.29)   ;
\draw  [fill={rgb, 255:red, 0; green, 0; blue, 0 }  ,fill opacity=1 ] (440,82.5) .. controls (440,79.74) and (442.24,77.5) .. (445,77.5) .. controls (447.76,77.5) and (450,79.74) .. (450,82.5) .. controls (450,85.26) and (447.76,87.5) .. (445,87.5) .. controls (442.24,87.5) and (440,85.26) .. (440,82.5) -- cycle ;
\draw  [fill={rgb, 255:red, 0; green, 0; blue, 0 }  ,fill opacity=1 ] (90,82.5) .. controls (90,79.74) and (92.24,77.5) .. (95,77.5) .. controls (97.76,77.5) and (100,79.74) .. (100,82.5) .. controls (100,85.26) and (97.76,87.5) .. (95,87.5) .. controls (92.24,87.5) and (90,85.26) .. (90,82.5) -- cycle ;
\draw  [fill={rgb, 255:red, 0; green, 0; blue, 0 }  ,fill opacity=1 ] (200,82.5) .. controls (200,79.74) and (202.24,77.5) .. (205,77.5) .. controls (207.76,77.5) and (210,79.74) .. (210,82.5) .. controls (210,85.26) and (207.76,87.5) .. (205,87.5) .. controls (202.24,87.5) and (200,85.26) .. (200,82.5) -- cycle ;
\draw  [fill={rgb, 255:red, 0; green, 0; blue, 0 }  ,fill opacity=1 ] (321,373) .. controls (321,370.24) and (323.24,368) .. (326,368) .. controls (328.76,368) and (331,370.24) .. (331,373) .. controls (331,375.76) and (328.76,378) .. (326,378) .. controls (323.24,378) and (321,375.76) .. (321,373) -- cycle ;
\draw  [fill={rgb, 255:red, 0; green, 0; blue, 0 }  ,fill opacity=1 ] (551,323) .. controls (551,320.24) and (553.24,318) .. (556,318) .. controls (558.76,318) and (561,320.24) .. (561,323) .. controls (561,325.76) and (558.76,328) .. (556,328) .. controls (553.24,328) and (551,325.76) .. (551,323) -- cycle ;
\draw [color={rgb, 255:red, 0; green, 0; blue, 0 }  ,draw opacity=1 ]   (96,323) -- (206,323) ;
\draw    (326,273) -- (206,323) ;
\draw    (446,323) -- (326,373) ;
\draw [color={rgb, 255:red, 0; green, 0; blue, 0 }  ,draw opacity=1 ]   (446,323) -- (326,273) ;
\draw    (326,373) -- (206,323) ;
\draw    (446,323) -- (556,323) ;
\draw [color={rgb, 255:red, 0; green, 0; blue, 0 }  ,draw opacity=1 ]   (95.89,403.46) -- (555.89,403.46) ;
\draw [color={rgb, 255:red, 0; green, 0; blue, 0 }  ,draw opacity=1 ]   (95.89,403.46) -- (95.89,385.46) ;
\draw [shift={(95.89,383.46)}, rotate = 90] [color={rgb, 255:red, 0; green, 0; blue, 0 }  ,draw opacity=1 ][line width=0.75]    (10.93,-3.29) .. controls (6.95,-1.4) and (3.31,-0.3) .. (0,0) .. controls (3.31,0.3) and (6.95,1.4) .. (10.93,3.29)   ;
\draw [color={rgb, 255:red, 0; green, 0; blue, 0 }  ,draw opacity=1 ]   (555.89,403.46) -- (555.89,385.46) ;
\draw [shift={(555.89,383.46)}, rotate = 90] [color={rgb, 255:red, 0; green, 0; blue, 0 }  ,draw opacity=1 ][line width=0.75]    (10.93,-3.29) .. controls (6.95,-1.4) and (3.31,-0.3) .. (0,0) .. controls (3.31,0.3) and (6.95,1.4) .. (10.93,3.29)   ;
\draw  [fill={rgb, 255:red, 0; green, 0; blue, 0 }  ,fill opacity=1 ] (321,273) .. controls (321,270.24) and (323.24,268) .. (326,268) .. controls (328.76,268) and (331,270.24) .. (331,273) .. controls (331,275.76) and (328.76,278) .. (326,278) .. controls (323.24,278) and (321,275.76) .. (321,273) -- cycle ;
\draw  [fill={rgb, 255:red, 0; green, 0; blue, 0 }  ,fill opacity=1 ] (91,323) .. controls (91,320.24) and (93.24,318) .. (96,318) .. controls (98.76,318) and (101,320.24) .. (101,323) .. controls (101,325.76) and (98.76,328) .. (96,328) .. controls (93.24,328) and (91,325.76) .. (91,323) -- cycle ;
\draw  [fill={rgb, 255:red, 0; green, 0; blue, 0 }  ,fill opacity=1 ] (201,323) .. controls (201,320.24) and (203.24,318) .. (206,318) .. controls (208.76,318) and (211,320.24) .. (211,323) .. controls (211,325.76) and (208.76,328) .. (206,328) .. controls (203.24,328) and (201,325.76) .. (201,323) -- cycle ;
\draw  [dash pattern={on 0.84pt off 2.51pt}]  (51.16,224.46) -- (601.16,222.96) ;
\draw [color={rgb, 255:red, 74; green, 144; blue, 226 }  ,draw opacity=1 ]   (325,32.5) .. controls (232.12,20.02) and (174.24,22.44) .. (104.21,69.74) ;
\draw [shift={(103.16,70.46)}, rotate = 325.75] [color={rgb, 255:red, 74; green, 144; blue, 226 }  ,draw opacity=1 ][line width=0.75]    (10.93,-3.29) .. controls (6.95,-1.4) and (3.31,-0.3) .. (0,0) .. controls (3.31,0.3) and (6.95,1.4) .. (10.93,3.29)   ;
\draw [color={rgb, 255:red, 74; green, 144; blue, 226 }  ,draw opacity=1 ]   (325,32.5) .. controls (413.21,18.03) and (472.56,18.95) .. (545.55,70.19) ;
\draw [shift={(546.66,70.96)}, rotate = 215.28] [color={rgb, 255:red, 74; green, 144; blue, 226 }  ,draw opacity=1 ][line width=0.75]    (10.93,-3.29) .. controls (6.95,-1.4) and (3.31,-0.3) .. (0,0) .. controls (3.31,0.3) and (6.95,1.4) .. (10.93,3.29)   ;
\draw  [fill={rgb, 255:red, 0; green, 0; blue, 0 }  ,fill opacity=1 ] (320,32.5) .. controls (320,29.74) and (322.24,27.5) .. (325,27.5) .. controls (327.76,27.5) and (330,29.74) .. (330,32.5) .. controls (330,35.26) and (327.76,37.5) .. (325,37.5) .. controls (322.24,37.5) and (320,35.26) .. (320,32.5) -- cycle ;
\draw [color={rgb, 255:red, 74; green, 144; blue, 226 }  ,draw opacity=1 ]   (446,323) .. controls (351.66,240.46) and (237.16,233.92) .. (104.66,311.46) ;
\draw [shift={(104.66,311.46)}, rotate = 329.66] [color={rgb, 255:red, 74; green, 144; blue, 226 }  ,draw opacity=1 ][line width=0.75]    (10.93,-3.29) .. controls (6.95,-1.4) and (3.31,-0.3) .. (0,0) .. controls (3.31,0.3) and (6.95,1.4) .. (10.93,3.29)   ;
\draw [color={rgb, 255:red, 74; green, 144; blue, 226 }  ,draw opacity=1 ]   (446,323) .. controls (478.17,306.67) and (512.61,306.91) .. (545.17,316.96) ;
\draw [shift={(546.66,317.42)}, rotate = 197.65] [color={rgb, 255:red, 74; green, 144; blue, 226 }  ,draw opacity=1 ][line width=0.75]    (10.93,-3.29) .. controls (6.95,-1.4) and (3.31,-0.3) .. (0,0) .. controls (3.31,0.3) and (6.95,1.4) .. (10.93,3.29)   ;
\draw  [fill={rgb, 255:red, 0; green, 0; blue, 0 }  ,fill opacity=1 ] (441,323) .. controls (441,320.24) and (443.24,318) .. (446,318) .. controls (448.76,318) and (451,320.24) .. (451,323) .. controls (451,325.76) and (448.76,328) .. (446,328) .. controls (443.24,328) and (441,325.76) .. (441,323) -- cycle ;

\draw (314,168.9) node [anchor=north west][inner sep=0.75pt]  [color={rgb, 255:red, 0; green, 0; blue, 0 }  ,opacity=1 ]  {$e( n)$};
\draw (315,409.4) node [anchor=north west][inner sep=0.75pt]  [color={rgb, 255:red, 0; green, 0; blue, 0 }  ,opacity=1 ]  {$e( n)$};
\draw (320.5,6.9) node [anchor=north west][inner sep=0.75pt]  [color={rgb, 255:red, 74; green, 144; blue, 226 }  ,opacity=1 ]  {$n$};
\draw (123.5,22.4) node [anchor=north west][inner sep=0.75pt]  [font=\small,color={rgb, 255:red, 74; green, 144; blue, 226 }  ,opacity=1 ]  {$\frac{2}{4}$};
\draw (517,23.9) node [anchor=north west][inner sep=0.75pt]  [font=\small,color={rgb, 255:red, 74; green, 144; blue, 226 }  ,opacity=1 ]  {$\frac{2}{4}$};
\draw (440.5,293.9) node [anchor=north west][inner sep=0.75pt]  [color={rgb, 255:red, 74; green, 144; blue, 226 }  ,opacity=1 ]  {$n$};
\draw (494,278.4) node [anchor=north west][inner sep=0.75pt]  [font=\small,color={rgb, 255:red, 74; green, 144; blue, 226 }  ,opacity=1 ]  {$\frac{3}{4}$};
\draw (133.5,259.4) node [anchor=north west][inner sep=0.75pt]  [font=\small,color={rgb, 255:red, 74; green, 144; blue, 226 }  ,opacity=1 ]  {$\frac{1}{4}$};

\end{tikzpicture}
    \caption{Both above and below we show the copy of $\mathcal{L}_1$ in $\mathcal{L}_{g(n)}$ containing an element $n\in \mathcal{L}_k$ with $n\geq 5$. We showcase in blue the two different possibilities for the basic transportation problem $b_n$.}
    \label{fig:Laabasis}
\end{figure}
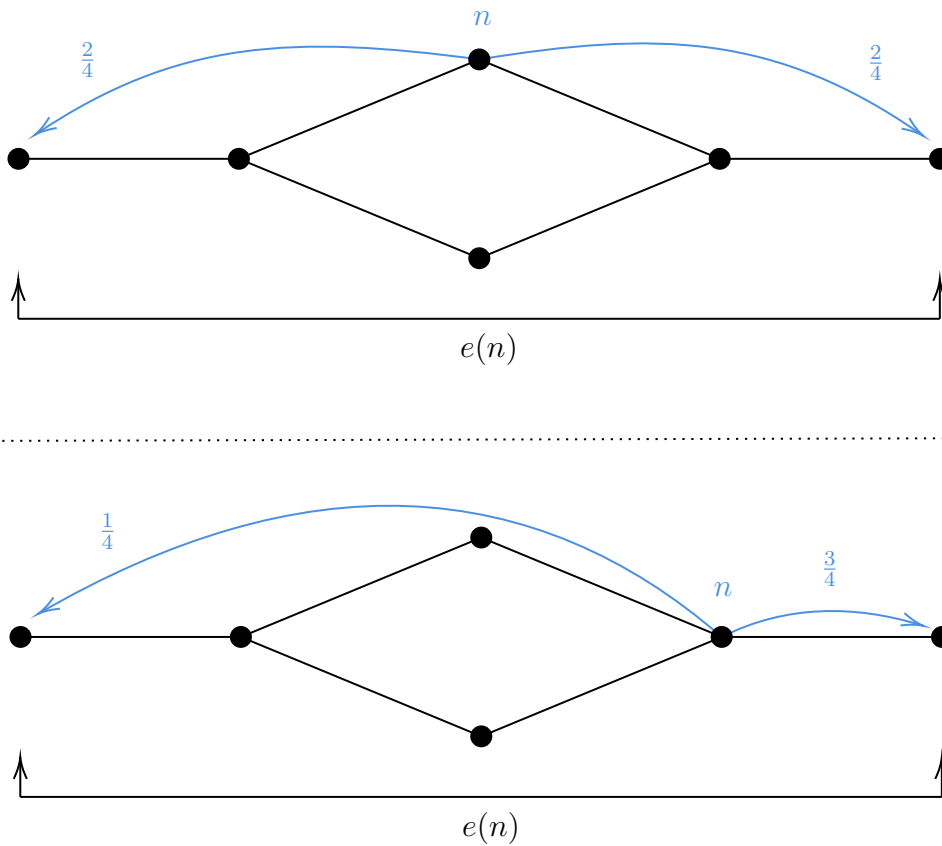

\begin{equation}\label{Leq2}
    b_n=\delta_n-\frac{d_{g(n)}\big(n,e(n)^-\big)\delta_{e(n)^+}+d_{g(n)}\big(n,e(n)^+\big)\delta_{e(n)^-}}{4},
\end{equation}
where $\{e(n)^-,e(n)^+\}=e(n)\in E(\mathcal{L}_{g(n)-1})$. See Figure \ref{fig:Laabasis} for a representation of $b_n$ as a transportation problem. We know that \eqref{Leq2} defines a stochastic basis because $e(n)^+,e(n)^-\in V(\mathcal{L}_{g(n)-1})$ implies that $g(e(n)^+),g(e(n)^-)<g(n)$ and hence, by \eqref{Leq1}, $e(n)^+,e(n)^-<n$. Also,
\begin{equation}\label{Leq3}
    d_{g(n)}\big(n,e(n)^-\big)+d_{g(n)}\big(n,e(n)^+\big)=4.
\end{equation}
Therefore, if we set for $0\leq i<n\leq N$ where $g(n)>1$,
\begin{equation}\label{Leq4}
    \lambda_{n,i}=\begin{cases}
        \frac{d_{g(n)}(n,e(n)^-)}{4}\;\;&\text{ if }  i=e(n)^+,\\
        \frac{d_{g(n)}(n,e(n)^+)}{4}\;\;&\text{ if } i=e(n)^-,\\
        0&\text{ else,}
    \end{cases}
\end{equation}
Then, for every $n\geq 6$ we have that $\sum_{i<n}\lambda_{n,i}=1$ and
$$b_n=\delta_n-\sum_{i<n}\lambda_{n,i}\delta_i.$$
From \eqref{Leq4} we deduce that for $i\neq 0$,
    \begin{equation}\label{Leq5}
        \lambda_{n,i}\neq0\;\;\;\Longleftrightarrow\;\;\;i\in e(n).
    \end{equation}
\begin{lemma}\label{Claim1}
    For every $1\leq n\leq N$,
    $$\|b_n\|\leq4^{k+1-g(n)}.$$
\end{lemma}
\begin{proof}
    We may assume that $n\geq 6$ because it is trivial for $n\leq5$ since $\text{diam}(\mathcal{L}_k)=4^k$ and $g(1)=\cdots=g(5)=1$. Then, by Proposition \ref{propbasisvect},
    \begin{equation}\label{Leq6}
        \|b_n\|=\frac{1}{4}\Big(d_{g(n)}\big(n,e(n)^-\big)d_k\big(n,e(n)^+\big)+d_{g(n)}\big(n,e(n)^+\big)d_k\big(n,e(n)^-\big)\Big).
    \end{equation}
    A straightforward induction proves that, for every $i=1,\dots,k$,
    \begin{equation}\label{Leq7}
        d_k(e^-,e^+)=4^{k-i}\;,\;\;\;\;\;\;\;\;\;\forall e=(e^-,e^+)\in E(\mathcal{L}_i).
    \end{equation}
    We know that $d_{g(n)}\big(n,e(n)^-\big),d_{g(n)}\big(n,e(n)^+\big)\in\{1,2,3\}$. Thus, taking into account \eqref{Leq3}, we divide the proof of the Lemma into 2 different cases.
    
    \textbf{Case 1:} $d_{g(n)}\big(n,e(n)^-\big),d_{g(n)}\big(n,e(n)^+\big)\in\{1,3\}$. We may assume without loss of generality that $d_{g(n)}\big(n,e(n)^-\big)=1$ and $d_{g(n)}\big(n,e(n)^+\big)=3$. Then,
    $$\begin{aligned}
        d_k\big(n,e(n)^-\big)&\stackrel{\eqref{Leq7}}{=}4^{k-g(n)},\\
        d_k\big(n,e(n)^+\big)&\stackrel{\eqref{Leq7}}{=}3\cdot4^{k-g(n)}.
    \end{aligned}$$
    Therefore, replacing this in equation \eqref{Leq6} we get
    $$\|b_n\|\leq\frac{6}{4}4^{k-g(F(n))}<4^{k+1-g(n)}.$$

    \textbf{Case 2:} $d_{g(n)}\big(n,e(n)^-\big)=d_{g(n)}\big(n,e(n)^+\big)=2$.
    In this case, again by equation \eqref{Leq7} we get
    $$d_k\big(n,e(n)^-\big)=d_k\big(n,e(n)^+\big)=2\cdot4^{k-g(n)},$$
    Again, replacing this in equation \eqref{Leq6} we get
    $$\|b_n\|\leq4^{k+1-g(n)}.$$
\end{proof}

\textbf{The coordinate functionals.} Consider $(b^*_n)_{n=1}^N$ the dual basis of $(b_n)_{n=1}^N\subset\mathcal{F}(\mathcal{L}_k)$ and $\{x,y\}\in E(\mathcal{L}_k)$ such that $x>y$. We denote during this subsection $\mu:=\delta_x-\delta_y\in\mathcal{F}(\mathcal{L}_k)$. We are going to prove the following result.
\begin{lemma}\label{Claim3}
    For every $i\in\{1,\dots,k\}$ there is $a_i:=\{n_i^-,n_i^+\}\in E(\mathcal{L}_i)$ satisfying:
    \begin{enumerate}
        \item \label{LaaP1} If $n\in\{1,\dots,N\}$ with $b^*_n(\mu)\neq0$ then $n\in a_{g(n)}$ and, if $n\geq6$, then $e(n)=a_{g(n)-1}$.
        \item \label{LaaP3} For $i\in\{1,\dots,k\}$ such that $n_i^\pm\neq 0$,
        $$b^*_{n_i^\pm}(\mu)=\pm4^{g(n_i^\pm)-k}.$$
    \end{enumerate}
\end{lemma}
\begin{proof}
    We first define inductively the $a_i$'s: We set $a_k=\{x,y\}$ and, assuming $a_{i+1}$ is defined, by Lemma \ref{Claim2} there must be $z\in a_{i+1}$ with $g(z)=i+1$. Then, we consider $a_i=e(z)$. By definition, we get that for every $n\geq6$, or equivalently $g(n)\geq2$,
    \begin{equation}\label{Leq8}
     n\in a_{g(n)}\;\;\;\;\Longrightarrow\;\;\;\;e(n)=a_{g(n)-1}.
    \end{equation}
    It is worth mentioning that \eqref{Leq8} implies that $a_k\triangleleft a_{k-1}\triangleleft\cdots\triangleleft a_1$. We now prove that the sequence $(a_i)_{i=1}^k$ satisfy Properties \eqref{LaaP1} and \eqref{LaaP3}.
    
    \textbf{Proof of Property \eqref{LaaP1}:} By \eqref{Leq8}, it is enough to show that, if we assume $b^*_n(\mu)\neq0$ for some $n\geq1$, then $n\in a_{g(n)}$. We proceed inductively in $n$. The first step ($n=N$) is immediate since $g(N)=k$ and, if $N\notin a_k=\{x,y\}$, it follows that
    $$b^*_n(\mu)\stackrel{\text{Prop }\ref{propbasisvect}}{=}\mu(N)+\sum_{m>N}b^*_m(\mu)\lambda_{m,N}=\mu(N)\stackrel{N\notin\{x,y\}}{=}0.$$
    Now, for the inductive step, we consider $n<N$ with $b^*_n(\mu)\neq0$ and assume as induction hypothesis $(IH)$ that Property \eqref{LaaP1} holds for every $m>n$. Then, either $n\in a_k=\{x,y\}$ or
    $$\begin{aligned}
        b^*_n(\mu)=&\mu(n)+\sum_{m>n}b^*_m(\mu)\lambda_{m,n}\stackrel{n\notin\{x,y\}}{=}\sum_{m>n}b^*_m(\mu)\lambda_{m,n}\\\stackrel{(IH)}{=}&\sum_{\substack{m>n\\m\in a_{g(m)}}}b^*_m(\mu)\lambda_{m,n}\stackrel{\eqref{Leq5},\eqref{Leq8}}{=}\sum_{\substack{m>n\\m\in a_{g(m)}\\n\in a_{g(m)-1}}}b^*_m(\mu)\lambda_{m,n}.\end{aligned}
    $$
    In any case, $b^*_n(\mu)\neq0$ implies that $n\in a_j$ for some $j\in\{1,\dots,k\}$. We are therefore done by the following claim.

    \textbf{Claim 1.} We claim that, if $n\in a_j$ for some $j\in \{1,\dots,k\}$, then $j\geq g(n)$ and $n\in a_s$ for every $s\in\{g(n),\dots,j\}$.
    
    Indeed, if $n\in a_j$ for $j>g(n)$ then by Lemma \ref{Claim2} we get that $g(m)=j$ and $n\in e(m)$. Now, by \eqref{Leq8} we get that $e(m)=a_{j-1}$ and therefore $n\in a_{j-1}$. Iterating this argument we get that $n\in a_s$ for every $s\in\{g(n),\dots,j\}$ and the claim is proven. 

    \textbf{Proof of Property \eqref{LaaP3}:} We denote $n_k^+=x$ and $n_k^-=y$. We also denote inductively $n_j^+,n_j^-\in a_j$ for $j<k$ such that
    \begin{equation}\label{Leq10}
        d_{j+1}\big(n_{j+1}^+,n_j^+\big)<d_{j+1}\big(n_{j+1}^+,n_j^-\big).
    \end{equation}
    This is always possible since $a_{j+1}\triangleleft a_j$. It is important to notice that, by \eqref{Leq10},
    \begin{equation}\label{Leq11}
        \text{If }n_j^\pm\in a_r\text{ for }r\neq j\text{ then }n_j^\pm=n_r^\pm.
    \end{equation}
    Now, consider $n$ such that $n\in a_{g(n)}$. We define the set 
    $$G_n=\{j\in\{1,\dots,k\}\;:\;n\in a_j\}.$$
    We set $j(n)=\max G_n$. We consider $n\in\{n_i^-,n_i^+\}$ for $i\in\{1,\dots,k\}$ and prove Property \eqref{LaaP3} for $n$ inductively in $g(n)$. For the first step of the induction $(g(n)=k)$ we have that $n_k^+=x$ and $n_k^-=y$. Also, for every $n\in\{1,\dots,N\}$,
    \begin{equation}\label{Laaeq17}
        b^*_n(\mu)\stackrel{\eqref{LaaP1}}{=}\mu(n)+\sum_{\substack{m>n\\m\in a_{g(m)}\\n\in a_{g(m)-1}}}b^*_m(\mu)\lambda_{m,n}.
    \end{equation}
    Therefore, if $g(n_k^\pm)=k$ then $b^*_{n_k^\pm}(\mu)=\mu(n_k^\pm)=\pm1$.
    
    Now, for the inductive step, we consider $n=n_i^\pm$ with $g(n_i^\pm)<k$ and assume as Induction Hypothesis $(IH)$ that Property \eqref{LaaP3} is satisfied for every other $m\in\{n_j^+,n_j^-\}_{j=1}^k$ with $g(m)>g(n)$. We divide the proof of this inductive step into 2 cases.

    \textbf{Case 1:} $n\neq y$. In this case, $n_i^\pm=n\notin\{x,y\}$ because $g(x)=k>i$ ($g(x)=k$ by Lemma \ref{Claim2} because $x>y$). Hence, $j(n_i^\pm)<k$ and
    \begin{equation}\label{Laaeq10}
    \begin{aligned}
        b_{n_i^\pm}^*(\mu)\stackrel{\eqref{Laaeq17}}{=}\sum_{\substack{m>n_i^\pm\\m\in a_{g(m)}\\{n_i^\pm}\in a_{g(m)-1}}}b^*_m(\mu)\lambda_{m,n_i^\pm}&=\sum_{j={g(n_i^\pm)}+1}^{k}\sum_{\substack{m>{n_i^\pm}\\g(m)=j\\m\in a_j\\{n_i^\pm}\in a_{j-1}}}b^*_m(\mu)\lambda_{m,n_i^\pm}\\&\stackrel{(*)}{=}\sum_{j={g(n_i^\pm)}+1}^{j({n_i^\pm})+1}\sum_{\substack{m>{n_i^\pm}\\g(m)=j\\m\in a_j\\{n_i^\pm}\in a_{j-1}}}b^*_m(\mu)\lambda_{m,n_i^\pm}
    \end{aligned}
    \end{equation}
    Here, $(*)$ holds true because, if $j>j({n_i^\pm})+1$, then ${n_i^\pm}\notin a_{j-1}$ by definition of $j({n_i^\pm})$. Therefore, we have that
    \begin{equation}\label{Laaeq14}
        \begin{aligned}
            b_{n_i^\pm}^*(\mu)&\stackrel{\eqref{Laaeq10}}{=}\sum_{\substack{m>{n_i^\pm}\\g(m)=j({n_i^\pm})+1\\m\in a_{j({n_i^\pm})+1}}}b^*_m(\mu)\lambda_{m,n_i^\pm}+\sum_{j={g(n_i^\pm)}+1}^{j({n_i^\pm})}\sum_{\substack{m>{n_i^\pm}\\g(m)=j\\m\in a_j\\{n_i^\pm}\in a_{j-1}}}b^*_m(\mu)\lambda_{m,n_i^\pm}\\
            &\stackrel{(*)}{=}\sum_{\substack{m>{n_i^\pm}\\g(m)=j({n_i^\pm})+1\\m\in a_{j({n_i^\pm})+1}}}b^*_m(\mu)\lambda_{m,n_i^\pm}+\sum_{j={g(n_i^\pm)}+1}^{j({n_i^\pm})}b^*_{n_j^\mp}(\mu)\lambda_{n_j^\mp,n_i^\pm}
        \end{aligned}
    \end{equation}
    In this case, equality $(*)$ holds because, if $m>{n_i^\pm}$ is such that $g(m)=j$, $m\in a_j$ and ${n_i^\pm}\in a_{j-1}\cap a_{j({n_i^\pm})}$ for some $j\in\{{g(n_i^\pm)}+1,\dots, j({n_i^\pm})\}$, then by Claim 1 ${n_i^\pm}\in a_j\setminus \{m\}$ and therefore ${n_i^\pm}=n_j^\pm$ and $m=n_j^\mp$ by \eqref{Leq11}.
    
    Under these conditions, $\{n_{j-1}^\mp,n_i^\pm\}=a_{j-1}\triangleleft a_j=\{n_j^\mp,n_i^\pm\}$ and we are in the situation described in Figure \ref{figLaa1}. Thus,
    \begin{equation}\label{Laaeq12}
        \lambda_{n_j^\mp,n_i^\pm}=\frac{d_j(n_j^\mp,n_{j-1}^\mp)}{4}=\frac{3}{4}.
    \end{equation}

    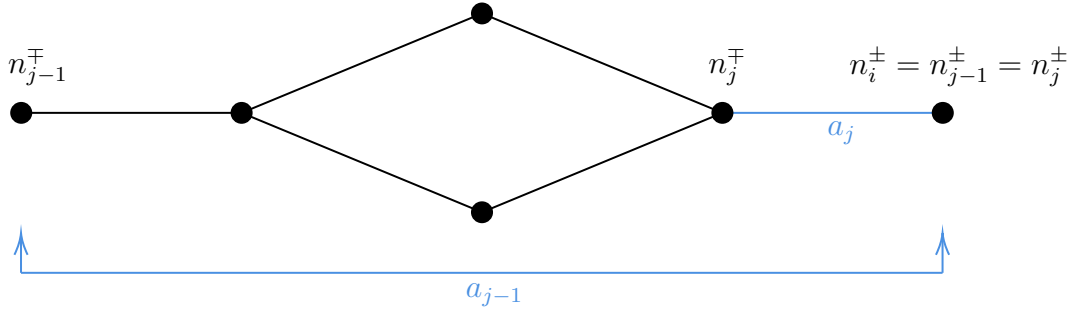
\begin{figure}

\tikzset{every picture/.style={line width=0.75pt}} 

\begin{tikzpicture}[x=0.75pt,y=0.75pt,yscale=-1,xscale=1]

\draw  [fill={rgb, 255:red, 0; green, 0; blue, 0 }  ,fill opacity=1 ] (320,25) .. controls (320,22.24) and (322.24,20) .. (325,20) .. controls (327.76,20) and (330,22.24) .. (330,25) .. controls (330,27.76) and (327.76,30) .. (325,30) .. controls (322.24,30) and (320,27.76) .. (320,25) -- cycle ;
\draw  [fill={rgb, 255:red, 0; green, 0; blue, 0 }  ,fill opacity=1 ] (320,125) .. controls (320,122.24) and (322.24,120) .. (325,120) .. controls (327.76,120) and (330,122.24) .. (330,125) .. controls (330,127.76) and (327.76,130) .. (325,130) .. controls (322.24,130) and (320,127.76) .. (320,125) -- cycle ;
\draw  [fill={rgb, 255:red, 0; green, 0; blue, 0 }  ,fill opacity=1 ] (200,75) .. controls (200,72.24) and (202.24,70) .. (205,70) .. controls (207.76,70) and (210,72.24) .. (210,75) .. controls (210,77.76) and (207.76,80) .. (205,80) .. controls (202.24,80) and (200,77.76) .. (200,75) -- cycle ;
\draw  [fill={rgb, 255:red, 0; green, 0; blue, 0 }  ,fill opacity=1 ] (90,75) .. controls (90,72.24) and (92.24,70) .. (95,70) .. controls (97.76,70) and (100,72.24) .. (100,75) .. controls (100,77.76) and (97.76,80) .. (95,80) .. controls (92.24,80) and (90,77.76) .. (90,75) -- cycle ;
\draw    (95,75) -- (205,75) ;
\draw    (325,25) -- (205,75) ;
\draw    (445,75) -- (325,125) ;
\draw    (445,75) -- (325,25) ;
\draw    (325,125) -- (205,75) ;
\draw [color={rgb, 255:red, 74; green, 144; blue, 226 }  ,draw opacity=1 ]   (445,75) -- (555,75) ;
\draw [color={rgb, 255:red, 74; green, 144; blue, 226 }  ,draw opacity=1 ]   (94.89,155.46) -- (554.89,155.46) ;
\draw [color={rgb, 255:red, 74; green, 144; blue, 226 }  ,draw opacity=1 ]   (94.89,155.46) -- (94.89,137.46) ;
\draw [shift={(94.89,135.46)}, rotate = 90] [color={rgb, 255:red, 74; green, 144; blue, 226 }  ,draw opacity=1 ][line width=0.75]    (10.93,-3.29) .. controls (6.95,-1.4) and (3.31,-0.3) .. (0,0) .. controls (3.31,0.3) and (6.95,1.4) .. (10.93,3.29)   ;
\draw [color={rgb, 255:red, 74; green, 144; blue, 226 }  ,draw opacity=1 ]   (554.89,155.46) -- (554.89,137.46) ;
\draw [shift={(554.89,135.46)}, rotate = 90] [color={rgb, 255:red, 74; green, 144; blue, 226 }  ,draw opacity=1 ][line width=0.75]    (10.93,-3.29) .. controls (6.95,-1.4) and (3.31,-0.3) .. (0,0) .. controls (3.31,0.3) and (6.95,1.4) .. (10.93,3.29)   ;
\draw  [fill={rgb, 255:red, 0; green, 0; blue, 0 }  ,fill opacity=1 ] (550,75) .. controls (550,72.24) and (552.24,70) .. (555,70) .. controls (557.76,70) and (560,72.24) .. (560,75) .. controls (560,77.76) and (557.76,80) .. (555,80) .. controls (552.24,80) and (550,77.76) .. (550,75) -- cycle ;
\draw  [fill={rgb, 255:red, 0; green, 0; blue, 0 }  ,fill opacity=1 ] (440,75) .. controls (440,72.24) and (442.24,70) .. (445,70) .. controls (447.76,70) and (450,72.24) .. (450,75) .. controls (450,77.76) and (447.76,80) .. (445,80) .. controls (442.24,80) and (440,77.76) .. (440,75) -- cycle ;

\draw (315.5,159.4) node [anchor=north west][inner sep=0.75pt]  [color={rgb, 255:red, 74; green, 144; blue, 226 }  ,opacity=1 ]  {$a_{j-1}$};
\draw (496,77.9) node [anchor=north west][inner sep=0.75pt]  [color={rgb, 255:red, 74; green, 144; blue, 226 }  ,opacity=1 ]  {$a_{j}$};
\draw (507,40.9) node [anchor=north west][inner sep=0.75pt]    {$n_{i}^{\pm } =n_{j-1}^{\pm } =n_{j}^{\pm }$};
\draw (436.5,40.4) node [anchor=north west][inner sep=0.75pt]    {$n_{j}^{\mp }$};
\draw (87,40.4) node [anchor=north west][inner sep=0.75pt]    {$n_{j-1}^{\mp }$};

\end{tikzpicture}
        \caption{The figure shows the copy of $\mathcal{L}_1$ in $\mathcal{L}_j$ containing $a_j\triangleleft a_{j-1}$.}
        \label{figLaa1}
    \end{figure}

    Also,if $j\in\{g(n_i^\pm+1,\dots,j(n_i^\pm))\}$ and $n_i^\pm=n_j^\pm$, then by Lemma \ref{Claim1} it follows that $g(n_j^\mp)=j$ so that by induction hypothesis $b^*_{n_j^\mp}(\mu)=\mp4^{j-k}$. Therefore, if $g(n_i^\pm)<j(n_i^\pm)$,
    \begin{equation}\label{Laaeq13}\begin{aligned}
    \sum_{j={g(n_i^\pm)}+1}^{j({n_i^\pm})}b^*_{n_j^\mp}(\mu)\lambda_{n_j^\mp,n_i^\pm}&\stackrel{\eqref{Laaeq12}}{=}\frac{3}{4}\sum_{j={g(n_i^\pm)}+1}^{j({n_i^\pm})}b^*_{n_j^\mp}(\mu)=\mp\frac{3}{4}\sum_{j={g(n_i^\pm)}+1}^{j({n_i^\pm})}4^{j-k}\\&=\mp\frac{3}{4}\sum_{j=0}^{j({n_i^\pm})-{g(n_i^\pm)}-1}4^{j({n_i^\pm})-k-j}\\&=\mp\frac{3}{4}4^{j({n_i^\pm})-k}\sum_{j=0}^{j(n)-{g(n_i^\pm)}-1}4^{-j}\\&=\mp4^{j({n_i^\pm})-k}\big(1-4^{{g(n_i^\pm)}-j({n_i^\pm})}\big).
    \end{aligned}\end{equation}
    It is important to mention that \eqref{Laaeq13} is trivially true in the case that ${g(n_i^\pm)}=j(n_i^\pm)$.
    Taking into account equations \eqref{Laaeq14} and\eqref{Laaeq13} we will be done with Case 1 if we manage to show that
    \begin{equation}
        \sum_{\substack{m>{n_i^\pm}\\g(m)=j({n_i^\pm})+1\\m\in a_{j({n_i^\pm})+1}}}b^*_m(\mu)\lambda_{m,n_i^\pm}=\pm4^{j({n_i^\pm})-k}.
    \end{equation}

    \textbf{Subcase 1.1:} $n_{j(n_i^\pm)}^\mp\notin a_{j(n_i^\pm)+1}$. Let us denote $j=j(n_i^\pm)$. In this case, $a_j\cap a_{j+1}=\emptyset$ since, by definition of $j=j(n_i^\pm)$, we know that $n_{j}^\pm\notin a_{j+1}$. Then, by Claim 1, we get that $g(z)=j+1$ for every $z\in a_{j+1}$. Indeed, if on the contrary $z\in a_{j+1}$ with $g(z)\leq j$, then by Claim 1 we would have that $z\in a_j\cap a_{j+1}$ which is impossible. Therefore, we have that $e(z)=a_j$ for every $z\in a_{j+1}=\{n_{j+1}^-,n_{j+1}^+\}$. This means that, by \eqref{Leq10}, our situation is as in Figure \ref{figLaa2} and thus
    \begin{equation}\label{Laaeq7}\begin{aligned}
    \lambda_{n_{j+1}^+,n_j^-}\stackrel{\eqref{Leq4}}{=}\frac{d_{j+1}\big(n_{j+1}^+,n_j^+\big)}{4}&=\frac{1}{4},\;\;\;\lambda_{n_{j+1}^+,n_j^+}\stackrel{\eqref{Leq4}}{=}\frac{d_{j+1}\big(n_{j+1}^+,n_j^-\big)}{4}&=\frac{3}{4},\\
    \lambda_{n_{j+1}^-,n_j^-}\stackrel{\eqref{Leq4}}{=}\frac{d_{j+1}\big(n_{j+1}^-,n_j^+\big)}{4}&=\frac{1}{2},\;\;\;\lambda_{n_{j+1}^-,n_j^+}\stackrel{\eqref{Leq4}}{=}\frac{d_{j+1}\big(n_{j+1}^-,n_j^-\big)}{4}&=\frac{1}{2}.
    \end{aligned}\end{equation}

    \begin{figure}

\tikzset{every picture/.style={line width=0.75pt}} 

\begin{tikzpicture}[x=0.75pt,y=0.75pt,yscale=-1,xscale=1]

\draw  [fill={rgb, 255:red, 0; green, 0; blue, 0 }  ,fill opacity=1 ] (320,125) .. controls (320,122.24) and (322.24,120) .. (325,120) .. controls (327.76,120) and (330,122.24) .. (330,125) .. controls (330,127.76) and (327.76,130) .. (325,130) .. controls (322.24,130) and (320,127.76) .. (320,125) -- cycle ;
\draw  [fill={rgb, 255:red, 0; green, 0; blue, 0 }  ,fill opacity=1 ] (200,75) .. controls (200,72.24) and (202.24,70) .. (205,70) .. controls (207.76,70) and (210,72.24) .. (210,75) .. controls (210,77.76) and (207.76,80) .. (205,80) .. controls (202.24,80) and (200,77.76) .. (200,75) -- cycle ;
\draw  [fill={rgb, 255:red, 0; green, 0; blue, 0 }  ,fill opacity=1 ] (550,75) .. controls (550,72.24) and (552.24,70) .. (555,70) .. controls (557.76,70) and (560,72.24) .. (560,75) .. controls (560,77.76) and (557.76,80) .. (555,80) .. controls (552.24,80) and (550,77.76) .. (550,75) -- cycle ;
\draw  [fill={rgb, 255:red, 0; green, 0; blue, 0 }  ,fill opacity=1 ] (90,75) .. controls (90,72.24) and (92.24,70) .. (95,70) .. controls (97.76,70) and (100,72.24) .. (100,75) .. controls (100,77.76) and (97.76,80) .. (95,80) .. controls (92.24,80) and (90,77.76) .. (90,75) -- cycle ;
\draw    (95,75) -- (205,75) ;
\draw    (325,25) -- (205,75) ;
\draw    (445,75) -- (325,125) ;
\draw [color={rgb, 255:red, 74; green, 144; blue, 226 }  ,draw opacity=1 ]   (445,75) -- (325,25) ;
\draw    (325,125) -- (205,75) ;
\draw    (445,75) -- (555,75) ;
\draw [color={rgb, 255:red, 74; green, 144; blue, 226 }  ,draw opacity=1 ]   (94.89,155.46) -- (554.89,155.46) ;
\draw [color={rgb, 255:red, 74; green, 144; blue, 226 }  ,draw opacity=1 ]   (94.89,155.46) -- (94.89,137.46) ;
\draw [shift={(94.89,135.46)}, rotate = 90] [color={rgb, 255:red, 74; green, 144; blue, 226 }  ,draw opacity=1 ][line width=0.75]    (10.93,-3.29) .. controls (6.95,-1.4) and (3.31,-0.3) .. (0,0) .. controls (3.31,0.3) and (6.95,1.4) .. (10.93,3.29)   ;
\draw [color={rgb, 255:red, 74; green, 144; blue, 226 }  ,draw opacity=1 ]   (554.89,155.46) -- (554.89,137.46) ;
\draw [shift={(554.89,135.46)}, rotate = 90] [color={rgb, 255:red, 74; green, 144; blue, 226 }  ,draw opacity=1 ][line width=0.75]    (10.93,-3.29) .. controls (6.95,-1.4) and (3.31,-0.3) .. (0,0) .. controls (3.31,0.3) and (6.95,1.4) .. (10.93,3.29)   ;
\draw  [fill={rgb, 255:red, 0; green, 0; blue, 0 }  ,fill opacity=1 ] (320,25) .. controls (320,22.24) and (322.24,20) .. (325,20) .. controls (327.76,20) and (330,22.24) .. (330,25) .. controls (330,27.76) and (327.76,30) .. (325,30) .. controls (322.24,30) and (320,27.76) .. (320,25) -- cycle ;
\draw  [fill={rgb, 255:red, 0; green, 0; blue, 0 }  ,fill opacity=1 ] (440,75) .. controls (440,72.24) and (442.24,70) .. (445,70) .. controls (447.76,70) and (450,72.24) .. (450,75) .. controls (450,77.76) and (447.76,80) .. (445,80) .. controls (442.24,80) and (440,77.76) .. (440,75) -- cycle ;

\draw (315.5,159.4) node [anchor=north west][inner sep=0.75pt]  [color={rgb, 255:red, 74; green, 144; blue, 226 }  ,opacity=1 ]  {$a_{j}$};
\draw (355.5,46.4) node [anchor=north west][inner sep=0.75pt]  [color={rgb, 255:red, 74; green, 144; blue, 226 }  ,opacity=1 ]  {$a_{j+1}$};
\draw (548.5,41.4) node [anchor=north west][inner sep=0.75pt]    {$n_{j}^{+}$};
\draw (436.5,40.4) node [anchor=north west][inner sep=0.75pt]    {$n_{j+1}^{+}$};
\draw (87,40.4) node [anchor=north west][inner sep=0.75pt]    {$n_{j}^{-}$};
\draw (334,-0.1) node [anchor=north west][inner sep=0.75pt]    {$n_{j+1}^{-}$};

\end{tikzpicture}
        \caption{The figure shows the copy of $\mathcal{L}_1$ in $\mathcal{L}_{j+1}$ containing $a_{j+1}\triangleleft a_{j}$.}
        \label{figLaa2}
    \end{figure}
    
    Now,
    \begin{equation}\label{Laaeq8}\begin{aligned}
        \sum_{\substack{m>n_j^\pm\\g(m)=j+1\\m\in a_{j+1}}}b^*_m(\mu)\lambda_{m,n_j^\pm}&\stackrel{(IH)}{=}4^{j+1-k}\lambda_{n_{j+1}^+,n_j^\pm}-4^{j+1-k}\lambda_{n_{j+1}^-,n_j^\pm}\\&=4^{j+1-k}\big(\lambda_{n_{j+1}^+,n_j^\pm}-\lambda_{n_{j+1}^-,n_j^\pm}\big)\\&\stackrel{\eqref{Laaeq7}}{=}4^{j+1-k}\bigg(\pm\frac{1}{4}\bigg)=\pm 4^{j-k}.
    \end{aligned}\end{equation}

    \textbf{Subcase 1.2:} $n_{j(n_i^\pm)}^\mp\in a_{j(n_i^\pm)+1}$. Let us denote again $j=j(n_i^\pm)$. Then, by \eqref{Leq11} $n_j^\mp=n_{j+1}^\mp$ so that $a_j=\{n_j^\mp,n_j^\pm\}$ and $a_{j+1}=\{n_j^\mp,n_{j+1}^\pm\}$. Therefore, by \eqref{Leq10}, our situation is as in Figure \ref{figLaa3} so that
    \begin{equation}\label{Laaeq14b}\begin{aligned}
    \lambda_{n_{j+1}^\pm,n_j^\pm}\stackrel{\eqref{Leq4}}{=}\frac{d_{j+1}\big(n_{j+1}^\pm,n_j^\mp\big)}{4}=\frac{1}{4}.
    \end{aligned}\end{equation}

    \begin{figure}

\tikzset{every picture/.style={line width=0.75pt}} 

\begin{tikzpicture}[x=0.75pt,y=0.75pt,yscale=-1,xscale=1]

\draw  [fill={rgb, 255:red, 0; green, 0; blue, 0 }  ,fill opacity=1 ] (320,125) .. controls (320,122.24) and (322.24,120) .. (325,120) .. controls (327.76,120) and (330,122.24) .. (330,125) .. controls (330,127.76) and (327.76,130) .. (325,130) .. controls (322.24,130) and (320,127.76) .. (320,125) -- cycle ;
\draw  [fill={rgb, 255:red, 0; green, 0; blue, 0 }  ,fill opacity=1 ] (550,75) .. controls (550,72.24) and (552.24,70) .. (555,70) .. controls (557.76,70) and (560,72.24) .. (560,75) .. controls (560,77.76) and (557.76,80) .. (555,80) .. controls (552.24,80) and (550,77.76) .. (550,75) -- cycle ;
\draw [color={rgb, 255:red, 74; green, 144; blue, 226 }  ,draw opacity=1 ]   (95,75) -- (205,75) ;
\draw    (325,25) -- (205,75) ;
\draw    (445,75) -- (325,125) ;
\draw [color={rgb, 255:red, 0; green, 0; blue, 0 }  ,draw opacity=1 ]   (445,75) -- (325,25) ;
\draw    (325,125) -- (205,75) ;
\draw    (445,75) -- (555,75) ;
\draw [color={rgb, 255:red, 74; green, 144; blue, 226 }  ,draw opacity=1 ]   (94.89,155.46) -- (554.89,155.46) ;
\draw [color={rgb, 255:red, 74; green, 144; blue, 226 }  ,draw opacity=1 ]   (94.89,155.46) -- (94.89,137.46) ;
\draw [shift={(94.89,135.46)}, rotate = 90] [color={rgb, 255:red, 74; green, 144; blue, 226 }  ,draw opacity=1 ][line width=0.75]    (10.93,-3.29) .. controls (6.95,-1.4) and (3.31,-0.3) .. (0,0) .. controls (3.31,0.3) and (6.95,1.4) .. (10.93,3.29)   ;
\draw [color={rgb, 255:red, 74; green, 144; blue, 226 }  ,draw opacity=1 ]   (554.89,155.46) -- (554.89,137.46) ;
\draw [shift={(554.89,135.46)}, rotate = 90] [color={rgb, 255:red, 74; green, 144; blue, 226 }  ,draw opacity=1 ][line width=0.75]    (10.93,-3.29) .. controls (6.95,-1.4) and (3.31,-0.3) .. (0,0) .. controls (3.31,0.3) and (6.95,1.4) .. (10.93,3.29)   ;
\draw  [fill={rgb, 255:red, 0; green, 0; blue, 0 }  ,fill opacity=1 ] (320,25) .. controls (320,22.24) and (322.24,20) .. (325,20) .. controls (327.76,20) and (330,22.24) .. (330,25) .. controls (330,27.76) and (327.76,30) .. (325,30) .. controls (322.24,30) and (320,27.76) .. (320,25) -- cycle ;
\draw  [fill={rgb, 255:red, 0; green, 0; blue, 0 }  ,fill opacity=1 ] (440,75) .. controls (440,72.24) and (442.24,70) .. (445,70) .. controls (447.76,70) and (450,72.24) .. (450,75) .. controls (450,77.76) and (447.76,80) .. (445,80) .. controls (442.24,80) and (440,77.76) .. (440,75) -- cycle ;
\draw  [fill={rgb, 255:red, 0; green, 0; blue, 0 }  ,fill opacity=1 ] (90,75) .. controls (90,72.24) and (92.24,70) .. (95,70) .. controls (97.76,70) and (100,72.24) .. (100,75) .. controls (100,77.76) and (97.76,80) .. (95,80) .. controls (92.24,80) and (90,77.76) .. (90,75) -- cycle ;
\draw  [fill={rgb, 255:red, 0; green, 0; blue, 0 }  ,fill opacity=1 ] (200,75) .. controls (200,72.24) and (202.24,70) .. (205,70) .. controls (207.76,70) and (210,72.24) .. (210,75) .. controls (210,77.76) and (207.76,80) .. (205,80) .. controls (202.24,80) and (200,77.76) .. (200,75) -- cycle ;

\draw (315.5,159.4) node [anchor=north west][inner sep=0.75pt]  [color={rgb, 255:red, 74; green, 144; blue, 226 }  ,opacity=1 ]  {$a_{j}$};
\draw (135.5,82.9) node [anchor=north west][inner sep=0.75pt]  [color={rgb, 255:red, 74; green, 144; blue, 226 }  ,opacity=1 ]  {$a_{j+1}$};
\draw (191.5,38.9) node [anchor=north west][inner sep=0.75pt]    {$n_{j+1}^{\pm }$};
\draw (41,41.4) node [anchor=north west][inner sep=0.75pt]    {$n_{j}^{\mp } =n_{j+1}^{\mp }$};
\draw (547,41.4) node [anchor=north west][inner sep=0.75pt]    {$n_{j}^{\pm }$};

\end{tikzpicture}
        \caption{The figure shows the copy of $\mathcal{L}_1$ in $\mathcal{L}_{j+1}$ containing $a_{j+1}\triangleleft a_{j}$.}
        \label{figLaa3}
    \end{figure}
    
    Clearly, $g(n_{j+1}^\mp)=g(n_{j}^\mp)<j+1$ so that by Lemma \ref{Claim2} we know $g(n_{j+1}^\pm)=j+1$. Therefore,
    \begin{equation}\label{Laaeq15}\begin{aligned}
        \sum_{\substack{m>n_j^\pm\\g(m)=j+1\\m\in a_{j+1}}}b^*_m(\mu)\lambda_{m,n_j^\pm}&=b^*_{n_{j+1}^\pm}(\mu)\lambda_{{n_{j+1}^\pm},n_j^\pm}\stackrel{(IH)}{=}\pm4^{j+1-k}\lambda_{n_{j+1}^\pm,n_j^\pm}\stackrel{\eqref{Laaeq14b}}{=}\pm 4^{j-k}.
    \end{aligned}\end{equation}

    \textbf{Case 2:} $n=y$. In this case, since $y=n_k^-$, by \eqref{Leq11} we get that $n=y=n_i^-$ and $j(n_i^-)=k$ by Lemma \ref{Claim2}. We may assume without loss of generality that $g(n_i^-)=i$. By \eqref{Leq11}, $n_i^-=\cdots=n_k^-=y$ so that, if $m\in a_j\setminus\{n_i^-\}$ for some $j\in\{i+1,\dots,k\}$, then $m=n_j^+$. Therefore,
    \begin{equation}\label{Laaeq18}\begin{aligned}
        \{m>n_i^-\;:\;m\in a_{g(m)},&\;n_i^-\in a_{g(m)-1}\}\\&=\{n_j^+\;:\;j=i+1,\dots,k\}.
    \end{aligned}\end{equation}
    Hence,
    \begin{equation}\label{Laaeq19}
        b_{n_i^-}^*(\mu)\stackrel{\eqref{Laaeq17}}{=}\mu({n_i^-})+\sum_{\substack{m>{n_i^-}\\m\in a_{g(m)}\\{n_i^-}\in a_{g(m)-1}}}b^*_m(\mu)\lambda_{m,{n_i^-}}\stackrel{\eqref{Laaeq18}}{=}-1+\sum_{j=i+1}^kb_{n_j^+}^*(\mu)\lambda_{n_j^+,n_i^-}.
    \end{equation}
    Finally, if $j\in\{g+1,\dots,k\}$, since $\{n_j^+,n_i^-\}=a_j\in E(\mathcal{L}_j)$, it holds that $d_j(n_j^+,n_i^-)=1$ and $n_i^-=n_{j-1}^-$ so that $d_j(n_j^+,n_{j-1}^+)=3$ as exposed in Figure \ref{figLaa4}. Therefore, for every $j\in\{i+1,\dots,k\}$,
    \begin{equation}\label{Laaeq20}
        \lambda_{n_j^+,n_i^-}=\lambda_{n_j^+,n_{j-1}^-}\stackrel{\eqref{Leq4}}{=}\frac{d_j\big(n_j^+,n_{j-1}^+\big)}{4}=\frac{3}{4}.
    \end{equation}

    \begin{figure}

\tikzset{every picture/.style={line width=0.75pt}} 

\begin{tikzpicture}[x=0.75pt,y=0.75pt,yscale=-1,xscale=1]

\draw  [fill={rgb, 255:red, 0; green, 0; blue, 0 }  ,fill opacity=1 ] (320,125) .. controls (320,122.24) and (322.24,120) .. (325,120) .. controls (327.76,120) and (330,122.24) .. (330,125) .. controls (330,127.76) and (327.76,130) .. (325,130) .. controls (322.24,130) and (320,127.76) .. (320,125) -- cycle ;
\draw  [fill={rgb, 255:red, 0; green, 0; blue, 0 }  ,fill opacity=1 ] (550,75) .. controls (550,72.24) and (552.24,70) .. (555,70) .. controls (557.76,70) and (560,72.24) .. (560,75) .. controls (560,77.76) and (557.76,80) .. (555,80) .. controls (552.24,80) and (550,77.76) .. (550,75) -- cycle ;
\draw [color={rgb, 255:red, 74; green, 144; blue, 226 }  ,draw opacity=1 ]   (95,75) -- (205,75) ;
\draw    (325,25) -- (205,75) ;
\draw    (445,75) -- (325,125) ;
\draw [color={rgb, 255:red, 0; green, 0; blue, 0 }  ,draw opacity=1 ]   (445,75) -- (325,25) ;
\draw    (325,125) -- (205,75) ;
\draw    (445,75) -- (555,75) ;
\draw [color={rgb, 255:red, 74; green, 144; blue, 226 }  ,draw opacity=1 ]   (94.89,155.46) -- (554.89,155.46) ;
\draw [color={rgb, 255:red, 74; green, 144; blue, 226 }  ,draw opacity=1 ]   (94.89,155.46) -- (94.89,137.46) ;
\draw [shift={(94.89,135.46)}, rotate = 90] [color={rgb, 255:red, 74; green, 144; blue, 226 }  ,draw opacity=1 ][line width=0.75]    (10.93,-3.29) .. controls (6.95,-1.4) and (3.31,-0.3) .. (0,0) .. controls (3.31,0.3) and (6.95,1.4) .. (10.93,3.29)   ;
\draw [color={rgb, 255:red, 74; green, 144; blue, 226 }  ,draw opacity=1 ]   (554.89,155.46) -- (554.89,137.46) ;
\draw [shift={(554.89,135.46)}, rotate = 90] [color={rgb, 255:red, 74; green, 144; blue, 226 }  ,draw opacity=1 ][line width=0.75]    (10.93,-3.29) .. controls (6.95,-1.4) and (3.31,-0.3) .. (0,0) .. controls (3.31,0.3) and (6.95,1.4) .. (10.93,3.29)   ;
\draw  [fill={rgb, 255:red, 0; green, 0; blue, 0 }  ,fill opacity=1 ] (320,25) .. controls (320,22.24) and (322.24,20) .. (325,20) .. controls (327.76,20) and (330,22.24) .. (330,25) .. controls (330,27.76) and (327.76,30) .. (325,30) .. controls (322.24,30) and (320,27.76) .. (320,25) -- cycle ;
\draw  [fill={rgb, 255:red, 0; green, 0; blue, 0 }  ,fill opacity=1 ] (440,75) .. controls (440,72.24) and (442.24,70) .. (445,70) .. controls (447.76,70) and (450,72.24) .. (450,75) .. controls (450,77.76) and (447.76,80) .. (445,80) .. controls (442.24,80) and (440,77.76) .. (440,75) -- cycle ;
\draw  [fill={rgb, 255:red, 0; green, 0; blue, 0 }  ,fill opacity=1 ] (90,75) .. controls (90,72.24) and (92.24,70) .. (95,70) .. controls (97.76,70) and (100,72.24) .. (100,75) .. controls (100,77.76) and (97.76,80) .. (95,80) .. controls (92.24,80) and (90,77.76) .. (90,75) -- cycle ;
\draw  [fill={rgb, 255:red, 0; green, 0; blue, 0 }  ,fill opacity=1 ] (200,75) .. controls (200,72.24) and (202.24,70) .. (205,70) .. controls (207.76,70) and (210,72.24) .. (210,75) .. controls (210,77.76) and (207.76,80) .. (205,80) .. controls (202.24,80) and (200,77.76) .. (200,75) -- cycle ;

\draw (316,159.4) node [anchor=north west][inner sep=0.75pt]  [color={rgb, 255:red, 74; green, 144; blue, 226 }  ,opacity=1 ]  {$a_{j-1}$};
\draw (135.5,82.9) node [anchor=north west][inner sep=0.75pt]  [color={rgb, 255:red, 74; green, 144; blue, 226 }  ,opacity=1 ]  {$a_{j}$};
\draw (196,40.4) node [anchor=north west][inner sep=0.75pt]    {$n_{j}^{+}$};
\draw (37.5,42.9) node [anchor=north west][inner sep=0.75pt]    {$n_{j-1}^{-} =n_{j}^{-} =y$};
\draw (547,41.4) node [anchor=north west][inner sep=0.75pt]    {$n_{j-1}^{+}$};

\end{tikzpicture}
        \caption{The figure shows the copy of $\mathcal{L}_1$ in $\mathcal{L}_j$ containing $a_j\triangleleft a_{j-1}$.}
        \label{figLaa4}
    \end{figure}
    
    Also, if $j>i$ then $g(n_j^+)=j$ since $g(n_j^-)=g(y)=i<j$. We are therefore done since
    $$\begin{aligned}b_{n_i^-}^*(\mu)&\stackrel{\eqref{Laaeq19}}{=}-1+\sum_{j=i+1}^kb_{n_j^+}^*(\mu)\lambda_{n_j^+,n_i^-}\stackrel{\eqref{Laaeq20}}{=}-1+\frac{3}{4}\sum_{j=i+1}^kb_{n_j^+}^*(\mu)\\&\stackrel{(IH)}{=}-1+\frac{3}{4}\sum_{j=i+1}^k4^{j-k}=-1+\frac{3}{4}\sum_{j=0}^{k-i-1}4^{-j}=-4^{i-k}.\end{aligned}$$
\end{proof}

\begin{proof}[Proof of Theorem \ref{Laath}]
    We will use Lemma \ref{Claim1} and Lemma \ref{Claim3} to show that 
    \begin{equation}
        \sum_{n=1}^N|b^*_n(\mu)|\|b_n\|\leq8k.
    \end{equation}
    Indeed,
    $$\begin{aligned}
        \sum_{n=1}^N|b^*_n(\mu)|\|b_n\|&\stackrel{\text{Lemma }\ref{Claim1}}{=}\sum_{n=1}^N|b^*_n(\mu)|4^{k+1-g(n)}\\&\stackrel{\text{Lemma }\ref{Claim3}\,\eqref{LaaP1}}{=}\sum_{j=1}^k\sum_{\substack{n\in\{1,\dots,N\}\\g(n)=j\\n\in a_j}}|b^*_n(\mu)|4^{k+1-g(n)}\\&\stackrel{\text{Lemma }\ref{Claim3}\,\eqref{LaaP3}}{=}\sum_{j=1}^k\sum_{\substack{n\in\{1,\dots,N\}\\g(n)=j\\n\in a_j}}4^{g(n)-k}4^{k+1-g(n)}\\&=\sum_{j=1}^k\sum_{\substack{n\in\{1,\dots,N\}\\g(n)=j\\n\in a_j}}4\leq\sum_{j=1}^k(4+4)=8k.
    \end{aligned}$$
\end{proof}

It is not difficult to see that $|V(\mathcal{L}_k)|=2+4\sum_{i=0}^{k-1}6^i$ and hence
$$\log{|V(\mathcal{L}_k)|}\approx k.$$
It is therefore natural to ask the following question.

\begin{question}
 Is there a universal constant $C>0$ such that $d_1(M)\leq C\log{|M|}$ for every finite metric space $M$? What about $sd_1(M)$?
\end{question}

Notice that, by \cite{FRT04}, every finite metric space $M$ is $O(\log|M|)$-isomorphic to a stochastic tree so that, by \cite[Corollary 3.2]{Gar24}, the Transportation Cost space over any finite metric space $M$ is $O(\log|M|)$-isomorphic to a $O(\log|M|)$-complemented subspace of $\ell_1$. In fact, in \cite{ST24} it is proved that $\log|M|$ is a tight upper bound of the stochastic tree distortion up to a universal constant factor for several families of metric spaces including the Laakso graphs.

\subsection{Finite Hyperbolic approximations of a metric space}

Firstly, we rigorously define the hyperbolic approximations of a metric space $M$. Let $(M,d)$ be an arbitrary (not necessarily finite) metric space. We are going to construct an increasing sequence of metric spaces $(M_k)_{k\in\N}$ which are known as finite hyperbolic approximations of $M$ (see \cite[Section 6.1]{BS07}). We first need to introduce the concept of $\varepsilon$-net for $\varepsilon>0$. Here, an $\varepsilon$-net of $M$ is simply a set $N\subset M$ satisfying the following 2 properties: 
\begin{itemize}
    \item For every $x,y\in N$ distinct we have that $d(x,y)\geq \varepsilon$. This is usually referred as $\varepsilon$-separated set.
    \item For every $x\in M$ there is $y\in N$ with $d(x,y)\leq \varepsilon$. This is usually referred as $\varepsilon$-dense set.
\end{itemize}
The existence of an $\varepsilon$-net of $M$ for every $\varepsilon>0$ is ensured by considering a maximal $\ep$-separated set of $M$.

In our setting, finite hyperbolic approximations will depend on 4 different inputs: 2 real valued parameters $\lambda,r>0$ with $r\leq1/6$ and $\lambda\geq2$, a metric space $M$, and a sequence $(L_i)_{i\in\mathbb{Z}}$ of $r^i$-nets of $M$ for $i\in\mathbb{Z}$. By now fixing these four inputs we construct the sequence of finite hyperbolic approximations.

We denote $B_M(x,R)$ the ball in $M$ with center $x\in M$ and radius $R\geq0$. We will say that $L_i$ is the $i^{th}$ layer of $M$ for $i\in \mathbb{Z}$. We now define the graph $G_k$ for $k\in\N$. We first set the vertices $V(G_k)=\bigsqcup_{i=-k}^kL_i$. Finally, let us define the edges $E(G_k)$. There are going to be edges connecting points from the same layer $L_i$, called \textit{horizontal edges}, and edges connecting elements from a layer $L_i$ to elements of the next layer $L_{i+1}$, called \textit{radial edges}. Two elements $x,y\in L_i$ are connected by a horizontal edge if and only if $B_M(x,\lambda r^i)\cap B_M(y,\lambda r^i)\neq \emptyset$. An element of $x\in L_i$ is connected to an element $y\in L_{i+1}$ by a radial edge if and only if $B_M(y,\lambda r^{i+1})\subset B_M(x,\lambda r^i)$. See Figure \ref{figHyp1}.

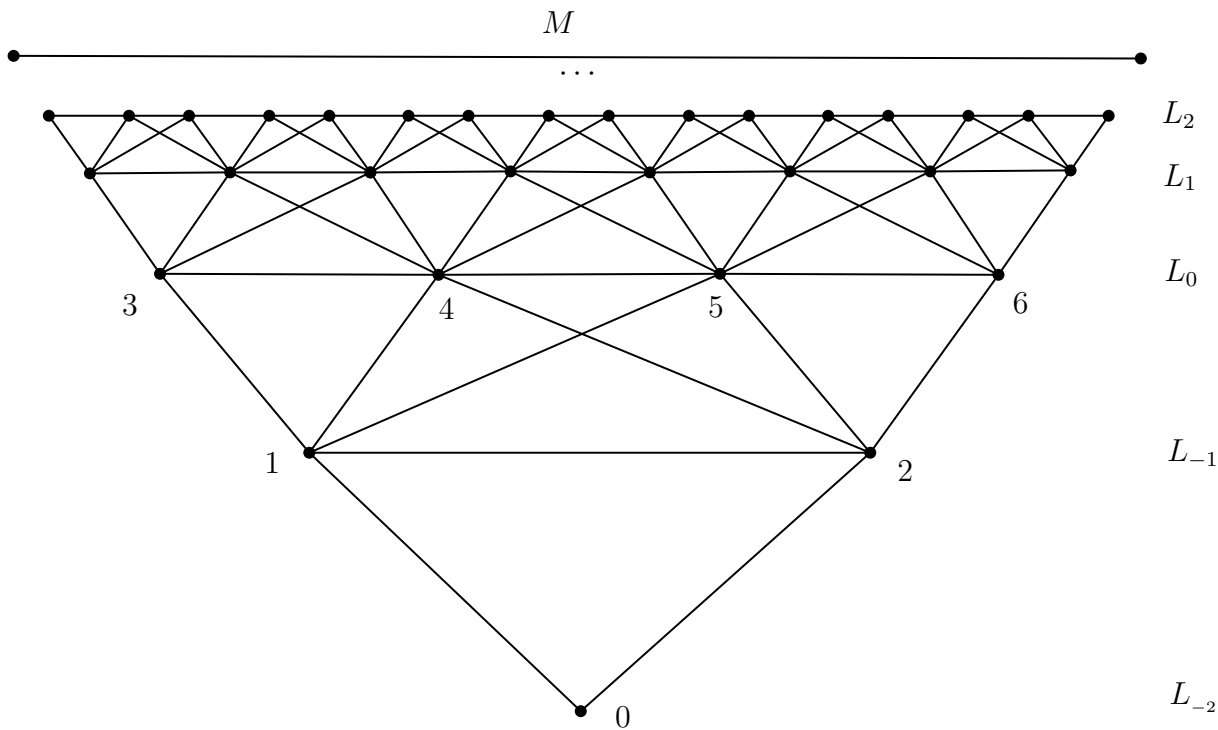
\begin{figure}

\tikzset{every picture/.style={line width=0.75pt}} 

\begin{tikzpicture}[x=0.75pt,y=0.75pt,yscale=-1,xscale=1]

\draw  [fill={rgb, 255:red, 0; green, 0; blue, 0 }  ,fill opacity=1 ] (103.67,60.67) .. controls (103.67,59.29) and (104.79,58.17) .. (106.17,58.17) .. controls (107.55,58.17) and (108.67,59.29) .. (108.67,60.67) .. controls (108.67,62.05) and (107.55,63.17) .. (106.17,63.17) .. controls (104.79,63.17) and (103.67,62.05) .. (103.67,60.67) -- cycle ;
\draw  [fill={rgb, 255:red, 0; green, 0; blue, 0 }  ,fill opacity=1 ] (133.67,60.67) .. controls (133.67,59.29) and (134.79,58.17) .. (136.17,58.17) .. controls (137.55,58.17) and (138.67,59.29) .. (138.67,60.67) .. controls (138.67,62.05) and (137.55,63.17) .. (136.17,63.17) .. controls (134.79,63.17) and (133.67,62.05) .. (133.67,60.67) -- cycle ;
\draw    (108.67,60.67) -- (133.67,60.67) ;
\draw  [fill={rgb, 255:red, 0; green, 0; blue, 0 }  ,fill opacity=1 ] (173.67,60.67) .. controls (173.67,59.29) and (174.79,58.17) .. (176.17,58.17) .. controls (177.55,58.17) and (178.67,59.29) .. (178.67,60.67) .. controls (178.67,62.05) and (177.55,63.17) .. (176.17,63.17) .. controls (174.79,63.17) and (173.67,62.05) .. (173.67,60.67) -- cycle ;
\draw    (178.67,60.67) -- (203.67,60.67) ;
\draw    (138.67,60.67) -- (173.67,60.67) ;
\draw  [fill={rgb, 255:red, 0; green, 0; blue, 0 }  ,fill opacity=1 ] (243.17,60.67) .. controls (243.17,59.29) and (244.29,58.17) .. (245.67,58.17) .. controls (247.05,58.17) and (248.17,59.29) .. (248.17,60.67) .. controls (248.17,62.05) and (247.05,63.17) .. (245.67,63.17) .. controls (244.29,63.17) and (243.17,62.05) .. (243.17,60.67) -- cycle ;
\draw  [fill={rgb, 255:red, 0; green, 0; blue, 0 }  ,fill opacity=1 ] (273.17,60.67) .. controls (273.17,59.29) and (274.29,58.17) .. (275.67,58.17) .. controls (277.05,58.17) and (278.17,59.29) .. (278.17,60.67) .. controls (278.17,62.05) and (277.05,63.17) .. (275.67,63.17) .. controls (274.29,63.17) and (273.17,62.05) .. (273.17,60.67) -- cycle ;
\draw    (248.17,60.67) -- (273.17,60.67) ;
\draw    (208.17,60.67) -- (243.17,60.67) ;
\draw  [fill={rgb, 255:red, 0; green, 0; blue, 0 }  ,fill opacity=1 ] (313.17,60.67) .. controls (313.17,59.29) and (314.29,58.17) .. (315.67,58.17) .. controls (317.05,58.17) and (318.17,59.29) .. (318.17,60.67) .. controls (318.17,62.05) and (317.05,63.17) .. (315.67,63.17) .. controls (314.29,63.17) and (313.17,62.05) .. (313.17,60.67) -- cycle ;
\draw  [fill={rgb, 255:red, 0; green, 0; blue, 0 }  ,fill opacity=1 ] (343.17,60.67) .. controls (343.17,59.29) and (344.29,58.17) .. (345.67,58.17) .. controls (347.05,58.17) and (348.17,59.29) .. (348.17,60.67) .. controls (348.17,62.05) and (347.05,63.17) .. (345.67,63.17) .. controls (344.29,63.17) and (343.17,62.05) .. (343.17,60.67) -- cycle ;
\draw    (318.17,60.67) -- (343.17,60.67) ;
\draw    (278.17,60.67) -- (313.17,60.67) ;
\draw  [fill={rgb, 255:red, 0; green, 0; blue, 0 }  ,fill opacity=1 ] (383.17,60.67) .. controls (383.17,59.29) and (384.29,58.17) .. (385.67,58.17) .. controls (387.05,58.17) and (388.17,59.29) .. (388.17,60.67) .. controls (388.17,62.05) and (387.05,63.17) .. (385.67,63.17) .. controls (384.29,63.17) and (383.17,62.05) .. (383.17,60.67) -- cycle ;
\draw  [fill={rgb, 255:red, 0; green, 0; blue, 0 }  ,fill opacity=1 ] (413.17,60.67) .. controls (413.17,59.29) and (414.29,58.17) .. (415.67,58.17) .. controls (417.05,58.17) and (418.17,59.29) .. (418.17,60.67) .. controls (418.17,62.05) and (417.05,63.17) .. (415.67,63.17) .. controls (414.29,63.17) and (413.17,62.05) .. (413.17,60.67) -- cycle ;
\draw    (388.17,60.67) -- (413.17,60.67) ;
\draw    (348.17,60.67) -- (383.17,60.67) ;
\draw  [fill={rgb, 255:red, 0; green, 0; blue, 0 }  ,fill opacity=1 ] (452.67,60.67) .. controls (452.67,59.29) and (453.79,58.17) .. (455.17,58.17) .. controls (456.55,58.17) and (457.67,59.29) .. (457.67,60.67) .. controls (457.67,62.05) and (456.55,63.17) .. (455.17,63.17) .. controls (453.79,63.17) and (452.67,62.05) .. (452.67,60.67) -- cycle ;
\draw  [fill={rgb, 255:red, 0; green, 0; blue, 0 }  ,fill opacity=1 ] (482.67,60.67) .. controls (482.67,59.29) and (483.79,58.17) .. (485.17,58.17) .. controls (486.55,58.17) and (487.67,59.29) .. (487.67,60.67) .. controls (487.67,62.05) and (486.55,63.17) .. (485.17,63.17) .. controls (483.79,63.17) and (482.67,62.05) .. (482.67,60.67) -- cycle ;
\draw    (457.67,60.67) -- (482.67,60.67) ;
\draw    (417.67,60.67) -- (452.67,60.67) ;
\draw  [fill={rgb, 255:red, 0; green, 0; blue, 0 }  ,fill opacity=1 ] (522.67,60.67) .. controls (522.67,59.29) and (523.79,58.17) .. (525.17,58.17) .. controls (526.55,58.17) and (527.67,59.29) .. (527.67,60.67) .. controls (527.67,62.05) and (526.55,63.17) .. (525.17,63.17) .. controls (523.79,63.17) and (522.67,62.05) .. (522.67,60.67) -- cycle ;
\draw  [fill={rgb, 255:red, 0; green, 0; blue, 0 }  ,fill opacity=1 ] (552.67,60.67) .. controls (552.67,59.29) and (553.79,58.17) .. (555.17,58.17) .. controls (556.55,58.17) and (557.67,59.29) .. (557.67,60.67) .. controls (557.67,62.05) and (556.55,63.17) .. (555.17,63.17) .. controls (553.79,63.17) and (552.67,62.05) .. (552.67,60.67) -- cycle ;
\draw    (527.67,60.67) -- (552.67,60.67) ;
\draw    (487.67,60.67) -- (522.67,60.67) ;
\draw    (68.67,60.67) -- (103.67,60.67) ;
\draw    (557.67,60.67) -- (592.67,60.67) ;
\draw  [fill={rgb, 255:red, 0; green, 0; blue, 0 }  ,fill opacity=1 ] (63.67,60.67) .. controls (63.67,59.29) and (64.79,58.17) .. (66.17,58.17) .. controls (67.55,58.17) and (68.67,59.29) .. (68.67,60.67) .. controls (68.67,62.05) and (67.55,63.17) .. (66.17,63.17) .. controls (64.79,63.17) and (63.67,62.05) .. (63.67,60.67) -- cycle ;
\draw  [fill={rgb, 255:red, 0; green, 0; blue, 0 }  ,fill opacity=1 ] (592.67,60.67) .. controls (592.67,59.29) and (593.79,58.17) .. (595.17,58.17) .. controls (596.55,58.17) and (597.67,59.29) .. (597.67,60.67) .. controls (597.67,62.05) and (596.55,63.17) .. (595.17,63.17) .. controls (593.79,63.17) and (592.67,62.05) .. (592.67,60.67) -- cycle ;
\draw  [fill={rgb, 255:red, 0; green, 0; blue, 0 }  ,fill opacity=1 ] (84.17,89.67) .. controls (84.17,88.29) and (85.29,87.17) .. (86.67,87.17) .. controls (88.05,87.17) and (89.17,88.29) .. (89.17,89.67) .. controls (89.17,91.05) and (88.05,92.17) .. (86.67,92.17) .. controls (85.29,92.17) and (84.17,91.05) .. (84.17,89.67) -- cycle ;
\draw  [color={rgb, 255:red, 0; green, 0; blue, 0 }  ,draw opacity=1 ][fill={rgb, 255:red, 0; green, 0; blue, 0 }  ,fill opacity=1 ] (154.17,89.17) .. controls (154.17,87.79) and (155.29,86.67) .. (156.67,86.67) .. controls (158.05,86.67) and (159.17,87.79) .. (159.17,89.17) .. controls (159.17,90.55) and (158.05,91.67) .. (156.67,91.67) .. controls (155.29,91.67) and (154.17,90.55) .. (154.17,89.17) -- cycle ;
\draw  [fill={rgb, 255:red, 0; green, 0; blue, 0 }  ,fill opacity=1 ] (224.17,89.17) .. controls (224.17,87.79) and (225.29,86.67) .. (226.67,86.67) .. controls (228.05,86.67) and (229.17,87.79) .. (229.17,89.17) .. controls (229.17,90.55) and (228.05,91.67) .. (226.67,91.67) .. controls (225.29,91.67) and (224.17,90.55) .. (224.17,89.17) -- cycle ;
\draw  [fill={rgb, 255:red, 0; green, 0; blue, 0 }  ,fill opacity=1 ] (294.17,88.67) .. controls (294.17,87.29) and (295.29,86.17) .. (296.67,86.17) .. controls (298.05,86.17) and (299.17,87.29) .. (299.17,88.67) .. controls (299.17,90.05) and (298.05,91.17) .. (296.67,91.17) .. controls (295.29,91.17) and (294.17,90.05) .. (294.17,88.67) -- cycle ;
\draw  [fill={rgb, 255:red, 0; green, 0; blue, 0 }  ,fill opacity=1 ] (363.67,89.17) .. controls (363.67,87.79) and (364.79,86.67) .. (366.17,86.67) .. controls (367.55,86.67) and (368.67,87.79) .. (368.67,89.17) .. controls (368.67,90.55) and (367.55,91.67) .. (366.17,91.67) .. controls (364.79,91.67) and (363.67,90.55) .. (363.67,89.17) -- cycle ;
\draw  [fill={rgb, 255:red, 0; green, 0; blue, 0 }  ,fill opacity=1 ] (433.67,88.67) .. controls (433.67,87.29) and (434.79,86.17) .. (436.17,86.17) .. controls (437.55,86.17) and (438.67,87.29) .. (438.67,88.67) .. controls (438.67,90.05) and (437.55,91.17) .. (436.17,91.17) .. controls (434.79,91.17) and (433.67,90.05) .. (433.67,88.67) -- cycle ;
\draw  [fill={rgb, 255:red, 0; green, 0; blue, 0 }  ,fill opacity=1 ] (503.67,88.67) .. controls (503.67,87.29) and (504.79,86.17) .. (506.17,86.17) .. controls (507.55,86.17) and (508.67,87.29) .. (508.67,88.67) .. controls (508.67,90.05) and (507.55,91.17) .. (506.17,91.17) .. controls (504.79,91.17) and (503.67,90.05) .. (503.67,88.67) -- cycle ;
\draw  [fill={rgb, 255:red, 0; green, 0; blue, 0 }  ,fill opacity=1 ] (573.67,88.17) .. controls (573.67,86.79) and (574.79,85.67) .. (576.17,85.67) .. controls (577.55,85.67) and (578.67,86.79) .. (578.67,88.17) .. controls (578.67,89.55) and (577.55,90.67) .. (576.17,90.67) .. controls (574.79,90.67) and (573.67,89.55) .. (573.67,88.17) -- cycle ;
\draw    (89.17,89.67) -- (154.17,89.17) ;
\draw    (159.17,89.17) -- (224.17,89.17) ;
\draw    (229.17,89.17) -- (294.17,88.67) ;
\draw    (299.17,88.67) -- (363.67,89.17) ;
\draw    (368.67,89.17) -- (433.67,88.67) ;
\draw    (438.67,88.67) -- (503.67,88.67) ;
\draw    (508.67,88.67) -- (573.67,88.17) ;
\draw    (66.17,60.67) -- (86.67,89.67) ;
\draw    (106.17,60.67) -- (86.67,89.67) ;
\draw    (136.17,60.67) -- (156.67,89.17) ;
\draw    (106.17,60.67) -- (156.67,89.17) ;
\draw    (136.17,60.67) -- (86.67,89.67) ;
\draw    (176.17,60.67) -- (156.67,89.17) ;
\draw    (176.17,60.67) -- (226.67,89.17) ;
\draw    (226.67,89.17) -- (245.67,60.67) ;
\draw    (275.67,60.67) -- (296.67,88.67) ;
\draw    (226.67,89.17) -- (275.67,60.67) ;
\draw    (296.67,88.67) -- (245.67,60.67) ;
\draw    (315.67,60.67) -- (296.67,88.67) ;
\draw    (345.67,60.67) -- (366.17,89.17) ;
\draw    (345.67,60.67) -- (296.67,88.67) ;
\draw    (315.67,60.67) -- (366.17,89.17) ;
\draw    (385.67,60.67) -- (366.17,89.17) ;
\draw    (415.67,60.67) -- (436.17,88.67) ;
\draw    (385.67,60.67) -- (436.17,88.67) ;
\draw    (415.67,60.67) -- (366.17,89.17) ;
\draw    (455.17,60.67) -- (436.17,88.67) ;
\draw    (455.17,60.67) -- (506.17,88.67) ;
\draw    (485.17,60.67) -- (436.17,88.67) ;
\draw    (485.17,60.67) -- (506.17,88.67) ;
\draw    (525.17,60.67) -- (506.17,88.67) ;
\draw    (525.17,60.67) -- (576.17,88.17) ;
\draw    (555.17,60.67) -- (506.17,88.67) ;
\draw    (555.17,60.67) -- (576.17,88.17) ;
\draw    (595.17,60.67) -- (576.17,88.17) ;
\draw  [fill={rgb, 255:red, 0; green, 0; blue, 0 }  ,fill opacity=1 ] (119.17,140.17) .. controls (119.17,138.79) and (120.29,137.67) .. (121.67,137.67) .. controls (123.05,137.67) and (124.17,138.79) .. (124.17,140.17) .. controls (124.17,141.55) and (123.05,142.67) .. (121.67,142.67) .. controls (120.29,142.67) and (119.17,141.55) .. (119.17,140.17) -- cycle ;
\draw  [fill={rgb, 255:red, 0; green, 0; blue, 0 }  ,fill opacity=1 ] (258.17,140.67) .. controls (258.17,139.29) and (259.29,138.17) .. (260.67,138.17) .. controls (262.05,138.17) and (263.17,139.29) .. (263.17,140.67) .. controls (263.17,142.05) and (262.05,143.17) .. (260.67,143.17) .. controls (259.29,143.17) and (258.17,142.05) .. (258.17,140.67) -- cycle ;
\draw  [fill={rgb, 255:red, 0; green, 0; blue, 0 }  ,fill opacity=1 ] (398.67,140.17) .. controls (398.67,138.79) and (399.79,137.67) .. (401.17,137.67) .. controls (402.55,137.67) and (403.67,138.79) .. (403.67,140.17) .. controls (403.67,141.55) and (402.55,142.67) .. (401.17,142.67) .. controls (399.79,142.67) and (398.67,141.55) .. (398.67,140.17) -- cycle ;
\draw  [fill={rgb, 255:red, 0; green, 0; blue, 0 }  ,fill opacity=1 ] (537.67,140.67) .. controls (537.67,139.29) and (538.79,138.17) .. (540.17,138.17) .. controls (541.55,138.17) and (542.67,139.29) .. (542.67,140.67) .. controls (542.67,142.05) and (541.55,143.17) .. (540.17,143.17) .. controls (538.79,143.17) and (537.67,142.05) .. (537.67,140.67) -- cycle ;
\draw    (86.67,89.67) -- (121.67,140.17) ;
\draw    (156.67,89.17) -- (121.67,140.17) ;
\draw    (156.67,89.17) -- (260.67,140.67) ;
\draw    (121.67,140.17) -- (260.67,140.67) ;
\draw    (121.67,140.17) -- (226.67,89.17) ;
\draw    (260.67,140.67) -- (226.67,89.17) ;
\draw    (401.17,140.17) -- (540.17,140.67) ;
\draw    (260.67,140.67) -- (401.17,140.17) ;
\draw    (260.67,140.67) -- (296.67,88.67) ;
\draw    (260.67,140.67) -- (366.17,89.17) ;
\draw    (401.17,140.17) -- (296.67,88.67) ;
\draw    (401.17,140.17) -- (366.17,89.17) ;
\draw    (401.17,140.17) -- (436.17,88.67) ;
\draw    (401.17,140.17) -- (506.17,88.67) ;
\draw    (540.17,140.67) -- (436.17,88.67) ;
\draw    (540.17,140.67) -- (576.17,88.17) ;
\draw    (540.17,140.67) -- (506.17,88.67) ;
\draw  [fill={rgb, 255:red, 0; green, 0; blue, 0 }  ,fill opacity=1 ] (193.67,230.17) .. controls (193.67,228.79) and (194.79,227.67) .. (196.17,227.67) .. controls (197.55,227.67) and (198.67,228.79) .. (198.67,230.17) .. controls (198.67,231.55) and (197.55,232.67) .. (196.17,232.67) .. controls (194.79,232.67) and (193.67,231.55) .. (193.67,230.17) -- cycle ;
\draw  [fill={rgb, 255:red, 0; green, 0; blue, 0 }  ,fill opacity=1 ] (473.67,230.17) .. controls (473.67,228.79) and (474.79,227.67) .. (476.17,227.67) .. controls (477.55,227.67) and (478.67,228.79) .. (478.67,230.17) .. controls (478.67,231.55) and (477.55,232.67) .. (476.17,232.67) .. controls (474.79,232.67) and (473.67,231.55) .. (473.67,230.17) -- cycle ;
\draw    (121.67,140.17) -- (196.17,230.17) ;
\draw    (260.67,140.67) -- (196.17,230.17) ;
\draw    (260.67,140.67) -- (476.17,230.17) ;
\draw    (401.17,140.17) -- (476.17,230.17) ;
\draw    (401.17,140.17) -- (196.17,230.17) ;
\draw    (540.17,140.67) -- (476.17,230.17) ;
\draw    (196.17,230.17) -- (476.17,230.17) ;
\draw  [fill={rgb, 255:red, 0; green, 0; blue, 0 }  ,fill opacity=1 ] (329.17,360.17) .. controls (329.17,358.79) and (330.29,357.67) .. (331.67,357.67) .. controls (333.05,357.67) and (334.17,358.79) .. (334.17,360.17) .. controls (334.17,361.55) and (333.05,362.67) .. (331.67,362.67) .. controls (330.29,362.67) and (329.17,361.55) .. (329.17,360.17) -- cycle ;
\draw    (196.17,230.17) -- (331.67,360.17) ;
\draw    (476.17,230.17) -- (331.67,360.17) ;
\draw [color={rgb, 255:red, 0; green, 0; blue, 0 }  ,draw opacity=1 ]   (156.67,89.17) -- (206.17,60.67) ;
\draw  [color={rgb, 255:red, 0; green, 0; blue, 0 }  ,draw opacity=1 ][fill={rgb, 255:red, 0; green, 0; blue, 0 }  ,fill opacity=1 ] (203.67,60.67) .. controls (203.67,59.29) and (204.79,58.17) .. (206.17,58.17) .. controls (207.55,58.17) and (208.67,59.29) .. (208.67,60.67) .. controls (208.67,62.05) and (207.55,63.17) .. (206.17,63.17) .. controls (204.79,63.17) and (203.67,62.05) .. (203.67,60.67) -- cycle ;
\draw  [color={rgb, 255:red, 0; green, 0; blue, 0 }  ,draw opacity=1 ][fill={rgb, 255:red, 0; green, 0; blue, 0 }  ,fill opacity=1 ] (224.17,89.17) .. controls (224.17,87.79) and (225.29,86.67) .. (226.67,86.67) .. controls (228.05,86.67) and (229.17,87.79) .. (229.17,89.17) .. controls (229.17,90.55) and (228.05,91.67) .. (226.67,91.67) .. controls (225.29,91.67) and (224.17,90.55) .. (224.17,89.17) -- cycle ;
\draw [color={rgb, 255:red, 0; green, 0; blue, 0 }  ,draw opacity=1 ]   (206.17,60.67) -- (226.67,89.17) ;
\draw    (48.57,30.54) -- (611.23,31.87) ;
\draw  [fill={rgb, 255:red, 0; green, 0; blue, 0 }  ,fill opacity=1 ] (46.07,30.54) .. controls (46.07,29.16) and (47.19,28.04) .. (48.57,28.04) .. controls (49.95,28.04) and (51.07,29.16) .. (51.07,30.54) .. controls (51.07,31.92) and (49.95,33.04) .. (48.57,33.04) .. controls (47.19,33.04) and (46.07,31.92) .. (46.07,30.54) -- cycle ;
\draw  [fill={rgb, 255:red, 0; green, 0; blue, 0 }  ,fill opacity=1 ] (608.73,31.87) .. controls (608.73,30.49) and (609.85,29.37) .. (611.23,29.37) .. controls (612.61,29.37) and (613.73,30.49) .. (613.73,31.87) .. controls (613.73,33.25) and (612.61,34.37) .. (611.23,34.37) .. controls (609.85,34.37) and (608.73,33.25) .. (608.73,31.87) -- cycle ;

\draw (347.33,356.23) node [anchor=north west][inner sep=0.75pt]  [font=\normalsize]  {$0$};
\draw (172.33,228.73) node [anchor=north west][inner sep=0.75pt]    {$1$};
\draw (488.33,232.73) node [anchor=north west][inner sep=0.75pt]    {$2$};
\draw (101.33,149.23) node [anchor=north west][inner sep=0.75pt]    {$3$};
\draw (259.33,151.23) node [anchor=north west][inner sep=0.75pt]    {$4$};
\draw (393.83,150.23) node [anchor=north west][inner sep=0.75pt]    {$5$};
\draw (545.83,148.23) node [anchor=north west][inner sep=0.75pt]    {$6$};
\draw (310.67,6.73) node [anchor=north west][inner sep=0.75pt]    {$M$};
\draw (319.17,34.4) node [anchor=north west][inner sep=0.75pt]    {$\cdots $};
\draw (624,344.07) node [anchor=north west][inner sep=0.75pt]    {$L_{_{-2}}$};
\draw (622.86,222.07) node [anchor=north west][inner sep=0.75pt]    {$L_{-1}$};
\draw (621.53,131.4) node [anchor=north west][inner sep=0.75pt]    {$L_{0}$};
\draw (620.86,83.4) node [anchor=north west][inner sep=0.75pt]    {$L_{1}$};
\draw (620.19,52.07) node [anchor=north west][inner sep=0.75pt]    {$L_{2}$};

\end{tikzpicture}
    \caption{A representation of a Hyperbolic approximation $G_k$ of a metric space $M$ for $k=2$. A possible order for $G_k$ is also represented.}
    \label{figHyp1}
\end{figure}

We shall see in the proof of Lemma \ref{reprmu} that if $M$ is bounded then for $k > \log_{\frac{1}{r}}(\text{diam}(M ))$ the graph $G_k$ is connected. Thus, for such  $k$, if we consider the family $(G_k)$ as weighted graphs with weight function $w\equiv1$, then we may define the geodesic graph metric spaces $M_k=M(G_k)$ for every $k\in\N$. In addition, if $M$ is compact then  $G_k$ is finite as a graph and thus $M_k$ is a finite metric space. We say that $(M_k)_{k\in\N}$ is the \textit{sequence of} $(\lambda,r)$\textit{-Finite Hyperbolic approximations of} $M$ \textit{with layers} $(L_i)_i$.

In this section we are going to prove the following result.

\begin{theorem}\label{theo:hypdoub}
    Let $\lambda\geq2$ and $r\in(0,1/6)$. If $M$ is a compact and doubling metric space with doubling constant $D>0$ and $(M_k)$ is any sequence of $(\lambda,r)$-finite hyperbolic approximations of $M$, then for $k>\log_{\frac{1}{r}}\text{diam}(M)$,
    $$d_1(M_k)\leq sd_1(M_k)\leq1+2D\lambda^{\log_2D}.$$
\end{theorem}

Theorem \ref{theo:hypdoub} is deduced from Proposition \ref{condoubling} and the more general Theorem \ref{thhyper} that will be stated and proved below.

Let us first introduce a property on $M$ that will be connected to the $\ell_1^N$-distortion of the finite hyperbolic approximations of $M$.

\begin{definition}\label{homodef}
    Given $\lambda>1$, $r\in(0,1/3)$ and $C>1$, a metric space $M$ is $(\lambda,r,C)$-\textit{homogeneous} if for every $\varepsilon>0$, every $\varepsilon$-net $N$ of $M$ and every $x\in M$ we have
    $$\big|B_{N}(x,\lambda \varepsilon)\big|\leq C\big|B_{N}(x,\lambda \varepsilon(1-3r))\big|<\infty.$$
\end{definition}

Our aim in this subsection is to prove the following theorem.

\begin{theorem}\label{thhyper}
        Let $\lambda\geq2$, $r\in(0,1/6)$ and $C>1$ be given. Let also $M$ be a compact metric space and $(M_k)$ be any sequence of $(\lambda,r)$-finite hyperbolic approximations of $M$. If $M$ is $(\lambda,r,C)$-homogeneous then for every $k>\log_{\frac{1}{r}}(\text{diam}(M))$,
    $$d_1(M_k)\leq 1+2C.$$
\end{theorem}

The most remarkable part of the upper bound displayed in Theorem \ref{thhyper} is the fact that it is independent of the amount of layers, i.e., the amount of points of the finite hyperbolic approximation.

It is important to notice that there is a close relationship between the notion given in Definition \ref{homodef} and the well-known doubling constant of metric spaces. In fact, the following result gives a quantitative connection between both constants. Let us remind the reader that a metric space is doubling with constant $D$ if for every $R>0$, every ball of radius $R>0$ can be covered by $D$ many balls of radius $R/2$.

\begin{proposition}\label{condoubling}
    Let $M$ be a doubling metric space whose doubling constant is $D\geq1$. Then, for every $\lambda\geq2$ and $r\in(0,1/6)$ the space $M$ is $(\lambda,r,D\lambda^{\log_2 D})$-homogeneous.
\end{proposition}

\begin{proof}
    Let us consider $\ep>0$, an $\varepsilon$-net $N\subset M$ and $x\in M$. Then, it is straightforward from Lemma 2.3 in \cite{Hyt10} that
    $$\big|B_{N}(x,\lambda \varepsilon)\big|\leq D\lambda^{\log_2 D}.$$
    We are therefore done since $\big|B_{N}(x,\lambda \varepsilon(1-3r))\big|\geq1$.
\end{proof}
From now on we will assume that $M$ is compact and, therefore, bounded. This subsection is devoted to the proof of Theorem \ref{thhyper} and, thus, we fix a sequence $(M_k)_k$ of $(\lambda,r)$-hyperbolic approximations of $M$ with layers $(L_i)_{i\in\mathbb{Z}}$. We also fix $k>\log_{\frac{1}{r}}(\text{diam}(M))$ so that the layer $L_{-k}$ consists of just one point which will be distinguished from now on as $0\in M$.

Let us now construct the normalised stochastic basis that we will be working with: we pick $F\in\Sigma (M_k)$ such that if $x\in L_i$ and $y\in L_j$ with $i<j$ then $F^{-1}(x)<F^{-1}(y)$ (as in Figure \ref{figHyp1}). Again, throughout the rest of this subsection we are going to identify $M_k$ with $\{0,\dots,N=|M_k|-1\}$ so that $F(n)=n$ for every $n\leq N$. We then define the basic vectors in $\mathcal{F}(M_k)$ in the following way. For every $x\in L_{i+1}$ we denote the set of radial neighbours of $x$ from $L_i$ by $N(x)$ and put
$$b_x=\delta_x-\frac{1}{|N(x)|}\sum_{y\in N(x)}\delta_y.$$

\begin{fact}
    The set $\beta_k:=(b_x)_{x\in M_k\setminus\{0\}}$ constitutes a normalised stochastic basis in $\mathcal{F}(M_k)$.
    
    Indeed, for any $x\in L_{i+1}$, $b_x$ is well-defined as $|N(x)|>0$ (see the proof of Lemma \ref{reprmu}).
\end{fact}

\textbf{Notation.} We will denote the horizontal edges of $M_k$ in the layer $L_i$ for $i=1,\dots,k$ as $E(L_i)$. In fact, we denote the graph with vertices $L_i$ and edges $E(L_i)$ also as $L_i$.

\begin{lemma}\label{reprmu}
    Consider $\rho\in \R$, $i\in\{-k,\dots,k-1\}$, $\widetilde e\in E(L_{i+1})$ and the element $\mu\in\mathcal{F}(L_i)$ given by
    $$\mu(n)=\sum_{\substack{v\in V(\widetilde e)\\\{v,n\}\in E(G_k)}}\frac{1}{|N(v)|}\rho(\delta_{\max V(\widetilde e)}-\delta_{\min V(\widetilde e)})(v)\;\;\;\;\;\;\;\;\;(n\in L_i).$$
    Then, $supp(\mu)$ is a complete subgraph of $L_i$. Moreover, if $M$ is $(\lambda,r,C)$-homogeneous for $C>1$ then the total variation of $\mu$ satisfies $|\mu|\leq |\rho|(2-2/C)$.
\end{lemma}
\begin{proof}
    We first show that $supp(\mu)$ is a complete subgraph of $L_i$. Indeed, if $n,m\in L_i$ such that $\mu(n), \mu(m)\neq 0$ then there must be $v_n,v_m\in V(\widetilde e)$ such that $\{v_n,n\}$ and $\{v_m,m\}$ are radial edges of $G_k$. Therefore,
    $$B_M(n,\lambda r^i)\cap B_M(m,\lambda r^i)\supset B_M(v_n,\lambda r^{i+1})\cap B_M(v_m,\lambda r^{i+1})\neq \emptyset.$$
    and thus $\{n,m\}$ is a horizontal edge of $L_i$.

    Now, let us assume that $M_k$ is $(\lambda,r,C)$-homogeneous and prove that $|\mu|\leq |\rho|(2-2/C)$. For that purpose, we call $x=\max V(\widetilde e)$ and $y=\min V(\widetilde e)$ and assume for now that $|N(x)|\geq|N(y)|$. Then,
    $$\begin{aligned}|\mu|=&\sum_{n\in L_i}|\mu(n)|=\sum_{n\in L_i}\bigg|\sum_{\substack{v\in V(\widetilde e)\\\{n,v\}\in E(G_k)}}\frac{\rho}{|N(v)|}(\delta_x-\delta_y)(v)\bigg|\\=&|\rho|\sum_{n\in N(x)\cup N(y)}\bigg|\sum_{\substack{v\in V(\widetilde e)\\\{n,v\}\in E(G_k)}}\frac{1}{|N(v)|}(\delta_x-\delta_y)(v)\bigg|\\=&|\rho|\bigg(\sum_{n\in N(x)\setminus N(y)}\frac{1}{|N(x)|}+\sum_{n\in N(y)\setminus N(x)}\frac{1}{|N(y)|}+\sum_{n\in N(x)\cap N(y)}\frac{1}{|N(y)|}-\frac{1}{|N(x)|}\bigg)\\=&|\rho|\bigg(1+\sum_{n\in N(x)\setminus N(y)}\frac{1}{|N(x)|}-\sum_{n\in N(x)\cap N(y)}\frac{1}{|N(x)|}\bigg)\\=&|\rho|\Big(1+\frac{|N(x)\setminus N(y)|-|N(x)\cap N(y)|}{|N(x)|}\Big).\end{aligned}$$
    Therefore, we will be done if we manage to prove that $|N(x)\setminus N(y)|\leq (1-2/C)|N(x)|+|N(x)\cap N(y)|$. For that purpose, since $M_k$ is $(\lambda,r,C)$-homogeneous, it is enough to show
    \begin{equation}\label{eqcont1}
        B_{L_i}(x,\lambda r^i(1-r))=N(x),
    \end{equation}
    and
    \begin{equation}\label{eqcont15}
        B_{L_i}(x,\lambda r^i(1-3r))\subset N(y),
    \end{equation}
    Indeed, if \eqref{eqcont1} and \eqref{eqcont15} are satisfied then
    $$\begin{aligned}
        |N(x)\setminus N(y)|\leq&|B_{L_i}(x,\lambda r^i(1-r))\setminus B_{L_i}(x,\lambda r^i(1-3r))|\\=&(1-2/C)|B_{L_i}(x,\lambda r^i(1-r))|+2/C|B_{L_i}(x,\lambda r^i(1-r))|-|B_{L_i}(x,\lambda r^i(1-3r))|\\\stackrel{\text{Def }\ref{homodef}}{\leq}& (1-2/C)|B_{L_i}(x,\lambda r^i(1-r))|+2|B_{L_i}(x,\lambda r^i(1-3r))|-|B_{L_i}(x,\lambda r^i(1-3r))|\\=&(1-2/C)|B_{L_i}(x,\lambda r^i(1-r))|+|B_{L_i}(x,\lambda r^i(1-3r))|\\\leq&(1-2/C))|N(x)|+|N(x)\cap N(y)|.
    \end{aligned}$$
    
    We first show \eqref{eqcont1}. If $n\in N(x)$ then $B_M(x,\lambda r^{i+1})\subset B_M(n,\lambda r^i)$ and thus $d(n,x)\leq \lambda r^i-\lambda r^{i+1}$ meaning that $n\in B_{L_i}(x,\lambda r^i(1-r))$. Now, if $n\in B_{L_i}(x,\lambda r^i(1-r))$ then clearly $x\in B_M(n,\lambda r^i(1-r))$ and hence $B_M(x,\lambda r^{i+1})\subset B_M(n,\lambda r^i)$ meaning that $n\in N(x)$.

    We then show \eqref{eqcont15}. If $n\notin N(y)$ then $B_M(y,\lambda r^{i+1})\not\subset B_M(n,\lambda r^i)$ meaning that there is $y_0\in B_M(y,\lambda r^{i+1})$ such that $d(y_0,n)>\lambda r^i$. Since $\{x,y\}\in E(G_k)$ is a horizontal edge in $L_{i+1}$ we have that $d(x,y)\leq 2\lambda r^{i+1}$ and therefore
    $$d(n,x)\geq d(n,y_0)-d(y_0,y)-d(y,x)>\lambda r^i-3\lambda r^{i+1}.$$
    This finishes the proof of \eqref{eqcont15} since the latter inequality yields $n\notin B_{L_i}(x,\lambda r^i(1-3r))$.


    If $|N(y)|\geq|N(x)|$ the same argument works interchanging the roles of $x$ and $y$.
\end{proof}

\begin{proof}[Proof of Theorem \ref{thhyper}]
    We are going to prove that $d_1(\beta_k)\leq1+2C$. We may denote $b^*_x$ the dual vectors of $\beta_k$ for $x\in M_k\setminus\{0\}$ as well as $\beta:=\beta_k$.
    
    Firstly, we easily see that by the triangle inequality
    $$\|b_x\|\leq\frac{1}{N(x)}\sum_{y\in N(x)}\|\delta_x-\delta_y\|=1.$$
    In fact, $\|b_x\|=1$ just by Kantorovich duality evaluating in $f_x\in S_{Lip_0(M_k)}$ given by $f_x(z)=d(z,x)-d(0,x)$. Therefore for every $\mu\in\mathcal{F}(M_k)$ we have
    $$\sum_{n=1}^N|b^*_n(\mu)|\|b_n\|=\sum_{n=1}^N|b^*_n(\mu)|.$$

    We fix now $\{x,y\}\in E(G_k)$ with $x>y$. By the latter equality, it is enough to show that
    \begin{equation}\label{Hypclaim}
        \sum_{n=1}^N|b^*_n(\delta_x-\delta_y)|\leq 1+2C,
    \end{equation}
    Again, since $\{x,y\}$ is fixed we denote $\alpha(n):=b^*_n(\delta_x-\delta_y)$ for every $n\in M_k$ and consider $i_0\in\{-k,\dots,k\}$ such that $y\in L_{i_0}$. Let us define for every $e\in E(G_k)$ the function $\sigma_e:V(e)\to\{-1,1\}$ given by $\sigma_e(\max V(e))=1$ and $\sigma_e(\min V(e))=-1$.
    
    \textbf{Claim.} We claim that for every $-k\leq i\leq i_0$ there exists $F_i\subset E(L_i)$ and $\alpha_i:F_i\to\R$ satisfying
    \begin{enumerate}
        \item \label{inprop1} For every $n\in L_i$, 
        $$\alpha(n)=\sum_{\substack{e\in F_i\\n\in V(e)}}\sigma_e(n)\alpha_i(e).$$
        \item \label{inprop3} If $i<i_0$ then $$\sum_{e\in F_i}|\alpha_i(e)|\leq\Big(1-\frac{1}{C}\Big)\sum_{e\in F_{i+1}}|\alpha_{i+1}(e)|.$$
    \end{enumerate}
    Let us find the $F_i's$ and $\alpha_i's$ inductively in $i$ downwards. If $i=i_0$ we divide the proof in two cases, namely, whether $\{x,y\}$ is a horizontal or a radial edge. If $\{x,y\}$ is horizontal, then it is clear that $F_{i_0}=\{\{x,y\}\}$ and $\alpha_{i_0}(\{x,y\})=1$ since for $n\in L_{i_0}\setminus\{x,y\}$ we know that
    $$\alpha(n)\stackrel{\text{Th }\ref{theobasisvect}}{=}\sum_{\substack{v\in L_{i_0+1}\\\{v,n\}\in E(G_k)}}\frac{1}{|N(v)|}\alpha(v)=0.$$
    Where the last equality holds since $\alpha(v)=0$ for every $v>x$ by Lemma \ref{greatalpha}.

    Otherwise, if $\{x,y\}$ is radial then $x\in L_{i_0+1}$ and we choose $F_{i_0}=\{\{n,y\}\;:\; n\in N(x)\setminus\{y\}\}$ and $\alpha_{i_0}(e)=\sigma_e(y)\Big(\frac{\frac{1}{|N(x)|}-1}{|F_{i_0}|}\Big)$ for every $e\in F_{i_0}$ (note that in case $F_{i_0}=\emptyset$ then $\alpha_{i_0}$ is an empty mapping). Again by Theorem \ref{theobasisvect} and Lemma \ref{greatalpha} we know that $\alpha(n)=0$ for every $n\in L_{i_0+1}\setminus\{x\}$ and $\alpha(x)=1$. Therefore,
    $$\begin{aligned}\alpha(y)\stackrel{\text{Th }\ref{theobasisvect}}{=}&-1+\sum_{\substack{v\in L_{i_0+1}\\\{v,y\}\in E(G_k)}}\frac{1}{|N(v)|}\alpha(v)=-1+\frac{1}{|N(x)|}=\sum_{e\in F_{i_0}}\sigma_e(y)^2\bigg(\frac{\frac{1}{|N(x)|}-1}{|F_{i_0}|}\bigg)\\=&\sum_{\substack{e\in F_{i_0}\\y\in V(e)}}\sigma_e(y)\alpha_{i_0}(e).\end{aligned}$$
    Also, for $n\in L_i\setminus\{y\}$, since $|F_{i_0}|=|N(x)\setminus\{y\}|=|N(x)|-1$ we conclude 
    $$\begin{aligned}\sum_{\substack{e\in F_{i_0}\\n\in V(e)}}\sigma_e(n)\alpha_{i_0}(e)=&\sum_{\substack{e\in F_{i_0}\\n\in V(e)}}\sigma_e(n)\sigma_e(y)\bigg(\frac{\frac{1}{|N(x)|}-1}{|F_{i_0}|}\bigg)=\frac{1-\frac{1}{N(x)}}{|F_{i_0}|}=\frac{1}{|N(x)|}\\=&\sum_{\substack{v\in L_{i_0+1}\\\{v,n\}\in E(G_k)}}\frac{1}{|N(v)|}\alpha(v)\stackrel{\text{Th }\ref{theobasisvect}}{=}\alpha(n).\end{aligned}$$

    Now for the inductive step consider that $F_{i+1}$ and $\alpha_{i+1}$ are given with properties \eqref{inprop1} and \eqref{inprop3}, which is our induction hypothesis $(IH)$. Then for each $\widetilde e\in F_{i+1}$ we define $\mu_{\widetilde e}\in \mathcal{F}(L_i)$ as
    \begin{equation}\label{mutilde}\mu_{\widetilde e}(n)=\sum_{\substack{v\in V(\widetilde e)\\\{v,n\}\in E(G_k)}}\frac{1}{|N(v)|}\alpha_{i+1}(\widetilde e)\sigma_{\widetilde e}(v)\;\;\;\;\;\;\;\;\;(n\in L_i).\end{equation}
    We know by Lemma \ref{reprmu} that $supp(\mu_{\widetilde e})$ is a complete subgraph of $L_i$ and $|\mu_{\widetilde e}|\leq|\alpha_{i+1}(\widetilde e)|(2-2/C)$.
    Now, consider an optimal representation of $\mu_{\widetilde e}$ expressed as sum of molecules,
    \begin{equation}\label{mutilde2}\mu_{\widetilde e}=\sum_{e\in E(supp(\mu_{\widetilde e}))}\rho_{\widetilde e,e}(\delta_{\max V(e)}-\delta_{\min V(e)}),\end{equation}
    where being an optimal representation means that
    $$\|\mu_{\widetilde e}\|=\sum_{e\in E(supp(\mu_{\widetilde e}))}|\rho_{\widetilde e,e}|\,d(\max V(e),\min V(e)).$$
    The existence of such a representation is guaranteed by \cite[Corollary 2.5]{Sch23}. Finally, put $F_i=\bigcup_{\widetilde e\in F_{i+1}}E(supp(\mu_{\widetilde e}))$ and for every $e\in F_i$,
    \begin{equation}\label{alphai}\alpha_i(e)=\sum_{\substack{\widetilde e\in F_{i+1}\\e\in E(supp(\mu_{\widetilde e}))}}\rho_{\widetilde e,e}.\end{equation}
    Let us check then that property \eqref{inprop1} is satisfied for the latter choices of $\alpha_i$ and $F_i$. For every $n\in L_i$,
    $$\begin{aligned}
        \alpha(n)\stackrel{\text{Th }\ref{theobasisvect}}{=}&\sum_{\substack{v\in L_{i+1}\\\{v,n\}\in E(G_k)}}\frac{1}{|N(v)|}\alpha(v)\stackrel{(IH)}{=}\sum_{\substack{v\in L_{i+1}\\\{v,n\}\in E(G_k)}}\frac{1}{|N(v)|}\sum_{\substack{\widetilde e\in F_{i+1}\\v\in V(\widetilde e)}}\sigma_{\widetilde e}(v)\alpha_{i+1}(\widetilde e)\\=&\sum_{\widetilde e\in F_{i+1}}\sum_{\substack{v\in V(\widetilde e)\\\{v,n\}\in E(G_k)}}\frac{1}{|N(v)|}\sigma_{\widetilde e}(v)\alpha_{i+1}(\widetilde e)\stackrel{\eqref{mutilde}}{=}\sum_{\widetilde e\in F_{i+1}}\mu_{\widetilde e}(n)\\\stackrel{\eqref{mutilde2}}{=}&\sum_{\widetilde e\in F_{i+1}}\sum_{e\in E(supp(\mu_{\widetilde e}))}\rho_{\widetilde e,e}(\delta_{\max V(e)}-\delta_{\min V(e)})(n)=\sum_{\widetilde e\in F_{i+1}}\sum_{\substack{e\in E(supp(\mu_{\widetilde e}))\\n\in V(e)}}\rho_{\widetilde e,e}\sigma_e(n)\\=&\sum_{\substack{e\in F_i\\n\in V(e)}}\sum_{\substack{\widetilde e\in F_{i+1}\\e\in E(supp(\mu_{\widetilde e}))}}\rho_{\widetilde e,e}\sigma_e(n)=\sum_{\substack{e\in F_i\\n\in V(e)}}\sigma_e(n)\sum_{\substack{\widetilde e\in F_{i+1}\\e\in E(supp(\mu_{\widetilde e}))}}\rho_{\widetilde e,e}\\\stackrel{\eqref{alphai}}{=}&\sum_{\substack{e\in F_i\\n\in V(e)}}\sigma_e(n)\alpha_i(e).
    \end{aligned}$$
    Now, we finish the inductive construction proving that $\alpha_i$ and $F_i$ also satisfy property \eqref{inprop3}. Since \eqref{mutilde2} is an optimal representation of $\mu_{\widetilde e}$ then by \cite[Proposition 2.4 and Corollary 2.5]{Sch23},
    \begin{equation}\label{lemeq}
        \sum_{e\in E(supp(\mu_{\widetilde e}))}|\rho_{\widetilde e,e}|=\frac{|\mu_{\widetilde e}|}{2}\stackrel{\text{Lemma }\ref{reprmu}}{\leq}|\alpha_{i+1}(\widetilde e)|\Big(1-\frac{1}{C}\Big).
    \end{equation}
    Hence,
    $$\begin{aligned}\sum_{e\in F_i}|\alpha_i(e)|=&\sum_{e\in F_i}\bigg|\sum_{\substack{\widetilde e\in F_{i+1}\\e\in E(supp(\mu_{\widetilde e}))}}\rho_{\widetilde e,e}\bigg|\leq\sum_{e\in F_i}\sum_{\substack{\widetilde e\in F_{i+1}\\e\in E(supp(\mu_{\widetilde e}))}}|\rho_{\widetilde e,e}|\\=&\sum_{\widetilde e\in F_{i+1}}\sum_{e\in E(supp(\mu_{\widetilde e}))}|\rho_{\widetilde e,e}|\stackrel{\eqref{lemeq}}{\leq}\sum_{\widetilde e\in F_{i+1}}\Big(1-\frac{1}{C}\Big)|\alpha_{i+1}(\widetilde e)|.\end{aligned}$$

    This finishes the proof of our Claim. Finally, we show that given the $F_i's$ and $\alpha_i's$ with properties \eqref{inprop1} and \eqref{inprop3} it is possible to prove inequality \eqref{Hypclaim}. Let us first define the sequence $(a_i)_{i=-k}^{i_0}$ where
    $$a_i:=\sum_{n\in L_i}\sum_{\substack{e\in F_i\\n\in V(e)}}|\alpha_i(e)|.$$
    We claim that $a_i\leq 2\big(1-\frac{1}{C}\big)^{i_0-i}$ for every $i=-k,\dots,i_0$. Indeed, it is clear that $a_{i_0}\leq2$. Now, using property \eqref{inprop3} of the inductive construction it is easy to see that $a_i\leq a_{i+1}\big(1-\frac{1}{C}\big)$ and thus the claim follows from a straightforward induction. Now, if $\{x,y\}$ is horizontal then
    $$\begin{aligned}\sum_{n=1}^N|\alpha(n)|=&\sum_{i=-k}^{i_0}\sum_{n\in L_i}|\alpha(n)|\stackrel{\eqref{inprop1}}{\leq}\sum_{i=-k}^{i_0}\sum_{n\in L_i}\sum_{\substack{e\in F_i\\n\in V(e)}}|\alpha_i(e)|\\=&\sum_{i=-k}^{i_0}a_i\leq\sum_{i=-k}^{i_0}2\Big(1-\frac{1}{C}\Big)^{i_0-i}\leq 2C.\end{aligned}$$
    Whereas if $\{x,y\}$ is radial then
    $$\begin{aligned}\sum_{n=1}^N|\alpha(n)|=&\sum_{i=-k}^{i_0+1}\sum_{n\in L_i}|\alpha(n)|\stackrel{\eqref{inprop1}}{\leq}|\alpha(x)|+\sum_{i=-k}^{i_0}\sum_{n\in L_i}\sum_{\substack{e\in F_i\\n\in V(e)}}|\alpha_i(e)|\\=&1+\sum_{i=-k}^{i_0}a_i\leq1+\sum_{i=-k}^{i_0}2\Big(1-\frac{1}{C}\Big)^{i_0-i}\leq 1+2C.\end{aligned}$$
    
\end{proof}

\textbf{Data availability statement.} This manuscript has no associated data.

\textbf{Conflict of interest statement.} There is no conflict of interest.

\textbf{Acknowledgements.} The authors are grateful for the valuable discussions and comments on the subject of the paper from Chris Gartland and Thomas Schlumprecht. The first named author is thankful to Texas A\&M University for their hospitality during the stay where this paper was initiated.

\textbf{Founding information.} The first author work has been supported by PID2021-122126NB-C31 AEI (Spain) project, by FPU19/04085 MIU (Spain) Grant, by Junta de Andalucia Grants FQM-0185 and by GA23-04776S project (Czech Republic). The second author's work has been partially supported by the National Science Foundation under Grant Number DMS2349322.

\bigskip
\bibliography{bibtfg}
\bibliographystyle{alpha}

\end{document}